\documentclass[11pt]{article}

\usepackage[utf8]{inputenc}
\usepackage{amsmath,amssymb,amsfonts,amsthm}
\usepackage{mathtools}
\usepackage[shortlabels]{enumitem}  
\usepackage[margin=1in]{geometry}
\usepackage{xcolor}
\usepackage[authoryear,round,sort&compress]{natbib}
\definecolor{linkcol}{rgb}{0,0,0.55}
\definecolor{citecol}{rgb}{0,0.45,0}
\definecolor{urlcol}{rgb}{0.55,0,0}
\usepackage[colorlinks=true,linkcolor=linkcol,citecolor=citecol,urlcolor=urlcol]{hyperref}
\usepackage{booktabs}
\usepackage{array}
\usepackage{graphicx}

\usepackage{libertinus}
\usepackage[libertine]{newtxmath}
\usepackage[T1]{fontenc}
\usepackage{microtype}

\newif\ifrevisionmarkup
\revisionmarkuptrue
\definecolor{revblue}{RGB}{0,70,180}

  {\ifrevisionmarkup\color{revblue}\fi}{}

\def\E{{\mathbb E}}

\newcommand{\db}{d_{\mathcal{B}}}

\newcommand{\norm}[1]{\left\lVert #1 \right\rVert}

\newcommand{\epsK}{\epsilon_K}

\newtheorem{theorem}{Theorem}[section]
\newtheorem{proposition}{Proposition}[section]
\newtheorem{corollary}{Corollary}[section]
\newtheorem{lemma}{Lemma}[section]
\newtheorem{definition}{Definition}[section]
\newtheorem{assumption}{Assumption}[section]
\newtheorem{remark}{Remark}[section]
\newtheorem{example}{Example}[section]

\numberwithin{equation}{section}

\title{Statistical Inference for Persistence Diagrams via\\
Landmark Embeddings: Minimax Theory and\\
Finite Approximation}

\author{Pramita Bagchi, Sushovan Majhi, Atish Mitra, \v{Z}iga Virk}

\date{}

\begin{document}

\maketitle

\begin{abstract}
Hilbert-space embeddings enable inference for populations of
persistence diagrams, but separation between individual diagrams need
not survive population averaging. We develop a framework for inference
on population mean embeddings, with particular attention to the
additive landmark representations PLACE and PALACE. Treating each
diagram as one independent observation, we apply Hilbert-space limit
theory to obtain covariance estimators, two-sample tests, and
confidence balls under suitable moment conditions, without requiring a
lower-distortion bound. For additive embeddings, we identify the
population mean as an embedding of the mean counting measure and show
that geometric separation of these measures alone cannot guarantee
uniform testing power. We then introduce a model with latent template
diagrams, missing features, and location perturbations. Under common
or feature-specific prevalence conditions, a diagram-level
lower-distortion certificate yields explicit lower bounds on
population mean separation. These margins provide finite-sample
uniform power guarantees, and an additional information-divergence
comparison gives matching sample-complexity bounds over restricted
scale ranges. Confidence sets yield lower bounds on transport
separation of population mean measures and exclusion guarantees for
specified structured alternatives. We also quantify how orthogonal
truncation changes the certified signal and the approximation
allowance needed for confidence sets targeting the full embedding,
relating sample size, retained coordinates, and template separation.
Simulations examine calibration, power, and coverage, and an analysis
of resting-state connectivity from the Autism Brain Imaging Data
Exchange illustrates the procedures.
\end{abstract}

\section{Introduction}
\label{sec:introduction}

Persistence diagrams summarize the topology of a filtered data object.
Each point records the birth and death of a topological feature as the
filtration parameter changes. Their stability under perturbations makes
them useful summaries of geometric structure \citep{Cohen-Steiner2007}.
In statistical applications, each observational unit may contribute a
diagram constructed from a point cloud, image, or network. Our motivating
application uses the Autism Brain Imaging Data Exchange (ABIDE), a
multisite collection of neuroimaging and phenotypic data
\citep{di2014autism}. We compare subject-level persistence diagrams
between participants with autism spectrum disorder and typically
developing controls.

Inference for diagram populations requires accounting for the geometry
of the observations. Under the bottleneck distance, persistence diagrams
have no natural addition or scalar multiplication. A population mean
therefore cannot be defined by ordinary averaging. Fr\'echet means
provide a metric-space alternative, although they may be nonunique and
difficult to compute \citep{Mileyko2011,Turner2014}. Another approach maps
each diagram into a linear space. Persistence landscapes, persistence
images, and kernel representations permit statistical analysis of the
resulting functions or vectors
\citep{Bubenik15,Adams2017,Reininghaus2015,Kusano2016}.

An upper stability bound controls how much a representation changes when
a diagram is perturbed. It does not guarantee that well-separated
diagrams remain separated after transformation. Landmark embeddings
with lower-distortion bounds address this loss of geometric information
\citep{Mitra2021,Mitra2024,PLACE,PALACE}. Their implications for population
inference require a further argument. A lower bound for the distance
between two embedded diagrams need not survive averaging over their
respective populations. We study the conditions under which such a bound
produces a detectable population signal.

Let $\mathcal D_n$ denote the space of persistence diagrams with at most
$n$ off-diagonal points in a bounded birth--death region. For a fixed
configuration $\nu$, consider a Borel map
$\Psi_\nu:\mathcal D_n\to\mathcal H_\nu$ into a separable Hilbert space.
Given independent diagram samples
$X_1,\ldots,X_m\sim P$ and $Y_1,\ldots,Y_{m'}\sim Q$, our inferential
target is
$$
\Delta_\nu^\Psi(P,Q)
:=
\left\|
\mathbb E_P\{\Psi_\nu(X)\}
-
\mathbb E_Q\{\Psi_\nu(Y)\}
\right\|_{\mathcal H_\nu}.
$$
Each sample consists of independent, identically distributed diagrams.
The persistence points within a diagram may be dependent. Thus $m$ and
$m'$ count independent diagrams, and the number of points within a
diagram does not increase the sample size. The procedures test equality
of population mean embeddings. Distinct diagram laws can have equal
mean embeddings, so these tests are not generally omnibus tests of
$P=Q$.

The geometric results use an additional lower-distortion certificate.
For admissible scales $t$ and diagram pairs satisfying the relevant
construction conditions, this certificate has the form
$$
d_B(D,E)\geq t
\quad\Longrightarrow\quad
\|\Psi_\nu(D)-\Psi_\nu(E)\|_{\mathcal H_\nu}
\geq \rho_-(t;\nu)>0.
$$
The admissible scales and pair restrictions depend on the embedding.
The original Mitra--Virk construction uses landmarks in diagram space.
PLACE pools point responses across prescribed scales, while PALACE uses
an adaptive landmark configuration. The latter constructions impose
coherence or non-interference conditions for their lower bounds.
Section~\ref{sec:embedding} states these restrictions and distinguishes
the statistical results that apply to each map.

For the additive PLACE and PALACE maps, the population mean has a
counting-measure interpretation. Write the embedding as
$\Phi(D;\nu)=T_\nu[\widetilde N_D]$, where $\widetilde N_D$ is the
diagram's counting measure padded to mass $n$ by an absence state.
The absence state has zero embedding response. Linearity gives
$\mathbb E_P\{\Phi(X;\nu)\}=T_\nu[\overline\mu_P]$, where
$\overline\mu_P=\mathbb E_P\widetilde N_X$ is the padded mean counting
measure. This measure records expected feature counts across the
birth--death region. If the point response is Lipschitz with constant
$L_\nu$, then
$$
\Delta_\nu(P,Q)
\leq
nL_\nu W_{\infty,0}(\overline\mu_P,\overline\mu_Q).
$$
Here $\Delta_\nu$ denotes the signal for the additive map, and
$W_{\infty,0}$ is the transport distance defined using a ground metric
that accounts for matching to the diagonal. A positive mean-embedding
signal therefore gives a lower bound on transport separation of the
mean measures. The original Mitra--Virk map is nonadditive, so its mean
embedding requires a separate population interpretation.

The reverse implication fails without restrictions on the populations.
A fixed $W_{\infty,0}$ discrepancy can be carried by an arbitrarily
small amount of expected mass. In that case, the mean-embedding signal
can tend to zero while the transport separation stays fixed.
Cancellation among point responses can also reduce the signal.
The rare-mass example in Section~\ref{sec:embedding} shows that
geometric separation alone does not yield a finite uniform sample
complexity.

We obtain sufficient conditions through a
template--thinning--perturbation model. Each population has a latent
template diagram. A template feature appears with a specified
probability, and its observed location may be perturbed. Feature
appearances and perturbations within a diagram may have arbitrary
dependence. The model allows variable cardinality and does not require
observed feature labels or matching across subjects.

Two prevalence regimes give explicit population bounds. Under common
scalar prevalence $p$, the weighted template response is $p$ times the
complete template response. Under feature-specific prevalence, a
restricted response condition controls cancellation caused by
reweighting. Writing
$s=d_B(D_P^\star,D_Q^\star)$ for template separation, the bounds take
the forms
$$
\Delta_\nu(P,Q)
\geq
p\left[\rho_-(s;\nu)-\mathcal E_{\mathrm{com}}(P,Q)\right]_+
$$
and
$$
\Delta_\nu(P,Q)
\geq
\left[
\eta_{P,Q}\rho_-(s;\nu)
-
\mathcal E_{\mathrm{het}}(P,Q)
\right]_+.
$$
The perturbation penalties depend on feature prevalence and expected
movement from the templates. The factor $\eta_{P,Q}$ measures the
separation retained under heterogeneous prevalence. Both results
require a certified template pair. They show when geometric separation
remains detectable after missing features and location variability are
accounted for. The geometric scale here concerns the latent templates;
it is distinct from transport separation of the population mean
measures.

The probability theory applies more broadly than this population model.
For a fixed Hilbert-space map, we establish the strong law and central
limit theorem under the corresponding moment conditions. We also give
trace-norm consistency of the empirical covariance operator. A
Berry--Esseen bound of order $m^{-1/2}$ holds for each fixed
finite-dimensional projection. Its constant depends on projected rank
and standardized third moments. These results support Gaussian
plug-in and multiplier calibration of mean-based tests and confidence
balls. Their derivation uses standard Hilbert-space probability
\citep{Bentkus2003}. Additivity and lower distortion
are needed for the population interpretations and geometric power
bounds.

For two-sample testing, we use the norm of the empirical mean
difference. A finite-sample bound relates its power to
$\Delta_\nu^\Psi(P,Q)$ and a covariance trace envelope $v^2$.
Uniform restrictions on prevalence, response conditioning, and
perturbation yield structured alternative classes
$\mathfrak A_r(\tau)$, indexed by a minimum template separation $\tau$.
For each regime $r$, these restrictions give an explicit margin
$g_r(\tau)$ such that
$\Delta_\nu(P,Q)\geq g_r(\tau)$ throughout the class.
For fixed Type I and Type II error targets, a sufficient effective
sample size is
$$
N_{\mathrm{eff}}
:=
\frac{mm'}{m+m'}
\gtrsim
\frac{v^2}{g_r(\tau)^2}.
$$

A matching minimax lower bound requires an information comparison
within the same diagram-law classes. We formulate this requirement
through the smallest Kullback--Leibler divergence between a structured
alternative and the mean-embedding null. When this divergence is at
most a constant multiple of $g_r(\tau)^2$, the minimax sample
complexity is of order $g_r(\tau)^{-2}$. Verifying the comparison
requires a construction-specific family of diagram distributions.
The lower-distortion certificate alone does not imply it.

Confidence balls quantify uncertainty about a population mean
embedding or a two-population mean contrast. For additive maps, the
upper Lipschitz bound converts these sets into lower confidence bounds
on transport separation of padded mean measures. An upper confidence
bound on the mean signal can also exclude a proposed structured
alternative class. These statements retain their distinct roles:
sampling uncertainty determines confidence coverage, while embedding
bounds provide the geometric interpretation. We also obtain
nearest-centroid decision certificates that account for centroid
estimation and bottleneck perturbations of a query diagram.
Regression extensions describe inference and prediction stability
when a finite representation is fixed.

Finite-coordinate implementations require a choice of inferential
target. A fixed projection defines its own mean embedding, for which
ordinary finite-dimensional inference applies. Statements about the
full embedding must also account for the omitted coordinates. Let
$\Phi_K=\Pi_K\Phi$, where $\Pi_K$ is an orthogonal projection, and set
$\epsilon_K=\sup_D\|\Phi(D;\nu)-\Phi_K(D;\nu)\|$.
Orthogonality gives the retained structured margin
$$
g_{r,K}(\tau)
=
\left[g_r(\tau)^2-4\epsilon_K^2\right]_+^{1/2}.
$$
Consequently, the sufficient sample-size bound becomes
$$
N_{\mathrm{eff}}
\gtrsim
\frac{v^2}{g_r(\tau)^2-4\epsilon_K^2},
\qquad
2\epsilon_K<g_r(\tau).
$$
This bound relates the number of observed diagrams to the
approximation error and the template separation to be resolved.
A zero lower bound is inconclusive, since the uniform approximation
bound may be conservative.

The embedding configuration is fixed independently of the inferential
sample throughout. A configuration learned from independent reference
data can be treated conditionally as fixed. Learning landmarks or
selecting coordinates from the inferential observations requires
accounting for that selection. The numerical strength of the
geometric guarantees also depends on the construction constants,
which can deteriorate with maximal diagram cardinality.
We examine the finite-sample behavior of the procedures in simulations.
The ABIDE application compares diagnostic groups using the participant
as the sampling unit.

\paragraph{Related work.}
Statistical inference using persistence landscapes and other
functional summaries predates the present framework
\citep{Bubenik15,Wasserman2018,Ali2023-ht}. Kernel methods also permit
mean-embedding comparisons and maximum mean discrepancy testing
\citep{Reininghaus2015,Kusano2016}. Characteristicness ensures that a
kernel mean identifies the underlying distribution. It does not by
itself specify a quantitative relationship between statistical signal
and bottleneck separation. Our analysis studies such a relationship
under explicit population restrictions. Inference based directly on
diagram distances and Fr\'echet summaries provides another approach
\citep{Chazal2014,Fasy2017,Mileyko2011,Turner2014}. The additive
representations studied here permit a linear analysis of population
mean counting measures.

STRAND \citep{STRAND} addresses a related two-sample problem using
survival-analysis methods for persistence values. It pools feature
lifetimes and uses a log-rank statistic. Our procedures retain the
joint birth--death response represented by the embedding and estimate
sampling variability across diagrams. They can therefore respond to
differences in feature location that are not captured by lifetimes
alone. PLACE and PALACE \citep{PLACE,PALACE} study landmark
representations and their classification properties. The present
paper develops population inference for fixed maps and identifies
conditions under which their geometric bounds imply detectable mean
differences.

\paragraph{Organization of the paper.}
Section~\ref{sec:diagrams} introduces persistence diagrams and their
counting-measure representation.
Section~\ref{sec:embedding} describes the embedding assumptions and
proves the structured population bounds.
Section~\ref{sec:probability} develops the probability theory, and
Section~\ref{sec:testing} studies two-sample testing and minimax bounds.
Section~\ref{sec:confidence-certificates} gives confidence sets and
geometric certificates.
Section~\ref{sec:finite-approximation} studies finite approximation,
and Section~\ref{sec:extensions} records regression and classification
extensions. Sections \ref{sec:numerical-illustrations} and \ref{sec:abide} present the simulations and ABIDE
analysis, respectively. Section~\ref{sec:discussion} concludes.

\section{Persistence Diagrams as Statistical Objects}
\label{sec:diagrams}

This section is intended for readers familiar with probability on metric
spaces but not necessarily with persistent homology.  We introduce the diagram space and its bottleneck metric, then define random diagrams and their mean counting measures. The construction background is limited to what is needed for the statistical analysis.

\subsection{From filtered data to persistence diagrams}
\label{sec:pd-construction}

Let $\{K_r\}_{r\geq 0}$ be a filtration, meaning that
$K_r\subseteq K_s$ whenever $r\leq s$. A filtration may be constructed from a
point cloud, weighted graph, image, or another structured observation. For a
point cloud, for example, one may form a Vietoris--Rips or \v{C}ech complex as
a distance threshold increases. For a weighted graph, edges may be added
according to their weights and a clique complex constructed at each threshold.
Persistent homology records how the topological features of $K_r$ evolve with
the filtration parameter $r$; see \cite{Cohen-Steiner2007} and
\cite{Chazal2017DTM} for background.

Fix a homological degree $q$ and a coefficient field $\mathbb F$.
To describe the construction, consider a filtration of finite
simplicial complexes with finitely many changes. The homology groups
and inclusion-induced maps form a persistence module, which decomposes
into interval components. Each interval with finite endpoints $b<d$
contributes a point $(b,d)$ to the persistence diagram. These endpoints
record the birth and death of the corresponding feature. In degree
zero, a finite death time records the merging of connected components.
In degree one, a loop may die when it becomes the boundary of a filled
region.

The persistence of a point $(b,d)$ is $d-b$. Its $\ell_\infty$
distance to the diagonal is $(d-b)/2$, which is also the cost of
deleting that point in the bottleneck metric defined below.
Near-diagonal features can therefore be removed at small bottleneck
cost. Persistence alone does not determine whether a feature is noise
or scientifically meaningful.

Figure~\ref{fig:pd-construction} illustrates the construction for a
simulated point cloud. The central loop produces a degree-one
persistence point. The statistical analysis begins with the resulting
diagram as the observed data object.

\begin{figure}[hbt]
  \centering
  \includegraphics[width=.95\textwidth]{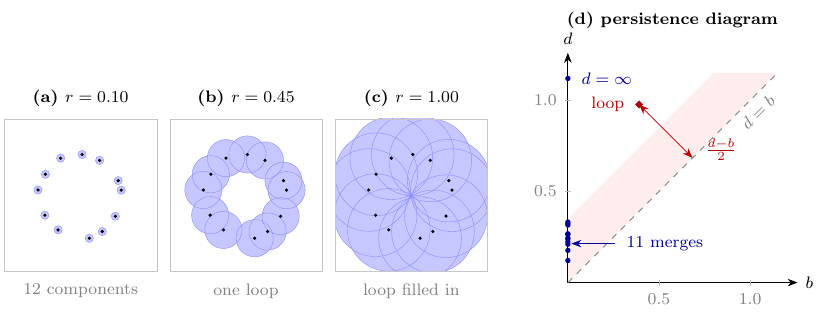}
 \caption{\textbf{From a point cloud to a persistence diagram.}
(a)--(c) Unions of balls around $12$ points sampled with noise from
a circle. At radius $r=0.10$, the balls are disjoint. By $r=0.45$,
they form one connected component enclosing a loop. By $r=1.00$,
the loop has filled in.
(d) The diagram displays degree-zero and degree-one features.
The eleven finite degree-zero classes are born at $r=0$ and die
by $r=0.33$. The loop is born at $r=0.39$ and dies at $r=0.98$.
One connected component persists indefinitely and is shown at
$d=\infty$. This essential class is excluded from the finite
diagrams used for inference. The shaded band contains points below
a fixed persistence threshold. A point $(b,d)$ can be deleted at
bottleneck cost $(d-b)/2$; membership in the band does not by itself
establish that the feature is noise.}
  \label{fig:pd-construction}
\end{figure}

Throughout the paper, persistence diagrams are restricted to a bounded
birth-death region. Fix $0 < L<\infty$ and define

$$
\mathbb{H}_L:=\bigl\{(b,d)\in[0,L]^2:b<d\bigr\},\qquad\mathsf{Diag}:=\bigl\{(t,t):0\leq t\leq L\bigr\}.
$$

The set $\mathbb{H}_L$ is the bounded off-diagonal persistence region, and
$\mathsf{Diag}$ is its geometric diagonal. We restrict attention to finite
off-diagonal features; essential classes with infinite death times are
excluded before analysis.

\begin{definition}[Persistence diagram]
\label{def:persistence-diagram}
A persistence diagram $D$ is a finite multiset of points in $\mathbb{H}_L$. Multiplicity
is retained: if distinct topological features have the same birth and death
values, the corresponding point appears more than once. For a fixed integer
$n\geq 1$, let

$$
\mathcal{D}_n := \bigl\{ D: D\text{ is a persistence diagram and }|D|\leq n
\bigr\}.
$$
\end{definition}

The bounded-cardinality restriction is substantive. The full space of
persistence diagrams with unbounded cardinality does not coarsely embed into a
Hilbert space \citep{BubenikWagner2020}. The cardinality bound enters the lower-distortion results and the
uniform bounds for additive embeddings used later. The associated
constants may deteriorate as $n$ increases. In applications, a common
bound may follow from the filtration or be imposed through
preprocessing. If only the most persistent features are retained,
the selection rule and its treatment of ties must be fixed across
observations. The inferential statements then concern the processed
diagrams. Hard selection by persistence can be discontinuous near
ties, so stability with respect to the original diagrams requires
separate control of that selection step.

\subsection{Bottleneck geometry}
\label{sec:bottleneck}

For $x=(b,d)\in \mathbb{H}_L$, define its distance to the diagonal by

$$
\delta_{\mathsf{Diag}}(x) := \inf_{z\in\mathsf{Diag}}\|x-z\|_\infty
=\frac{d-b}{2}.
$$

The bottleneck distance compares two diagrams by matching their points.
Features that are not matched across diagrams are instead matched to the
diagonal, with cost equal to one-half of their persistence.

Let $D,E\in\mathcal{D}_n$. A partial matching $M$ between $D$ and $E$ is a
collection of pairs $(x,y)\in D\times E$ such that each point, including its
multiplicity, appears in at most one pair. Write $\operatorname{dom}(M)$ and
$\operatorname{ran}(M)$ for the matched points of $D$ and $E$, respectively.

\begin{definition}[Bottleneck distance]
\label{def:bottleneck-distance}
The bottleneck distance between $D$ and $E$ is
\begin{align}
d_B(D,E)
:= \inf_M
\max\Bigg\{ \max_{(x,y)\in M}\|x-y\|_\infty, \max_{x\in D\setminus\operatorname{dom}(M)}\delta_{\mathsf{Diag}}(x),
\max_{y\in E\setminus\operatorname{ran}(M)}
\delta_{\mathsf{Diag}}(y)
\Bigg\},
\label{eq:bottleneck-partial-matching}
\end{align}
where the infimum is taken over all partial matchings $M$ between $D$ and $E$.
The maximum over an empty set is defined to be zero.
\end{definition}



Three aspects of Definition~\ref{def:bottleneck-distance} are important.
First, $d_B$ is a maximum rather than an aggregate: a single feature that
cannot be matched at low cost determines the distance even when the remaining
features agree closely, while many short-lived differences have little effect
if they can be matched near the diagonal. Second, displacements are measured
in $\ell_\infty$, so moving both coordinates of a point by at most $\epsilon$
incurs cost at most $\epsilon$. Third, matching to the diagonal allows
diagrams of different cardinalities to be compared, at cost equal to one-half
of a feature's persistence; equivalently, one may adjoin the diagonal with
infinite multiplicity and minimize, over all bijections, the largest
displacement.

This maximum-matching geometry is appropriate for the stability theory of
persistent homology \citep{Cohen-Steiner2007}, but it differs substantially
from Euclidean or aggregate discrepancies between fixed-dimensional vectors.
In particular, $d_B$ measures the largest geometric discrepancy but does not
quantify how many persistence features differ. This distinction motivates the
separate control of geometric separation and differing population mass in the
population lower-bound analysis.

We equip $\mathcal{D}_n$ with the Borel $\sigma$-field generated by the
bottleneck metric $d_B$. This provides the measurable structure needed to
treat persistence diagrams as random elements.

\begin{lemma}[Compactness]
\label{lem:Dn-polish}
The metric space $(\mathcal D_n,d_B)$ is compact. In particular,
it is complete and separable.
\end{lemma}
\noindent
The proof is given in Appendix~\ref{app:Dn-polish}. The ``at most $n$''
convention is important because a bottleneck-Cauchy sequence may lose
off-diagonal points as they approach the diagonal. Separability follows by
approximating the remaining off-diagonal points by points with rational
coordinates.

\subsection{Random diagrams and counting measures}
\label{sec:random-diagrams}
\label{sec:padded-mean-measures}

\begin{definition}[Random persistence diagram]
\label{def:random-diagram}
Let $(\Omega,\mathcal{F},\mathbb{P})$ be a probability space. A random
persistence diagram is a measurable map

$$ X:\Omega\rightarrow\mathcal{D}_n.$$
Its distribution is the probability measure
$P=\mathbb{P}\circ X^{-1}$ on $\mathcal{D}_n$, and we write $X\sim P$.
\end{definition}


Here the diagram is the observational unit, and its persistence points are
components of a single structured observation. The law $P$ represents between-unit variation in the observed diagrams, including variation arising from the underlying structured objects,
sampling, measurement, and preprocessing. In the ABIDE application, for
example, $P$ describes the population distribution of subject-level
persistence diagrams within a diagnostic group.

A diagram $D$ may also be represented by its counting measure
$$N_D:=\sum_{x\in D}\delta_x.$$
The counting measure is a faithful representation of the off-diagonal
persistence diagram: it retains every point and its multiplicity and therefore
uniquely determines $D$. This measure is not normalized by $|D|$, because diagram cardinality may itself contain information. Normalization would remove this information; for example, duplicating every point of a diagram would leave its normalized counting measure unchanged.

For every Borel set $B\subseteq\mathbb H_L$, the map
$D\mapsto N_D(B)$ is Borel measurable and takes values in
$\{0,\ldots,n\}$. For open $B$, measurability follows from lower
semicontinuity under bottleneck convergence. A monotone-class argument
extends the conclusion to all Borel sets. For a probability distribution $P$ on
$\mathcal{D}_n$, define its mean counting measure by

$$ \Lambda_P(B) := \mathbb{E}_P\bigl\{N_X(B)\bigr\} = \mathbb{E}_P
\left\{ \sum_{x\in X}\mathbf{1}_{\{x\in B\}} \right\}, \qquad
B\in\mathcal{B}(\mathbb{H}_L). $$
The measure $\Lambda_P$ records the expected persistence-point intensity in
each measurable region of the birth-death plane. Its total mass is
$$ \Lambda_P(\mathbb{H}_L) = \mathbb{E}_P{|X|},$$
which equals the expected diagram cardinality. Unlike the counting measure of
an individual diagram, $\Lambda_P$ need not be integer-valued and need not
itself correspond to a persistence diagram. The mean counting measure does not determine the diagram law.
For example, suppose $n\geq2$ and choose distinct
$a,b\in\mathbb H_L$. Under $P$, let the diagram be $\{a\}$ or
$\{b\}$ with equal probability. Under $Q$, let it be the empty
diagram or $\{a,b\}$ with equal probability. Then
$$
\Lambda_P=\Lambda_Q=\tfrac12\delta_a+\tfrac12\delta_b,
$$
although the cardinality distributions differ. This distinction
limits what can be identified from an additive mean embedding.




Because diagrams may have different cardinalities, their counting
measures need not have the same total mass. Introduce an absence state
$\varnothing$ and write
$$
\mathbb H_L^\varnothing:=\mathbb H_L\cup\{\varnothing\},
\qquad
\widetilde N_D:=N_D+(n-|D|)\delta_\varnothing.
$$
Every padded counting measure has total mass $n$. The additive maps
introduced in Section~\ref{sec:embedding} assign zero response to
$\varnothing$, so padding does not change their values.

For $x,y\in\mathbb H_L$, define
$$
d_0(x,y):=\min\left\{\|x-y\|_\infty,\,
\max\{\delta_{\mathsf{Diag}}(x),
      \delta_{\mathsf{Diag}}(y)\}
\right\}.
$$
Also set
$$
d_0(x,\varnothing)=d_0(\varnothing,x)=\delta_{\mathsf{Diag}}(x),
\qquad
d_0(\varnothing,\varnothing)=0.
$$
This is the bottleneck metric on the space of one-point diagrams
and the empty diagram. The first term in the minimum represents
a direct match. The second represents deleting both points.
Thus an off-diagonal transport pair can encode deletion even when
neither padded measure has mass at $\varnothing$.

For finite nonnegative measures $\alpha$ and $\gamma$ with the same
positive total mass, let $\Gamma(\alpha,\gamma)$ denote their
couplings. Define
$$
W_{\infty,0}(\alpha,\gamma):=\inf_{\pi\in\Gamma(\alpha,\gamma)}
\|d_0\|_{L^\infty(\pi)}.
$$
The quantity $\|d_0\|_{L^\infty(\pi)}$ is the essential supremum
of the transport cost under $\pi$. Scaling both measures by the
same positive constant leaves this distance unchanged.

\begin{remark}[Bottleneck distance as transport distance]
\label{rem:bottleneck-transport}
For every $D,E\in\mathcal D_n$,
$$
d_B(D,E)=W_{\infty,0}(\widetilde N_D,\widetilde N_E).
$$
A coupling can represent direct matches, deletion of paired
off-diagonal points, or deletion through the absence state.
Appendix~\ref{app:padded-transport} proves the identity.
\end{remark}

The embedding maps introduced in Section~\ref{sec:embedding} are
Borel measurable. Consequently, $\Psi_\nu(X)$ is a well-defined
Hilbert-space-valued random element whenever $X$ is a random diagram.

\subsection{Statistical implications of diagram geometry}
\label{sec:direct-inference-obstructions}

The space $\mathcal D_n$ is not a vector space. Averaging point
coordinates requires a choice of correspondence across diagrams.
Fr\'echet means provide diagram-valued summaries, although they
need not be unique \citep{Mileyko2011,Turner2014}. The mean counting
measure $\Lambda_P$ is always uniquely defined and describes
population-average feature intensity.

Persistence points within a diagram may be dependent. Treating all
points as independent observations can misrepresent uncertainty and
change the weighting of subjects when cardinalities vary. We therefore
formulate sampling and inference at the diagram level. The
counting-measure representation retains multiplicities and accommodates
variable cardinality without requiring observed feature labels.

For a fixed measurable Hilbert-space representation, mean-based
inference is available under suitable moment conditions. Geometric
power guarantees require additional information. An upper stability
bound controls changes in the representation under small diagram
perturbations. It does not ensure that separated diagrams have
well-separated representations, and a diagram-level lower bound
need not survive population averaging.

Section~\ref{sec:embedding} states the assumptions for general
fixed-map inference and the additional structure used by additive
embeddings. It then introduces the template--thinning--perturbation
model and gives sufficient conditions for geometric separation
to produce a nonzero population mean signal.

\section{Landmark Embeddings as an Interface}
\label{sec:embedding}

Section~\ref{sec:diagrams} identified four features of diagram space that
block standard inference: no linear structure, matching-dependent geometry,
variable cardinality, and no stable feature labels. This section converts
each obstacle into a requirement on a map and states the theory through
those requirements rather than through a construction. Two tiers of
structure are kept separate: a minimal fixed-map level
(Assumptions~\ref{ass:fixed-hilbert-embedding}--\ref{ass:uniformly-bounded-embedding}),
which suffices for all limit theory and two-sample
procedures and covers every construction we consider, and an additive
interface (Assumptions~\ref{ass:additive-interface}--\ref{ass:finite-approximation}),
which identifies population mean
embeddings with embeddings of mean measures and supports the population
analysis. The distinction matters most where the theory is hardest: the
diagram-level distortion certificate does not by itself transfer to
populations (Example~\ref{ex:pop-obstruction}), and the structure that makes the transfer hold
is the subject of Sections~\ref{sec:probability}--\ref{sec:testing}.

\subsection{Fixed Hilbert embeddings and the additive interface}
\label{sec:additive-interface}

We begin with the minimal structure needed for inference after embedding.

\begin{assumption}[Fixed Hilbert embedding]
\label{ass:fixed-hilbert-embedding}
For a configuration $\nu$, let
$$
\Psi_\nu:\mathcal D_n\longrightarrow\mathcal H_\nu
$$
be a Borel map into a separable real Hilbert space
$\mathcal H_\nu$. The configuration $\nu$ is fixed independently of the
observations used for inference.
\end{assumption}

The configuration may be selected using an independent pilot sample. In
that case, all probability statements are understood conditionally on the
pilot sample and its realized configuration.

For probability laws $P$ and $Q$ on $\mathcal D_n$, let $X\sim P$ and $Y\sim Q$. Whenever the corresponding Bochner expectations exist, define the population mean-embedding separation by 
\begin{equation} \label{eq:general-population-signal} 
\Delta_\nu^\Psi(P,Q) := \left\| \mathbb E_P\{\Psi_\nu(X)\} - \mathbb E_Q\{\Psi_\nu(Y)\} \right\|_{\mathcal H_\nu}. 
\end{equation} 
The required integrability conditions are stated with the results that use them. Existence of the population mean requires first moments, whereas the Hilbert-space central limit theorem, covariance estimation, and quantitative Gaussian approximation results require the stronger moment conditions stated in the corresponding sections. Some of the finite-sample and quantitative approximation results require the
following stronger boundedness condition.

\begin{assumption}[Uniform boundedness of the embedding]
\label{ass:uniformly-bounded-embedding}
For the fixed configuration $\nu$,
$$B_\nu := \sup_{D\in\mathcal D_n} \|\Psi_\nu(D)\|_{\mathcal H_\nu}<\infty.$$
\end{assumption}

\noindent The original Mitra -- Virk map and the fixed PLACE and PALACE maps satisfy this formulation. The remainder of this subsection introduces the additive structure available for the pooled PLACE and PALACE point-response maps.

Recall the padded state space
$\mathbb H_L^{\varnothing}=\mathbb H_L\cup\{\varnothing\}$ and padded counting
measure
$$ \widetilde N_D=N_D+(n-|D|)\delta_{\varnothing},$$
which has total mass $n$ for every $D\in\mathcal D_n$. The absence state has zero response. Section~\ref{sec:padded-mean-measures} defines the ground metric
$d_0$ and the associated transport distance $W_{\infty,0}$.
Appendix~\ref{app:padded-transport} proves the transport
representation. In particular,
\begin{equation}
\label{eq:bottleneck-padded-transport}
d_B(D,E) = W_{\infty,0}(\widetilde N_D,\widetilde N_E).
\end{equation}

\begin{assumption}[Additive embedding interface]
\label{ass:additive-interface}
A Borel map $\Phi(\,\cdot\,;\nu):\mathcal D_n\to\ell_2$ satisfies the additive
interface if:
\begin{enumerate}
\item[(E1)] \textbf{Additive mean-measure representation.}
There is a Borel map
$\varphi_\nu^0:\mathbb H_L^{\varnothing}\to\ell_2$ with
$\varphi_\nu^0(\varnothing)=0$ such that
\begin{equation}
\label{eq:additive-embedding-representation}
\Phi(D;\nu)=T_\nu[\widetilde N_D], \qquad
T_\nu[\mu] :=
\int_{\mathbb H_L^{\varnothing}}\varphi_\nu^0(x)\,d\mu(x).
\end{equation}

\item[(E2)] \textbf{Point-level Lipschitz continuity.}
For a constant $L_\nu \in (0,\infty)$,
\begin{equation}
\label{eq:point-map-bounds}
\|\varphi_\nu^0(x)-\varphi_\nu^0(y)\|_2 \leq L_\nu d_0(x,y),
\qquad x,y\in\mathbb H_L^{\varnothing}.
\end{equation}
Since the padded state space is bounded, this also gives
$\kappa_\nu:=\sup_x\|\varphi_\nu^0(x)\|_2<\infty$.

\end{enumerate}
\end{assumption}

\begin{assumption}[Finite approximation]
\label{ass:finite-approximation}
For the fixed embedding
$\Psi_\nu:\mathcal D_n\to\mathcal H_\nu$, there exist finite-rank orthogonal
projections $\Pi_K$ such that
\begin{equation}
\label{eq:uniform-truncation-error}
\epsilon_K^\Psi
:=
\sup_{D\in\mathcal D_n}
\|\Psi_\nu(D)-\Pi_K\Psi_\nu(D)\|_{\mathcal H_\nu}
\longrightarrow0.
\end{equation}
For an additive landmark embedding whose coordinates are partitioned into
ordered finite blocks $J_1,J_2,\ldots$, the projection $\Pi_K$ may be chosen
to retain the first $K$ blocks. In that case, write
$\Phi_K=\Pi_K\Phi$ and $\epsilon_K:=\epsilon_K^\Phi$. For a
finite-dimensional embedding, the approximation
error is zero once all coordinate blocks have been retained.
\end{assumption}

\begin{assumption}[Additional diagram-level lower-distortion certificate]
\label{ass:diagram-distortion-floor}
For the fixed embedding $\Psi_\nu$, there exist a resolution threshold $R_0(\nu)\geq 0$ and a positive and non-decreasing function $\rho_-(\,\cdot\,;\nu): \bigl(R_0(\nu),\infty\bigr)\longrightarrow(0,\infty)$ such that, for every pair $(D,E)$ satisfying the construction-specific certification conditions,  
\begin{equation} \label{eq:diagram-distortion-floor} 
d_B(D,E)\geq r  >R_0(\nu), \qquad \Longrightarrow \qquad \|\Psi_\nu(D)-\Psi_\nu(E)\|_{\mathcal H_\nu} \geq \rho_-(r;\nu)>0. \end{equation} 
\end{assumption}

The resolution threshold depends on the construction, the
embedding configuration, and the target scale. If a construction theorem gives
only an actual-separation floor $\rho_{\mathrm{raw},\nu}^{\Psi}(u)$ and that
function is not known to be nondecreasing the valid
threshold floor is the positive minorant
\[
\rho_{-,\nu}^{\Psi}(t)
:=
\inf_{\substack{u\in\mathcal I_\nu^\Psi\\u\geq t}}
\rho_{\mathrm{raw},\nu}^{\Psi}(u),
\]
provided the infimum is positive on the stated domain. If the cited theorem
already gives a threshold implication or a nondecreasing floor, that function
is used directly. Failure of the resolution threshold means only that Assumption \ref{ass:diagram-distortion-floor} is inconclusive. 

Because the embedding is fixed throughout the main inferential development,
we abbreviate $\rho_{-,\nu}^{\Psi}(t)$ by $\rho_-(t;\nu)$ in the main text.

\begin{proposition}[Consequences of the additive interface]
\label{prop:interface-consequences}
Under (E1)--(E2),
\begin{equation}
\label{eq:embedding-bounds}
\|\Phi(D;\nu)\|_2\leq n\kappa_\nu,
\qquad
\|\Phi(D;\nu)-\Phi(E;\nu)\|_2
\leq nL_\nu d_B(D,E)
\end{equation}
for all $D,E\in\mathcal D_n$.
\end{proposition}

Thus a nonzero additive embedding difference certifies a nonzero bottleneck
difference without Assumption \ref{ass:diagram-distortion-floor}. The lower-distortion certificate supplies the converse direction only on its construction-specific certified scale set. The proof is deferred to
Appendix~\ref{app:additive-interface-proofs}.

\subsection{Landmark constructions}
\label{sec:embedding-instances}

PLACE and the raw PALACE summation map have coordinates obtained by summing
localized point responses and therefore satisfy the additive form when their
embedding-configuration conditions hold. The original Mitra -- Virk map is instead a
whole-diagram landmark map. Table~\ref{tab:landmark-instances} separates the
general fixed-Hilbert inference scope from the additive population
interpretation.

\begin{table}[!ht]
\centering
\footnotesize
\begin{tabular}{>{\raggedright\arraybackslash}p{2.0cm}>{\raggedright\arraybackslash}p{4.05cm}>{\raggedright\arraybackslash}p{7.05cm}}
\toprule
Construction & Statistical scope & Additive interpretation and lower-distortion certificate \\
\midrule
Mitra--Virk \citep{Mitra2024}
& General fixed-Hilbert inference applies to an admissible fixed map under the
result-specific moment conditions.
& The map is whole-diagram and not additive. Theorem~4.3 gives normalized
infinite-dimensional coarse and uniform maps; Theorem~5.1 gives a
finite-dimensional bounded-domain map. No PLACE/PALACE-style pair coherence is
imposed, subject to the theorem-specific domain, scale, and weight conditions. \\
PLACE \citep{PLACE}
& General inference and additive population theory apply to the fixed pooled map.
& Proposition~2.1(b) gives the step certificate for $d_B\geq3R_1$ under
$\nu$-coherence. Corollary~2.1 gives the separate affine certificate on
$R_1<d_B\leq L$ under coherence at every active scale. \\
PALACE \citep{PALACE}
& General inference and additive population theory apply to the fixed raw
summation map.
& Theorem~2.1 applies at a target $\tau$ when the embedding configuration is
$\tau$-admissible, both diagrams lie in the covered support, and the pair
satisfies Definition~2.3 non-interference. \\
\bottomrule
\end{tabular}
\caption{Scope of the three certified landmark constructions. Exact constants,
scale sets, pair classes, and normalization conventions are given in
Appendix~\ref{app:construction-verification}.}
\label{tab:landmark-instances}
\end{table}

\begin{remark}[The original Mitra -- Virk construction]
\label{rem:original-mitra-virk}
The construction of \cite{Mitra2024} is the foundational certified
whole-diagram embedding. Theorem~4.3 supplies normalized coarse and uniform
maps on $\mathcal D_n$, and Theorem~5.1 supplies a finite-dimensional
bounded-domain map. These results do not impose the additional pair-specific
coherence conditions used by PLACE and PALACE, but they remain subject to the
bounded-cardinality, domain, scale, sequence, and weight conditions stated in
the cited theorems. The product cover has multiplicity $4^n$, and
Example~5.2 has $N4^n$ coordinates, illustrating the exponential cost in the
maximal diagram cardinality. PLACE and PALACE replace this product
representation by pooled or adaptively selected point responses and impose
their own lower-distortion certification conditions. General fixed-Hilbert inference applies
to a fixed admissible Mitra--Virk map, whereas the additive padded
mean-measure identity does not. Inference based on its population mean detects
alternatives with distinct mean embeddings; it does not automatically
distinguish every pair of unequal diagram laws. Appendix~\ref{app:mv-population-transfer}
gives a separate sufficient whole-law population bound.
\end{remark}

Additive representations are not unique to landmark methods; the distinctive
ingredient here is the explicit bottleneck lower-distortion certificate available for the
certified landmark constructions. Detailed verification is deferred to
Appendix~\ref{app:construction-verification}.

\subsection{Population mean embeddings and the obstruction}
\label{sec:population-mean-transfer}

Let $X\sim P$ and $Y\sim Q$ be random diagrams. Their padded mean counting
measures are defined setwise by
\begin{equation}
\label{eq:padded-mean-measures}
\overline\mu_P(B)
:=
\mathbb E_P\{\widetilde N_X(B)\},
\qquad
\overline\mu_Q(B)
:=
\mathbb E_Q\{\widetilde N_Y(B)\},
\end{equation}
for Borel $B\subseteq\mathbb H_L^{\varnothing}$. Both have total mass $n$.
Unlike an average diagram, these finite measures are always well defined and
require no feature correspondence across observations.

\begin{proposition}[Mean embedding and population transport]
\label{prop:mean-embedding-transfer}
Under (E1)--(E2),
\begin{equation}
\label{eq:mean-measure-identity}
\mathbb E_P\{\Phi(X;\nu)\}=T_\nu[\overline\mu_P],
\qquad
\mathbb E_Q\{\Phi(Y;\nu)\}=T_\nu[\overline\mu_Q].
\end{equation}
Define the population mean-embedding separation by
\begin{equation}
\label{eq:population-mean-embedding-signal}
\Delta_\nu(P,Q)
:=
\left\|
\mathbb E_P\{\Phi(X;\nu)\}
-
\mathbb E_Q\{\Phi(Y;\nu)\}
\right\|_2.
\end{equation}
Then
\begin{equation}
\label{eq:population-upper-transfer}
\Delta_\nu(P,Q)
\leq
nL_\nu W_{\infty,0}(\overline\mu_P,\overline\mu_Q).
\end{equation}
\end{proposition}

\noindent Consequently,
\[
\Delta_\nu(P,Q)>0
\quad\Longrightarrow\quad
W_{\infty,0}(\overline\mu_P,\overline\mu_Q)
\geq
\frac{\Delta_\nu(P,Q)}{nL_\nu}>0.
\]
This interpretation uses neither Assumption \ref{ass:diagram-distortion-floor} nor any restriction on the data-generating laws $P$ and $Q$. It is a first-order statement: the additive estimand records transformed mean
intensity, not every aspect of the full diagram laws.

\begin{example}[Rare separated mass]
\label{ex:pop-obstruction}
Choose $z\in\mathbb H_L$ with
$\tau:=\delta_{\mathsf{Diag}}(z)>0$. Let $D_z$ be the one-point diagram at
$z$, let $D_\varnothing$ be the empty diagram, and set
\[
Q=\delta_{D_\varnothing},
\qquad
P_\varepsilon
=(1-\varepsilon)\delta_{D_\varnothing}
+\varepsilon\delta_{D_z},
\qquad 0<\varepsilon<1.
\]
Then
$W_{\infty,0}(\overline\mu_{P_\varepsilon},\overline\mu_Q)=\tau$ for every
$\varepsilon>0$, whereas
\[
\Delta_\nu(P_\varepsilon,Q)
=
\varepsilon
\|\Phi(D_z;\nu)-\Phi(D_\varnothing;\nu)\|_2
\longrightarrow0.
\]
\end{example}

Thus a diagram-level lower-distortion certificate cannot by itself yield a uniform
population reverse implication: separated mass may become arbitrarily rare,
and weighted responses may also cancel. The verification is in
Appendix~\ref{app:population-transfer-proofs}.

\subsection{Structured population detectability}
\label{sec:template-population-transfer}

We next identify structured population alternatives for which template
separation or weighted template-mass differences produce a nonzero population
embedding signal.

\begin{assumption}[Template--thinning--perturbation model]
\label{def:template-thinning-perturbation}
For each $r\in\{P,Q\}$, let
\begin{equation}
\label{eq:template-diagram}
D_r^\star=\{z_{r1},\ldots,z_{rm_r}\},
\qquad m_r\leq n,
\end{equation}
be an unordered template multiset. For each latent feature $j$, let
$I_{rj}\in\{0,1\}$ satisfy
$\Pr(I_{rj}=1)=p_{rj}\in(0,1]$. Let $Z_{rj}$ be defined on the ambient
probability space, arbitrarily on $\{I_{rj}=0\}$, and satisfy
\[
\mathcal L(Z_{rj}\mid I_{rj}=1)=G_{rj}.
\]
The observed unpadded counting measure is
\begin{equation}
\label{eq:latent-template-counting-measure}
N_{X_r}=\sum_{j=1}^{m_r}I_{rj}\delta_{Z_{rj}}.
\end{equation}
The collection
$\{(I_{rj},Z_{rj}):1\leq j\leq m_r\}$ may have arbitrary joint
dependence.
\end{assumption}

By linearity of expectation,
\[
\mathbb E N_{X_r}
=
\sum_{j=1}^{m_r}p_{rj}G_{rj},
\]
and therefore the padded population mean measure is
\begin{equation}
\label{eq:template-padded-mean-measure}
\overline\mu_r
=
\sum_{j=1}^{m_r}p_{rj}G_{rj}
+
\left(n-\sum_{j=1}^{m_r}p_{rj}\right)\delta_\varnothing.
\end{equation}
Consequently,
\begin{equation}
\label{eq:template-population-mean-embedding}
T_\nu[\overline\mu_r]
=
\sum_{j=1}^{m_r}p_{rj}
\int_{\mathbb H_L}\varphi_\nu^0(x)\,dG_{rj}(x).
\end{equation}
No independence among feature appearances or perturbations is required for
these identities. Repeated latent template locations are permitted; their
prevalence masses are aggregated when the distinct union support is formed
below.

\begin{assumption}[Feature perturbation tails]
\label{ass:template-perturbation-tails}
For every $r\in\{P,Q\}$ and every latent template feature,
\begin{equation}
\label{eq:template-tail-condition}
G_{rj}\{x:d_0(x,z_{rj})>t\}
\leq
A_0\exp\left(-\frac{t^2}{2\sigma_{rj}^2}\right),
\qquad t\geq0,
\end{equation}
where $A_0 > 1$.
\end{assumption}

Define
\begin{equation}
\label{eq:expected-template-displacement}
e_{rj}
:=
\int d_0(x,z_{rj})\,dG_{rj}(x).
\end{equation}
Assumption~\ref{ass:template-perturbation-tails} implies
\begin{equation}
\label{eq:template-tail-mean-bound}
e_{rj}
\leq
A_0\sigma_{rj}\sqrt{\frac{\pi}{2}}.
\end{equation}
Sub-Gaussianity is used only to obtain this explicit first-moment bound; any
condition that controls $e_{rj}$ can be substituted.

\paragraph{Shared approximation bound.}
\begin{lemma}[Approximation by the weighted template response]
\label{lem:center-approximation}
Under (E1)--(E2) and the model in
Assumption~\ref{def:template-thinning-perturbation},
\begin{align}
\Delta_\nu(P,Q)
\geq
\Bigg[
&\left\|
\sum_{j=1}^{m_P}p_{Pj}\varphi_\nu^0(z_{Pj})
-
\sum_{\ell=1}^{m_Q}p_{Q\ell}\varphi_\nu^0(z_{Q\ell})
\right\|_2\notag\\
&-
L_\nu
\left\{
\sum_{j=1}^{m_P}p_{Pj}e_{Pj}
+
\sum_{\ell=1}^{m_Q}p_{Q\ell}e_{Q\ell}
\right\}
\Bigg]_+.
\label{eq:shared-perturbation-bound-general}
\end{align}
Under Assumption~\ref{ass:template-perturbation-tails},
\begin{align}
\Delta_\nu(P,Q)
\geq
\Bigg[
&\left\|
\sum_{j=1}^{m_P}p_{Pj}\varphi_\nu^0(z_{Pj})
-
\sum_{\ell=1}^{m_Q}p_{Q\ell}\varphi_\nu^0(z_{Q\ell})
\right\|_2\notag\\
&-
A_0L_\nu\sqrt{\frac{\pi}{2}}
\left\{
\sum_{j=1}^{m_P}p_{Pj}\sigma_{Pj}
+
\sum_{\ell=1}^{m_Q}p_{Q\ell}\sigma_{Q\ell}
\right\}
\Bigg]_+.
\label{eq:shared-perturbation-bound}
\end{align}
\end{lemma}

\begin{assumption}[Prevalence and identifiability regimes]
\label{ass:prevalence-identifiability-regimes}
Put
\begin{equation}
\label{eq:template-separation-scale}
s:=d_B(D_P^\star,D_Q^\star).
\end{equation}
We use one of the following two, not necessarily disjoint, regimes.
\begin{enumerate}[label=(\roman*)]
\item \textbf{Common scalar prevalence.}
There exists $p\in(0,1]$ such that $p_{rj}=p$ for every genuine template
feature and both $r\in\{P,Q\}$. The pair is certified at its separation
scale:
\[
s\in\mathcal I_\nu,
\qquad
(D_P^\star,D_Q^\star)\in\mathcal C_{\nu,s}.
\]
The deterministic center response then satisfies
\begin{equation}
\label{eq:common-noiseless-response}
\sum_jp\varphi_\nu^0(z_{Pj})
-
\sum_\ell p\varphi_\nu^0(z_{Q\ell})
=
p\{\Phi(D_P^\star;\nu)-\Phi(D_Q^\star;\nu)\}.
\end{equation}
The case $p=1$ is the complete-feature model: every template feature appears
in every observed diagram and only its location is perturbed.

\item \textbf{Feature-specific prevalence with geometric response
identifiability.}
The probabilities $p_{rj}\in(0,1]$ may vary across features and populations.
Let
\begin{equation}
\label{eq:union-support}
\{x_1,\ldots,x_M\}=D_P^\star\cup D_Q^\star,
\qquad M\leq2n,
\end{equation}
be the distinct template locations. Multiplicities are retained in the
following contrasts. Define the prevalence-mass contrast
\begin{equation}
\label{eq:union-masses}
c_i
:=
\sum_{j:z_{Pj}=x_i}p_{Pj}
-
\sum_{\ell:z_{Q\ell}=x_i}p_{Q\ell}
\end{equation}
and the unweighted template-count contrast
\begin{equation}
\label{eq:union-count-contrast}
q_i
:=
\#\{j:z_{Pj}=x_i\}
-
\#\{\ell:z_{Q\ell}=x_i\}.
\end{equation}
If repeated locations are excluded, then $q_i\in\{-1,0,1\}$. Define the
finite response operator
\begin{equation}
\label{eq:response-matrix}
V_{P,Q}a
:=
\sum_{i=1}^M a_i\varphi_\nu^0(x_i),
\qquad a\in\mathbb R^M.
\end{equation}
Then
\begin{align}
\label{eq:heterogeneous-noiseless-response}
\sum_jp_{Pj}\varphi_\nu^0(z_{Pj})
-
\sum_\ell p_{Q\ell}\varphi_\nu^0(z_{Q\ell})
&=V_{P,Q}c,\\
\label{eq:unweighted-template-response}
\Phi(D_P^\star;\nu)-\Phi(D_Q^\star;\nu)
&=V_{P,Q}q.
\end{align}
The second identity remains valid when the template cardinalities differ,
because padded absence slots have zero response. Assume
\[
s\in\mathcal I_\nu,
\qquad
(D_P^\star,D_Q^\star)\in\mathcal C_{\nu,s},
\qquad
c\neq0,
\qquad
q\neq0.
\]
Let
\begin{equation}
\label{eq:pair-response-subspace}
\mathcal S_{P,Q}:=\operatorname{span}\{c,q\},
\end{equation}
and define the restricted gains
\begin{align}
\label{eq:pair-restricted-min-gain}
\sigma_{\min}(V_{P,Q};\mathcal S_{P,Q})
&:=
\inf_{\substack{a\in\mathcal S_{P,Q}\\\|a\|_2=1}}
\|V_{P,Q}a\|_2,\\
\label{eq:pair-restricted-max-gain}
\sigma_{\max}(V_{P,Q};\mathcal S_{P,Q})
&:=
\sup_{\substack{a\in\mathcal S_{P,Q}\\\|a\|_2=1}}
\|V_{P,Q}a\|_2.
\end{align}
Assume
$\sigma_{\min}(V_{P,Q};\mathcal S_{P,Q})>0$, and define
\begin{align}
\label{eq:pair-response-condition-number}
\operatorname{cond}(V_{P,Q};\mathcal S_{P,Q})
&:=
\frac{\sigma_{\max}(V_{P,Q};\mathcal S_{P,Q})}
{\sigma_{\min}(V_{P,Q};\mathcal S_{P,Q})},\\
\label{eq:geometric-retention-factor}
\eta_{P,Q}
&:=
\frac{\|c\|_2}
{\operatorname{cond}(V_{P,Q};\mathcal S_{P,Q})\|q\|_2}.
\end{align}
Under the stated conditions, $\eta_{P,Q}>0$.
\end{enumerate}
\end{assumption}

The two parts are alternative mechanisms for transferring the Assumption \ref{ass:diagram-distortion-floor}-certified
unweighted template separation to the population center response. Under common
scalar prevalence, the weighted contrast is exactly $p$ times the complete
template contrast. Under feature-specific prevalence, the weighted contrast
$V_{P,Q}c$ need not align with the unweighted contrast $V_{P,Q}q$; the
restricted response conditioning on $\mathcal S_{P,Q}$ controls the retained
geometric signal. Common scalar prevalence is a special case, but part~(i) gives the
factor $p$ directly without the additional conditioning notation. The stronger
direct heterogeneous response inequality, which does not require Assumption \ref{ass:diagram-distortion-floor}, is
recorded in Appendix~\ref{app:template-transfer-proofs}.

\begin{theorem}[Assumption \ref{ass:diagram-distortion-floor}-based population lower bounds under two prevalence regimes]
\label{thm:population-lower-bound}
Suppose that Assumption~\ref{ass:additive-interface}, the model in
Assumption~\ref{def:template-thinning-perturbation}, and
Assumption~\ref{ass:template-perturbation-tails} hold.
\begin{enumerate}[label=(\roman*)]
\item Under Assumption~\ref{ass:prevalence-identifiability-regimes}(i) and the
additive instance of Assumption~\ref{ass:diagram-distortion-floor},
\begin{align}
\Delta_\nu(P,Q)
\geq
p\Bigg[
&\rho_-(s;\nu)\notag\\
&-
A_0L_\nu\sqrt{\frac{\pi}{2}}
\left\{
\sum_{j=1}^{m_P}\sigma_{Pj}
+
\sum_{\ell=1}^{m_Q}\sigma_{Q\ell}
\right\}
\Bigg]_+.
\label{eq:common-prevalence-bound}
\end{align}

\item Under Assumption~\ref{ass:prevalence-identifiability-regimes}(ii) and
the additive instance of Assumption~\ref{ass:diagram-distortion-floor},
\begin{align}
\Delta_\nu(P,Q)
\geq
\Bigg[
&\eta_{P,Q}\rho_-(s;\nu)\notag\\
&-
A_0L_\nu\sqrt{\frac{\pi}{2}}
\left\{
\sum_{j=1}^{m_P}p_{Pj}\sigma_{Pj}
+
\sum_{\ell=1}^{m_Q}p_{Q\ell}\sigma_{Q\ell}
\right\}
\Bigg]_+.
\label{eq:heterogeneous-directional-bound}
\end{align}
\end{enumerate}
\end{theorem}

A positive lower bound follows whenever the corresponding structural signal
exceeds the perturbation penalty. If
$\sigma_{rj}\leq\overline\sigma$, each perturbation penalty is bounded by
$2A_0nL_\nu\overline\sigma\sqrt{\pi/2}$; the precise corollary is recorded in
Appendix~\ref{app:template-transfer-proofs}.

\begin{remark}[Geometric retention under heterogeneous prevalence]
The floor $\rho_-(s;\nu)$ is the unweighted diagram-level signal supplied by
Assumption \ref{ass:diagram-distortion-floor}. Under common scalar prevalence, the center weighting is aligned exactly
with the complete template contrast, so the retention factor is $p$. Under
feature-specific prevalence, $c$ need not align with the unweighted count
contrast $q$. The factor $\eta_{P,Q}$ quantifies how much of the Assumption \ref{ass:diagram-distortion-floor}-certified
separation is retained after this reweighting: poor conditioning of
$V_{P,Q}$ on $\mathcal S_{P,Q}$, or a small ratio
$\|c\|_2/\|q\|_2$, weakens the guarantee. In particular, part~(ii) certifies a
positive population signal when
\[
\eta_{P,Q}\rho_-(s;\nu)
>
A_0L_\nu\sqrt{\frac{\pi}{2}}
\left\{
\sum_jp_{Pj}\sigma_{Pj}
+
\sum_\ell p_{Q\ell}\sigma_{Q\ell}
\right\}.
\]
If $c=0$, there is no weighted center contrast. If $q=0$, the complete
template diagrams coincide as multisets and no positive-scale Assumption \ref{ass:diagram-distortion-floor} comparison
is available. If
$\sigma_{\min}(V_{P,Q};\mathcal S_{P,Q})=0$, response cancellation prevents
transfer of the Assumption \ref{ass:diagram-distortion-floor} floor to $c$. Likewise, the result is unavailable outside
the certified pair class or admissible scale set. Failure of any sufficient
condition, or a zero displayed lower bound, is inconclusive rather than proof
that $\Delta_\nu(P,Q)=0$.
\end{remark}

\paragraph{Interpretation and scope.}
Both bounds subtract a perturbation penalty from a structural signal and then
take the positive part; both now use Assumption \ref{ass:diagram-distortion-floor} for geometric interpretation. The
stronger direct heterogeneous inequality in
Proposition~\ref{prop:direct-heterogeneous-response-bound} replaces the
Assumption \ref{ass:diagram-distortion-floor}-based signal by $\|V_{P,Q}c\|_2$ and does not require Assumption \ref{ass:diagram-distortion-floor}.
Theorem~\ref{thm:population-lower-bound}(ii) is its more interpretable Assumption \ref{ass:diagram-distortion-floor}
consequence. These are pair-specific conclusions. Uniform power or minimax
theory requires class-wide lower bounds on the relevant retention, prevalence,
and response quantities and upper bounds on perturbation, introduced only in
the later testing section.

The general Hilbert-space probability and testing theory applies under
Assumption~\ref{ass:fixed-hilbert-embedding}, including to a fixed admissible
Mitra--Virk map. The mean-measure identity and the template results require the
additional additive interface, while truncation results invoke
Assumption~\ref{ass:finite-approximation}. Section~\ref{sec:probability}
develops the limit theory for the embedded observations.


\section{Probabilistic Theory for Embedded Diagrams}
\label{sec:probability}

Section~\ref{sec:embedding} identifies the population quantities represented by
an embedding and the conditions under which they admit geometric
interpretations. This section defines the population mean and covariance,
establishes one- and two-sample limit theory, studies covariance estimation,
and gives quantitative Gaussian and finite-approximation results used in the
subsequent inferential procedures.

Fix an embedding configuration $\nu$ independently of the inferential sample,
as in Assumption~\ref{ass:fixed-hilbert-embedding}. Let
$X_1,\ldots,X_m\overset{\mathrm{i.i.d.}}{\sim}P$ and write
$$
\Psi_i:=\Psi_\nu(X_i)\in\mathcal H_\nu,
\qquad
\overline\Psi_m:=\frac{1}{m}\sum_{i=1}^m\Psi_i.
$$
Once $\nu$ is fixed, the statistical problem is one of inference for a Borel
random element of a separable Hilbert space. The sample size $m$ counts
independent diagrams. The persistence points within a diagram remain
components of one observation and do not increase the effective sample size.

The law of large numbers, central limit theorem, covariance results, and
mean-based testing theory require only the fixed-Hilbert assumptions and the
stated moments. Additivity supplies the padded mean-measure interpretation from
Section~\ref{sec:embedding}. Assumption~\ref{ass:finite-approximation} is used
only when a finite-coordinate procedure is intended to approximate the full
embedding, and the lower-distortion certificate Assumption \ref{ass:diagram-distortion-floor} is not used in this
section.

\subsection{Population mean and covariance}
\label{sec:probability-mean-covariance}

Assume first that
$\mathbb E_P\|\Psi_\nu(X)\|_{\mathcal H_\nu}<\infty$. The Bochner mean
embedding is
$$
\theta_{P,\nu}^{\Psi}
:=
\mathbb E_P\{\Psi_\nu(X)\}
\in\mathcal H_\nu.
$$
For two populations, define
$$
\delta_\nu^{\Psi}(P,Q)
:=
\theta_{P,\nu}^{\Psi}-\theta_{Q,\nu}^{\Psi},
\qquad
\Delta_\nu^{\Psi}(P,Q)
:=
\|\delta_\nu^{\Psi}(P,Q)\|_{\mathcal H_\nu},
$$
consistent with the population signal defined in
Section~\ref{sec:embedding}. The mean embedding is the parameter targeted by
the mean-based procedures developed below. It is not a mean persistence
diagram. For an additive embedding $\Psi_\nu=\Phi(\cdot;\nu)$,
Proposition~\ref{prop:mean-embedding-transfer} gives the additional identity
$\theta_{P,\nu}^{\Phi}=T_\nu[\overline\mu_P]$. For the whole-diagram
Mitra--Virk map, $\theta_{P,\nu}^{\Psi}$ remains an ordinary Hilbert-space
expectation without a padded mean-measure representation.

The mean embedding does not, in general, identify the full law of a random
diagram. Distinct distributions may have the same population mean embedding,
so $P\neq Q$ does not necessarily imply
$\Delta_\nu^{\Psi}(P,Q)>0$. Accordingly, the mean-based procedures in this
paper detect alternatives with distinct population mean embeddings. Universal
consistency against every $P\neq Q$ would require an additional
characteristicness or population-identifiability condition.

If $\mathbb E_P\|\Psi_\nu(X)\|_{\mathcal H_\nu}^2<\infty$, define the
covariance operator
$$
\Sigma_{P,\nu}^{\Psi}
:=
\mathbb E_P
\left[
\{\Psi_\nu(X)-\theta_{P,\nu}^{\Psi}\}
\otimes
\{\Psi_\nu(X)-\theta_{P,\nu}^{\Psi}\}
\right],
$$
where $(u\otimes v)h=\langle v,h\rangle_{\mathcal H_\nu}u$. When the
population and embedding are fixed, we write $\theta$, $\Sigma$, and
$U:=\Psi_\nu(X)-\theta$. The operator $\Sigma$ describes between-diagram
variation after embedding. Its trace is the total variance across orthogonal
embedding directions, whereas its operator norm is the largest directional
variance.

Under a finite first moment, $\theta$ is the unique element satisfying
$\langle\theta,h\rangle=\mathbb E\langle\Psi_\nu(X),h\rangle$ for every
$h\in\mathcal H_\nu$. Under a finite second moment, $\Sigma$ is positive,
self-adjoint, and trace class, with
$$
\operatorname{tr}(\Sigma)
=
\mathbb E\|U\|_{\mathcal H_\nu}^2,
\qquad
\|\Sigma\|_{\mathrm{op}}
\leq
\operatorname{tr}(\Sigma).
$$
If $\|\Psi_\nu(D)\|_{\mathcal H_\nu}\leq B_\nu$ uniformly in $D$, then
$\operatorname{tr}(\Sigma)\leq B_\nu^2$. Trace-class covariance is the
infinite-dimensional analogue of finite total variance. It ensures that the
Gaussian limits below define Radon probability measures on $\mathcal H_\nu$
and that their coordinate variances are summable.

\subsection{Laws of large numbers and central limit theory}
\label{sec:hilbert-limit-theory}

\begin{theorem}[Strong law and Hilbert-space central limit theorem]
\label{thm:hilbert-slln-clt}
Let $\Psi_1,\Psi_2,\ldots$ be i.i.d. copies of $\Psi_\nu(X)$.
\begin{enumerate}[label=(\roman*)]
\item If $\mathbb E\|\Psi_\nu(X)\|_{\mathcal H_\nu}<\infty$, then
$$
\overline\Psi_m\longrightarrow\theta
\qquad\text{almost surely in }\mathcal H_\nu.
$$
\item If $\mathbb E\|\Psi_\nu(X)\|_{\mathcal H_\nu}^2<\infty$, then
$$
\sqrt m\,(\overline\Psi_m-\theta)
\ \rightsquigarrow\
\mathcal G,
\qquad
\mathcal G\sim\mathcal N_{\mathcal H_\nu}(0,\Sigma).
$$
\end{enumerate}
\end{theorem}

The strong law justifies the empirical mean embedding as an estimator of the
population mean embedding. The central limit theorem describes its stochastic
error at the usual $m^{-1/2}$ scale. The Gaussian limit may be degenerate; no
invertibility of $\Sigma$ is required. Every fixed linear functional of the
centered empirical mean is asymptotically normal. Norm-based statistics instead
converge to functionals of the Gaussian random element.

Let $\lambda_1,\lambda_2,\ldots$ denote the eigenvalues of $\Sigma$, repeated
according to multiplicity. The continuous mapping theorem and the spectral
representation of a Hilbert-space Gaussian element give
$$
m\|\overline\Psi_m-\theta\|_{\mathcal H_\nu}^2
\ \rightsquigarrow\
\|\mathcal G\|_{\mathcal H_\nu}^2
\ \stackrel{d}{=}\
\sum_{j\geq1}\lambda_j Z_j^2,
$$
where the $Z_j$ are independent standard normal variables. The series is
well defined because $\sum_j\lambda_j=\operatorname{tr}(\Sigma)<\infty$.
This weighted chi-square limit is the basic null approximation for
mean-distance statistics. Uniform boundedness of the embedding is sufficient
for the moment conditions above, but is not necessary.

For the two-sample problem, define
$$
\overline\Psi_{P,m}:=\frac{1}{m}\sum_{i=1}^m\Psi_\nu(X_i),
\qquad
\overline\Psi_{Q,m'}:=\frac{1}{m'}\sum_{j=1}^{m'}\Psi_\nu(Y_j).
$$
The following form is used in Section~\ref{sec:testing}.

\begin{corollary}[Two-sample Hilbert-space central limit theorem]
\label{cor:two-sample-hilbert-clt}
Let $X_1,\ldots,X_m\overset{\mathrm{i.i.d.}}{\sim}P$ and
$Y_1,\ldots,Y_{m'}\overset{\mathrm{i.i.d.}}{\sim}Q$ be independent, and
assume finite second moments under both laws. Put $N=m+m'$ and suppose
$m/N\to\lambda\in(0,1)$. Then
$$
\sqrt N
\left[
\{\overline\Psi_{P,m}-\overline\Psi_{Q,m'}\}
-
\delta_\nu^{\Psi}(P,Q)
\right]
\ \rightsquigarrow\
\mathcal N_{\mathcal H_\nu}(0,\Sigma_{P,Q,\lambda}),
$$
where
$$
\Sigma_{P,Q,\lambda}
:=
\lambda^{-1}\Sigma_{P,\nu}^{\Psi}
+
(1-\lambda)^{-1}\Sigma_{Q,\nu}^{\Psi}.
$$
\end{corollary}

The factors $\lambda^{-1}$ and $(1-\lambda)^{-1}$ show how sample imbalance
inflates uncertainty. When one group is substantially smaller, its covariance
contribution receives the larger weight.

\subsection{Covariance estimation}
\label{sec:covariance-estimation}

Define the empirical covariance operator
$$
\widehat\Sigma_m
:=
\frac{1}{m}
\sum_{i=1}^m
(\Psi_i-\overline\Psi_m)
\otimes
(\Psi_i-\overline\Psi_m).
$$
The normalization by $m$ is convenient for operator convergence; replacing it
by $m-1$ does not change the conclusions.

\begin{proposition}[Consistency of covariance estimation]
\label{prop:covariance-consistency}
If $\mathbb E\|\Psi_\nu(X)\|_{\mathcal H_\nu}^2<\infty$, then
$$
\|\widehat\Sigma_m-\Sigma\|_{\mathrm{tr}}
\longrightarrow0
\qquad\text{almost surely}.
$$
Consequently, convergence also holds in Hilbert--Schmidt and operator norm. If
$\mathbb E\|U\|_{\mathcal H_\nu}^4<\infty$, then
$$
\mathbb E\|\widehat\Sigma_m-\Sigma\|_{\mathrm{HS}}^2
=
O(m^{-1}),
\qquad
\mathbb E\|\widehat\Sigma_m-\Sigma\|_{\mathrm{op}}^2
=
O(m^{-1}).
$$
In particular, the rate holds for a uniformly bounded embedding.
\end{proposition}

This result supports estimation of directional variances and of the leading
covariance spectrum used in Gaussian and spectral calibration. The rate is
controlled by the number of independent diagrams and the fourth moment of
their embedded variation, not by the number of persistence points within each
diagram.

\subsection{Quantitative Gaussian approximation in finite coordinates}
\label{sec:finite-coordinate-gaussian-approximation}

The Hilbert-space central limit theorem gives weak convergence but does not,
by itself, provide a finite-sample error bound. A quantitative approximation
is available after projection to a fixed finite-dimensional subspace. Let
$\Pi$ be a finite-rank orthogonal projection, define
$$
U_\Pi:=\Pi U,
\qquad
\Sigma_\Pi:=\Pi\Sigma\Pi,
$$
and let $\mathcal R_\Pi$ be the range of $\Sigma_\Pi$, with
$r_\Pi:=\dim(\mathcal R_\Pi)$. Write $\Sigma_\Pi^{\dagger/2}$ for the
Moore--Penrose inverse square root on $\mathcal R_\Pi$ and set
$$
\beta_{3,\Pi}
:=
\mathbb E
\|\Sigma_\Pi^{\dagger/2}U_\Pi\|_{\mathcal H_\nu}^3.
$$

\begin{theorem}[Finite-coordinate Berry--Esseen bound]
\label{thm:finite-coordinate-berry-esseen}
Suppose $0<r_\Pi<\infty$ and $\beta_{3,\Pi}<\infty$. Let
$\mathcal G_\Pi\sim\mathcal N_{\mathcal R_\Pi}(0,\Sigma_\Pi)$. Then there is a
universal constant $C_{\mathrm{BE}}$ such that
$$
\sup_{A\in\mathfrak C(\mathcal R_\Pi)}
\left|
\Pr\{\sqrt m\,\Pi(\overline\Psi_m-\theta)\in A\}
-
\Pr(\mathcal G_\Pi\in A)
\right|
\leq
\frac{C_{\mathrm{BE}}r_\Pi^{1/4}\beta_{3,\Pi}}{\sqrt m},
$$
where $\mathfrak C(\mathcal R_\Pi)$ denotes the convex Borel subsets of
$\mathcal R_\Pi$. The same bound therefore holds uniformly over centered
balls in $\mathcal R_\Pi$.
\end{theorem}

For a finite-dimensional embedding, the theorem may be applied with $\Pi$
equal to the identity. For an infinite-dimensional embedding, it applies to
each fixed finite-coordinate projection. The approximation constant depends
on the projected rank and on covariance conditioning through $\beta_{3,\Pi}$.
A uniform bound on the unstandardized embedding alone does not control this
quantity. Thus the Gaussian approximation can deteriorate when too many weakly
identified coordinates are retained, even if the embedding itself is bounded.
Because centered balls are convex, the theorem bounds the oracle Gaussian
approximation error for probabilities defined by centered norm balls in a
fixed finite-coordinate projection. It does not include error from
covariance estimation, spectral truncation, Monte Carlo approximation, or
data-adaptive selection of the projection.

\subsection{Finite-coordinate approximation}
\label{sec:finite-coordinate-approximation}

A fixed projection defines a valid finite-dimensional target without invoking
Assumption~\ref{ass:finite-approximation}. That assumption is needed only when
the projected statistic is intended to approximate the full embedding. For
$\Psi_{\nu,K}:=\Pi_K\Psi_\nu$, define
$\overline\Psi_{m,K}:=\Pi_K\overline\Psi_m$ and
$\theta_K:=\Pi_K\theta$.

\begin{proposition}[Approximation of means and mean separation]
\label{prop:finite-approximation-first-moments}
Under Assumption~\ref{ass:finite-approximation},
$$
\|\theta-\theta_K\|_{\mathcal H_\nu}
\leq\epsilon_K^\Psi,
\qquad
\|\overline\Psi_m-\overline\Psi_{m,K}\|_{\mathcal H_\nu}
\leq\epsilon_K^\Psi
\quad\text{almost surely},
$$
and
$$
\|(\overline\Psi_m-\theta)
-(\overline\Psi_{m,K}-\theta_K)\|_{\mathcal H_\nu}
\leq2\epsilon_K^\Psi.
$$
For two populations embedded by the same map, define
$\Delta_{\nu,K}^{\Psi}(P,Q)
:=\|\Pi_K\delta_\nu^{\Psi}(P,Q)\|_{\mathcal H_\nu}$. Then
$$
|\Delta_\nu^{\Psi}(P,Q)-\Delta_{\nu,K}^{\Psi}(P,Q)|
\leq2\epsilon_K^\Psi.
$$
If $(K_m)$ is a deterministic sequence satisfying
$\sqrt m\,\epsilon_{K_m}^\Psi\to0$, then the truncation error in the
centered sample mean is negligible at the central-limit scale.
\end{proposition}

The last condition separates approximation error from sampling error. The
empirical mean fluctuates on the $m^{-1/2}$ scale, so a truncation intended to
recover the full-embedding limit must have error smaller than that scale. If
this condition is not imposed, inference remains valid for the projected
target, but it should not be interpreted as inference for the full mean
embedding without accounting for the approximation bias. Section~7 develops
the additional geometric consequences of finite approximation.

\begin{remark}[Scope and provenance]
\label{rem:probability-provenance}
The strong law, Hilbert-space central limit theorem, and finite-dimensional
Berry--Esseen inequality are standard probability results applied after
verifying measurability and the required moments
\citep{Bentkus2003}. The covariance and approximation
propositions record the operator consequences needed by the later testing and
confidence-set procedures. The diagram-specific content lies in the choice of
$\Psi_\nu$, the interpretation of $\theta_{P,\nu}^{\Psi}$, and the verified
bounds that imply the required moments or finite approximation. Additivity is
used only through
$\theta_{P,\nu}^{\Phi}=T_\nu[\overline\mu_P]$. Assumption \ref{ass:diagram-distortion-floor} is not used in this
section; it enters when an embedding-space conclusion is converted into a
geometric certificate.
\end{remark}

Section~\ref{sec:testing} uses the two-sample central limit theorem and covariance consistency
to calibrate tests of equality of population mean embeddings.

\section{Two-Sample Inference for Mean Embeddings}
\label{sec:testing}

Fix an embedding configuration $\nu$ independently of the inferential sample.
Let
$X_1,\ldots,X_m\overset{\mathrm{i.i.d.}}{\sim}P$ and
$Y_1,\ldots,Y_{m'}\overset{\mathrm{i.i.d.}}{\sim}Q$ be independent diagram
samples, and write
$$
\Psi_i:=\Psi_\nu(X_i),
\qquad
\Psi_j':=\Psi_\nu(Y_j),
$$
with empirical means $\overline\Psi_{P,m}$ and
$\overline\Psi_{Q,m'}$ as defined in Section~\ref{sec:probability}. The unit
of replication is the diagram. Points within a diagram are components of one
structured observation and are not treated as independent replicates.

The procedures in this section test equality of population mean embeddings,
$$
H_0^{\mathrm{mean}}:
\theta_{P,\nu}^{\Psi}=\theta_{Q,\nu}^{\Psi}
\qquad\text{versus}\qquad
H_1^{\mathrm{mean}}:
\Delta_\nu^{\Psi}(P,Q)>0.
$$
The distributional null $P=Q$ implies $H_0^{\mathrm{mean}}$, but the converse
need not hold. The tests are therefore not omnibus tests of equality of diagram
laws. Their target is the first-order population signal represented by the
fixed embedding.

Most of the testing theory uses only Assumption~\ref{ass:fixed-hilbert-embedding}
and the moment conditions stated below. Additivity and Assumption~\ref{ass:diagram-distortion-floor} enter only when the mean-embedding alternative is related to the
structured geometric alternatives of Section~\ref{sec:template-population-transfer}.
Finite-coordinate implementations additionally use
Assumption~\ref{ass:finite-approximation} when they are intended to approximate
the full embedding.

The analysis proceeds at three levels. We first study testing in the estimable
separation $\Delta_\nu^\Psi(P,Q)$ without requiring additivity or a
geometric certificate. We then use the structured population results of
Section~\ref{sec:template-population-transfer} to identify geometric
alternatives that produce a positive mean-embedding signal. Finally, we
consider uniform and minimax statements over classes of alternatives. This
last step requires class-wide prevalence, perturbation, response-conditioning,
and information restrictions that are not needed for inference at a fixed
population pair.

\subsection{Canonical statistic and null calibration}
\label{sec:testing-statistic}

Put $N=m+m'$, $\lambda_N=m/N$, and define the effective sample size
$$
N_{\mathrm{eff}}
:=
\frac{mm'}{m+m'}.
$$
The empirical mean difference and the canonical squared-distance statistic are
$$
\widehat\delta_{m,m'}
:=
\overline\Psi_{P,m}-\overline\Psi_{Q,m'},
\qquad
T_{m,m'}
:=
N_{\mathrm{eff}}
\|\widehat\delta_{m,m'}\|_{\mathcal H_\nu}^2.
$$
The scaling by $N_{\mathrm{eff}}$ gives the usual two-sample normalization and
simplifies the covariance under balanced sampling. The squared Hilbert norm
aggregates differences over all embedding coordinates and does not require a
prespecified direction. It is therefore a natural default when the form of the
population difference is unknown. Directional procedures can be more powerful
when scientific information or an independent sample identifies a relevant
low-dimensional contrast.

\begin{theorem}[Asymptotic null distribution of the mean-distance statistic]
\label{thm:mean-distance-null}
Assume finite second moments under $P$ and $Q$, and suppose
$\lambda_N\to\lambda\in(0,1)$. Under $H_0^{\mathrm{mean}}$,
$$
T_{m,m'}
\ \rightsquigarrow\
\|\mathcal G_{P,Q,\lambda}\|_{\mathcal H_\nu}^2
\ \stackrel{d}{=}\
\sum_{j\geq1}\gamma_jZ_j^2,
$$
where
$$
\mathcal G_{P,Q,\lambda}
\sim
\mathcal N_{\mathcal H_\nu}(0,\Gamma_{P,Q,\lambda}),
\qquad
\Gamma_{P,Q,\lambda}
:=(1-\lambda)\Sigma_{P,\nu}^{\Psi}
+\lambda\Sigma_{Q,\nu}^{\Psi},
$$
and $\gamma_1,\gamma_2,\ldots$ are the eigenvalues of
$\Gamma_{P,Q,\lambda}$. Under the stronger null $P=Q$, the limiting covariance
reduces to $\Sigma_{P,\nu}^{\Psi}$.
\end{theorem}

The theorem gives an asymptotic level-$\alpha$ test by comparing
$T_{m,m'}$ with the $(1-\alpha)$ quantile of the weighted chi-square limit.
The group-specific empirical covariance operators yield the estimator
$$
\widehat\Gamma_{m,m'}
:=
\frac{m'}{N}\widehat\Sigma_{P,m}
+
\frac{m}{N}\widehat\Sigma_{Q,m'}.
$$
Let $\widehat\gamma_1,\widehat\gamma_2,\ldots$ denote the nonzero
eigenvalues of $\widehat\Gamma_{m,m'}$, and let
$\widehat q_{1-\alpha}$ be the conditional $(1-\alpha)$ quantile of
$$
\sum_{j\geq1}\widehat\gamma_j Z_j^2,
$$
where the $Z_j$ are independent standard normal variables. The plug-in
test rejects when
$$
T_{m,m'}>\widehat q_{1-\alpha}.
$$
If $\Gamma_{P,Q,\lambda}\neq0$, trace-norm consistency of
$\widehat\Gamma_{m,m'}$ implies consistency of the corresponding Gaussian
law and its $(1-\alpha)$ quantile. The resulting test therefore has
asymptotic level $\alpha$. If $\Gamma_{P,Q,\lambda}=0$, the null
distribution is degenerate and the nonrandom test is conservative.

Trace-norm consistency from Proposition~\ref{prop:covariance-consistency}
justifies spectral plug-in calibration when the leading covariance spectrum is
estimated with adequate precision. For a fixed finite-coordinate projection, an analogous convex-set
Berry--Esseen bound can be obtained by applying the multivariate result to
the weighted independent summands from the two groups. Theorem
\ref{thm:finite-coordinate-berry-esseen} gives the corresponding
one-sample statement. Covariance estimation and projection selection
introduce additional errors not included in either oracle bound.

Calibration depends on the null hypothesis. Gaussian and spectral procedures
target equality of population mean embeddings and allow the two groups to
differ in higher-order distributional features. Under the stronger
distributional null $P=Q$, permutation calibration gives finite-sample validity
because the pooled diagrams are exchangeable. The embedding configuration must
remain fixed. If a statistic uses a data-dependent projection, regularization
parameter, or coordinate selection rule, the complete fitting step must be
repeated within each permutation unless it was estimated from an independent
sample. Unmodified permutation calibration is not generally exact under the
weaker null of equal means. Equality of the two null hypotheses should therefore
not be assumed.

\begin{corollary}[Consistency against mean-embedding alternatives]
\label{cor:mean-distance-consistency}
Under the assumptions of Theorem~\ref{thm:mean-distance-null}, if
$\Delta_\nu^{\Psi}(P,Q)>0$, then
$$
\frac{T_{m,m'}}{N_{\mathrm{eff}}}
\longrightarrow
\{\Delta_\nu^{\Psi}(P,Q)\}^2
\qquad\text{almost surely},
$$
and hence $T_{m,m'}\to\infty$ almost surely. Consequently, any test that
rejects when $T_{m,m'}>c_{m,m'}$ is consistent at the fixed pair $(P,Q)$
provided that $c_{m,m'}=O_{\mathbb P}(1)$ under that pair.
\end{corollary}

A fixed-direction test is useful when a scientifically specified contrast or an
independently estimated direction is available. For a unit vector
$h\in\mathcal H_\nu$, define
$$
\widehat\omega_h^2
:=
\frac{m'}{N}
\langle\widehat\Sigma_{P,m}h,h\rangle
+
\frac{m}{N}
\langle\widehat\Sigma_{Q,m'}h,h\rangle,
$$
$$
Z_{m,m'}(h)
:=
\frac{\sqrt{N_{\mathrm{eff}}}
\langle\widehat\delta_{m,m'},h\rangle}
{\widehat\omega_h}.
$$
For a fixed or independently estimated $h$, Slutsky's theorem gives a standard
normal null limit when the limiting directional variance is positive. The
statement does not cover a direction selected from the same observations and
then treated as fixed. Sample splitting or refitting the selection rule inside
a permutation procedure avoids that dependence.

\subsection{Power in mean-embedding separation}
\label{sec:testing-power}

The primitive signal for the procedures above is
$\Delta_\nu^{\Psi}(P,Q)$. The following finite-sample benchmark makes the
inverse-square signal dependence explicit without imposing Gaussianity.

\begin{proposition}[A finite-sample signal-to-noise bound]
\label{prop:finite-sample-power}
Fix $v<\infty$ and consider laws satisfying
$$
\operatorname{tr}(\Sigma_{P,\nu}^{\Psi})\leq v^2,
\qquad
\operatorname{tr}(\Sigma_{Q,\nu}^{\Psi})\leq v^2.
$$
Define
$$
\phi_{\alpha,v}
:=
\mathbf 1
\left\{
\|\widehat\delta_{m,m'}\|_{\mathcal H_\nu}
>
\frac{v}{\sqrt{\alpha N_{\mathrm{eff}}}}
\right\}.
$$
Then $\phi_{\alpha,v}$ has level at most $\alpha$ under
$H_0^{\mathrm{mean}}$. At a fixed alternative, its power is at least
$1-\beta$ whenever
$$
N_{\mathrm{eff}}
\geq
\frac{v^2}{\{\Delta_\nu^{\Psi}(P,Q)\}^2}
\left(
\alpha^{-1/2}+\beta^{-1/2}
\right)^2.
$$
\end{proposition}

The constants in Proposition~\ref{prop:finite-sample-power} are conservative;
the result uses only a second-moment bound. Gaussian plug-in calibration can yield substantially smaller critical values when the covariance is estimated accurately. Under the stronger null $P=Q$, permutation calibration can also be less conservative while retaining finite-sample validity. Ordinary permutation calibration does not
generally control level under the weaker null of equal mean embeddings. The proposition isolates the main rate and separates three components of testing difficulty. The factor
$\{\Delta_\nu^\Psi(P,Q)\}^{-2}$ is the inverse-square effect-size penalty,
$v^2$ measures between-diagram variability in the embedded space, and the
factor involving $\alpha$ and $\beta$ records the desired Type I and Type II
error control. Section~\ref{sec:template-population-transfer} enters only
through a lower bound on $\Delta_\nu(P,Q)$; it does not change this basic
signal-to-noise structure. With bounded between-diagram variance, the proposition gives a sufficient effective sample size of order $\Delta^{-2}$ for testing a fixed pair, where $\Delta=\Delta_\nu^\Psi(P,Q)$.
When $\|\Psi_\nu(D)\|_{\mathcal H_\nu}\leq B_\nu$ uniformly, one may take $v=B_\nu$ in this benchmark. Without a known deterministic envelope, the test $\phi_{\alpha,v}$ is a theoretical finite-sample benchmark rather than a directly calibrated procedure.

For a fixed direction $h$, consider a sequence of population pairs
$(P_N,Q_N)$. Write
$$
\delta_{h,N}
:=
\left\langle
\delta_\nu^\Psi(P_N,Q_N),h
\right\rangle
$$
and
$$
\omega_{h,N}^2
:=
(1-\lambda_N)
\left\langle\Sigma_{P_N,\nu}^{\Psi}h,h\right\rangle
+
\lambda_N
\left\langle\Sigma_{Q_N,\nu}^{\Psi}h,h\right\rangle.
$$
Suppose the corresponding scalar triangular-array central limit theorem
holds,
$$
\omega_{h,N}^2\longrightarrow\omega_h^2>0,
\qquad
\sqrt{N_{\mathrm{eff}}}\,\delta_{h,N}\longrightarrow d_h>0.
$$
Then the one-sided test that rejects for
$Z_{m,m'}(h)>z_{1-\alpha}$ has limiting power
$$
1-\Phi_0\left(
z_{1-\alpha}-\frac{d_h}{\omega_h}
\right),
$$
where $\Phi_0$ is the standard normal distribution function. If the direction were known and chosen as
$h=\delta_\nu^{\Psi}(P,Q)/\Delta_\nu^{\Psi}(P,Q)$, this yields the oracle
sample-size benchmark
$$
N_{\mathrm{eff}}
\asymp
\frac{\omega_h^2}
{\{\Delta_\nu^{\Psi}(P,Q)\}^2},
\qquad
\omega_h^2
\leq
\|\Gamma_{P,Q,\lambda}\|_{\mathrm{op}}.
$$
This oracle calculation explains the operator-norm signal-to-noise ratio, but
it is not an implementable test unless the direction is prespecified or learned
independently.

\subsection{Structured geometric alternatives}
\label{sec:structured-testing-alternatives}

Section~\ref{sec:template-population-transfer} supplies pair-specific lower
bounds that convert a certified template separation into a lower bound on
$\Delta_\nu(P,Q)$ for additive embeddings. Define the pair-specific
perturbation penalties
$$
\mathcal E_{\mathrm{com}}(P,Q)
:=
A_0L_\nu\sqrt{\frac{\pi}{2}}
\left\{
\sum_{j=1}^{m_P}\sigma_{Pj}
+
\sum_{\ell=1}^{m_Q}\sigma_{Q\ell}
\right\},
$$
$$
\mathcal E_{\mathrm{het}}(P,Q)
:=
A_0L_\nu\sqrt{\frac{\pi}{2}}
\left\{
\sum_{j=1}^{m_P}p_{Pj}\sigma_{Pj}
+
\sum_{\ell=1}^{m_Q}p_{Q\ell}\sigma_{Q\ell}
\right\}.
$$
Here $m_P$ and $m_Q$ denote the numbers of features in the two template
diagrams. They are unrelated to the statistical sample sizes $m$ and $m'$. The pair-specific structural margins are
$$
G_{\mathrm{com}}(P,Q)
:=
p
\left[
\rho_-(s;\nu)-\mathcal E_{\mathrm{com}}(P,Q)
\right]_+,
$$
$$
G_{\mathrm{het}}(P,Q)
:=
\left[
\eta_{P,Q}\rho_-(s;\nu)
-
\mathcal E_{\mathrm{het}}(P,Q)
\right]_+.
$$
Theorem~\ref{thm:population-lower-bound} states that
$\Delta_\nu(P,Q)\geq G_{\mathrm{com}}(P,Q)$ in the common-prevalence regime
and $\Delta_\nu(P,Q)\geq G_{\mathrm{het}}(P,Q)$ in the
feature-specific regime.

Uniform power requires uniform restrictions on these pair-specific quantities.
The next definition specifies one such parameter class. It is introduced here,
rather than in Section~3, because it concerns a collection of alternatives
rather than a fixed population pair.

\begin{definition}[Uniform structured alternative classes at scale $\tau$]
\label{ass:uniform-structured-alternatives}
Fix a configuration $\nu$ satisfying
Assumptions~\ref{ass:additive-interface}
and~\ref{ass:diagram-distortion-floor}. Let
$\tau$ belong to the admissible scale set for this configuration, and let
$0 < v<\infty$.

\begin{enumerate}[label=(\roman*)]
\item The common-prevalence class
$\mathfrak A_{\mathrm{com}}(\tau)$ consists of pairs satisfying the model and
tail assumptions of Section~\ref{sec:template-population-transfer},
Assumption~\ref{ass:prevalence-identifiability-regimes}(i), and
$$
s\geq\tau,
\qquad
p\geq\pi_{\min}>0,
\qquad
\mathcal E_{\mathrm{com}}(P,Q)
\leq
\overline e_{\mathrm{com}}(\tau),
$$
with $\operatorname{tr}(\Sigma_{P,\nu}^{\Phi})$ and
$\operatorname{tr}(\Sigma_{Q,\nu}^{\Phi})$ bounded by $v^2$.

\item The feature-specific class
$\mathfrak A_{\mathrm{het}}(\tau)$ consists of pairs satisfying the model and
tail assumptions of Section~\ref{sec:template-population-transfer},
Assumption~\ref{ass:prevalence-identifiability-regimes}(ii), and
$$
s\geq\tau,
\qquad
\eta_{P,Q}\geq\eta_{\min}>0,
\qquad
\mathcal E_{\mathrm{het}}(P,Q)
\leq
\overline e_{\mathrm{het}}(\tau).
$$
For a fixed constant $p_*\in(0,1/2)$, the prevalence probabilities also
satisfy
$$
p_*\leq p_{Pj}\leq1-p_*,
\qquad
p_*\leq p_{Q\ell}\leq1-p_*
$$
for every template feature $j$ and $\ell$. The covariance traces are bounded
by $v^2$.
\end{enumerate}

Because $\rho_-$ is the threshold floor supplied by
Assumption~\ref{ass:diagram-distortion-floor},
$\rho_-(s;\nu)\geq\rho_-(\tau;\nu)$ whenever $s\geq\tau$ and both scales
are in the applicable admissible domain. Define
$$
g_{\mathrm{com}}(\tau)
:=
\pi_{\min}
\left[
\rho_-(\tau;\nu)-\overline e_{\mathrm{com}}(\tau)
\right]_+,
$$
$$
g_{\mathrm{het}}(\tau)
:=
\left[
\eta_{\min}\rho_-(\tau;\nu)
-
\overline e_{\mathrm{het}}(\tau)
\right]_+.
$$
The uniform-power statements below are restricted to scales for which the
corresponding margin is strictly positive.
\end{definition}

These restrictions define the alternative classes used for uniform power;
they are not asserted to hold for unrestricted diagram populations. The
prevalence and response-retention bounds exclude rare-mass and cancellation
sequences, while the perturbation bounds ensure that the certified structural
signal remains positive. If the embedding is uniformly bounded by $B_\nu$,
the covariance trace condition follows with $v=B_\nu$.

For $r\in\{\mathrm{com},\mathrm{het}\}$, suppose that
$\mathfrak A_r(\tau)$ is nonempty and define its intrinsic signal by
$$
\underline\Delta_r(\tau)
:=
\inf_{(P,Q)\in\mathfrak A_r(\tau)}
\Delta_\nu(P,Q).
$$
The structured population bounds imply
$$
\underline\Delta_r(\tau)\geq g_r(\tau).
$$
Thus $g_r(\tau)$ is an explicit certified lower bound for the smallest
mean-embedding signal in the class, rather than a definition of that signal.
The resulting power bound may be conservative when
$\underline\Delta_r(\tau)$ is substantially larger than $g_r(\tau)$.

\begin{corollary}[Uniform power over structured alternatives]
\label{cor:structured-uniform-power}
Let $r\in\{\mathrm{com},\mathrm{het}\}$, and suppose that
$\mathfrak A_r(\tau)$ is nonempty. Under
Definition~\ref{ass:uniform-structured-alternatives}, suppose
$g_r(\tau)>0$. Let $\phi_{\alpha,v}$ be the test from Proposition~\ref{prop:finite-sample-power}. Then
$$
\inf_{(P,Q)\in\mathfrak A_r(\tau)}\Pr_{P,Q}\{\phi_{\alpha,v}=1\} \geq 1-\beta
$$
whenever
$$
N_{\mathrm{eff}} \geq \frac{v^2}{g_r(\tau)^2} \left(\alpha^{-1/2}+\beta^{-1/2}\right)^2.
$$
\end{corollary}

These restrictions define the classes used for uniform power. They are not
asserted to hold for unrestricted diagram populations. The prevalence and
response-retention conditions exclude rare-mass and cancellation sequences.
The perturbation bounds control the loss of structural signal, while the
condition $g_r(\tau)>0$ ensures that a positive uniform margin remains. If
the embedding is uniformly bounded by $B_\nu$, the covariance trace
condition holds with $v=B_\nu$.

To make the unrestricted comparison explicit, define
$$
\mathfrak A_{\mathrm{sep}}(\tau) := \left\{ (P,Q): W_{\infty,0}(\overline\mu_P,\overline\mu_Q)\geq\tau \right\}.
$$
This class imposes geometric separation but no lower bound on the amount of
population mass carrying that separation.

\begin{proposition}[No uniform power from geometric separation alone]
\label{prop:no-uniform-power-unrestricted}
Fix the rare-mass construction of
Example~\ref{ex:pop-obstruction}, for which
$$ W_{\infty,0}(\overline\mu_{P_\varepsilon},\overline\mu_Q)=\tau$$
for every $\varepsilon>0$. For every sample size $m$, every test
$\phi_m$ satisfying
$$
\Pr_{Q,Q}(\phi_m=1)\leq\alpha,
$$
and every $\beta<1-\alpha$, there exists $\varepsilon>0$ such that
$$
\Pr_{Q,P_\varepsilon}(\phi_m=1)<1-\beta.
$$
Consequently, $\mathfrak A_{\mathrm{sep}}(\tau)$ has no finite uniform
sample complexity.
\end{proposition}

\subsection{Minimax bounds through an information modulus}
\label{sec:testing-minimax}

Section~3 and Definition~\ref{ass:uniform-structured-alternatives} provide the
uniform signal margin needed for an upper testing bound. A lower bound requires
a different quantity: the smallest information divergence between a structured
alternative and the mean-embedding null. Defining this quantity explicitly
avoids imposing the existence of a least-favorable parametric submodel as a
standing assumption.

Define the mean-null class
$$
\mathfrak H_0^{\mathrm{mean}}(v)
:=
\left\{
(P_0,Q_0):
\theta_{P_0,\nu}^{\Phi}=\theta_{Q_0,\nu}^{\Phi},
\ \operatorname{tr}(\Sigma_{P_0,\nu}^{\Phi})\leq v^2,
\ \operatorname{tr}(\Sigma_{Q_0,\nu}^{\Phi})\leq v^2
\right\}.
$$
Let $m_{\mathrm{com}}^*(\alpha,\beta,\tau)$ be the smallest equal per-group
sample size for which some test has level at most $\alpha$ uniformly over
$\mathfrak H_0^{\mathrm{mean}}(v)$ and power at least $1-\beta$ uniformly over
$\mathfrak A_{\mathrm{com}}(\tau)$. The corresponding one-observation information modulus is

$$
\mathcal K_{\mathrm{com}}(\tau) :=
\inf_{\substack{(P,Q)\in\mathfrak A_{\mathrm{com}}(\tau)\\
(P_0,Q_0)\in\mathfrak H_0^{\mathrm{mean}}(v)}}
\left\{\operatorname{KL}(P\|P_0)+\operatorname{KL}(Q\|Q_0)\right\},
$$
where a divergence is $+\infty$ when absolute continuity fails.

The upper bound requires only the class-wide signal restriction in
Definition~\ref{ass:uniform-structured-alternatives}. For the minimax lower
bound, the alternative class must also contain a regular hard subfamily. We
state this requirement over the scale set on which the lower bound will be
used.

\begin{assumption}[Common-prevalence minimax richness]
\label{ass:common-prevalence-minimax-richness}
Fix a nonempty scale set
$\mathcal T_{\mathrm{com}}\subseteq\mathcal I_\nu$. There are constants $c_\rho>0$, $\gamma\in(0,1)$, and $e_0>0$,
independent of $\tau$, such that
$$
\rho_-(\tau;\nu)\ge c_\rho\tau, \quad \text{ and } \quad e_0 \le\overline e_{\mathrm{com}}(\tau) \le (1-\gamma)\rho_-(\tau;\nu)
$$
for every $\tau\in\mathcal T_{\mathrm{com}}$.
\end{assumption}

\begin{assumption}[Feature-specific-prevalence minimax richness]
\label{ass:feature-specific-prevalence-minimax-richness}
Fix a nonempty scale set
$\mathcal T_{\mathrm{het}}\subseteq\mathcal I_\nu$. Suppose the following
conditions hold.

\begin{enumerate}[label=(\roman*)]

\item There are an integer $M\leq n$, a compact set $ K_{\mathrm{het}} \subset \{(b,d):0<b<d<L\},$
and a constant $C_0<\infty$ such that the following holds. For every
$\tau\in\mathcal T_{\mathrm{het}}$, there are alternative centers
$z_{P1,\tau},\ldots,z_{PM,\tau},$ and $z_{Q1,\tau},\ldots,z_{QM,\tau},$
common reference centers
$z_{1,\tau}^0,\ldots,z_{M,\tau}^0,$ and prevalence vectors
$p_{P,\tau},p_{Q,\tau},p_\tau^0\in[p_*,1-p_*]^M$, such that
all centers belong to $K_{\mathrm{het}}$. Define
$$
D_{P,\tau}^\star=\{z_{P1,\tau},\ldots,z_{PM,\tau}\},
\qquad
D_{Q,\tau}^\star=\{z_{Q1,\tau},\ldots,z_{QM,\tau}\}.
$$
The alternative templates satisfy
$$
d_B(D_{P,\tau}^\star,D_{Q,\tau}^\star)=\tau \quad
\text{ and } \quad
(D_{P,\tau}^\star,D_{Q,\tau}^\star) \in\mathcal C_{\nu,\tau}.
$$
Moreover,
$$
\begin{aligned}
&\sum_{j=1}^M
\left\{
\|z_{Pj,\tau}-z_{j,\tau}^0\|_2^2
+
\|z_{Qj,\tau}-z_{j,\tau}^0\|_2^2
\right\}
+
\sum_{j=1}^M
\left\{
|p_{Pj,\tau}-p_{j,\tau}^0|^2
+
|p_{Qj,\tau}-p_{j,\tau}^0|^2
\right\}
\leq
C_0\tau^2.
\end{aligned}
$$

\item Let $c_\tau$, $q_\tau$, $V_\tau$, and $\mathcal S_\tau
=\operatorname{span}\{c_\tau,q_\tau\}$
be the prevalence contrast, count contrast, response operator, and restricted
contrast space associated with
$(D_{P,\tau}^\star,D_{Q,\tau}^\star)$. There are constants
$\kappa_*<\infty$ and $c_*>0$, independent of $\tau$, such that
$$
\operatorname{cond}(V_\tau;\mathcal S_\tau)
\leq\kappa_* \quad
\text{ and } \quad
\|c_\tau\|_2 \geq c_*\|q_\tau\|_2.
$$
The class constant satisfies
$$
\eta_{\min}\leq\frac{c_*}{\kappa_*}.
$$

\item There are constants
$c_{\rho,\mathrm{het}}>0$,
$\gamma_{\mathrm{het}}\in(0,1)$, and
$e_{0,\mathrm{het}}>0$, independent of $\tau$, such that
$$
\rho_-(\tau;\nu) \geq c_{\rho,\mathrm{het}}\tau \quad 
\text{ and } \quad
e_{0,\mathrm{het}} \leq \overline e_{\mathrm{het}}(\tau)
\leq (1-\gamma_{\mathrm{het}}) \eta_{\min}\rho_-(\tau;\nu)
$$
for every $\tau\in\mathcal T_{\mathrm{het}}$.

\end{enumerate}
\end{assumption}

\begin{remark}[Scale conditions for PLACE and PALACE]
\label{rem:minimax-scale-conditions}
The linear-floor conditions in
Assumption~\ref{ass:common-prevalence-minimax-richness} and
Assumption~\ref{ass:feature-specific-prevalence-minimax-richness}(iii)
are conditions on the embedding configuration and the selected scale set.
They can be verified directly for PLACE and PALACE.

For the PLACE step certificate, suppose
$\mathcal T_r \subseteq [3R_1,\overline\tau]$, where $r\in\{\mathrm{com},\mathrm{het}\},
$ with $\overline\tau<\infty$ and $w_1>0$. Then
$$
\rho_-^{\mathrm{PLACE},\mathrm{step}}(\tau;\nu) \geq \frac{w_1R_1}{16} \geq
\frac{w_1R_1}{16\overline\tau}\tau.
$$
Thus the required linear comparison holds on every bounded scale interval
above the PLACE step threshold.

For the PLACE affine certificate, suppose the scale set is bounded away from
$R_1$. If $\tau\geq(1+\delta)R_1$
for some fixed $\delta>0$, then
$$
\begin{aligned}
\rho_-^{\mathrm{PLACE},\mathrm{aff}}(\tau;\nu)
&=
\lambda(\nu)(\tau-R_1)
&\geq
\lambda(\nu)\frac{\delta}{1+\delta}\tau.
\end{aligned}
$$
The linear comparison therefore holds on any such restricted scale set.

For PALACE, suppose the embedding configuration is finite and all its weights
are strictly positive. Let $w_{\min}:=\min_{1\leq k\leq J_\nu}w_k>0.$
At every admissible scale,
$$
\begin{aligned}
\rho_-^{\mathrm{PALACE}}(\tau;\nu)
&=
\frac{\tau}{4}
\min_{k:r_k\geq\tau/4}w_k
&\geq
\frac{w_{\min}}{4}\tau.
\end{aligned}
$$
Hence the PALACE floor also satisfies the required linear comparison on its
admissible scale set.

The remaining inequalities involving
$\overline e_{\mathrm{com}}(\tau)$ and
$\overline e_{\mathrm{het}}(\tau)$ define the noise regime considered in the
minimax result. Their upper bounds require the permitted perturbation to remain
a fixed fraction below the certified signal. By contrast, the positive lower bounds
$e_0>0$ and $e_{0,\mathrm{het}}>0$ are technical conditions for the particular
fixed-noise hard families considered here. They are not needed for the
population lower bounds or the uniform-power results.
\end{remark}

\begin{remark}[Interpretation of feature-specific minimax richness]
\label{rem:feature-specific-minimax-richness-interpretation}
Parts~(i) and~(ii) of
Assumption~\ref{ass:feature-specific-prevalence-minimax-richness}
describe a regular local family of feature-specific alternatives. They are
conditions on the parameter class and the fixed embedding configuration. 

Part~(i) requires the model class to contain alternatives obtained by small
changes in feature locations and prevalence probabilities around a common
reference configuration. This condition is natural when the template features
have positive persistence, remain away from the boundary of the persistence
region, and retain a fixed correspondence along the selected path. The
reference features should also remain separated enough to avoid feature
collisions or changes in the relevant bottleneck matching. These are local
regularity conditions on the diagram model. They do not depend on the specific
form of the PLACE or PALACE coordinates.

For PLACE, such local paths can be selected inside the bounded persistence
region. The template pairs must also satisfy the relevant PLACE certification
condition. For PALACE, the reference configuration and its local path must lie
inside the covered support $\mathcal R$. The PALACE configuration must remain
admissible over the selected scale set. Thus part~(i) is compatible with both
constructions, although the PALACE support requirement must be checked for the
chosen configuration.

Part~(ii) is a uniform restricted singular-value condition. It requires the
response operator to remain identifiable and well-conditioned on the
at most two-dimensional space generated by the count and prevalence
contrasts. A convenient sufficient condition is
$$
\inf_{\tau\in\mathcal T_{\mathrm{het}}}
\sigma_{\min}(V_\tau;\mathcal S_\tau)>0,
\qquad
\sup_{\tau\in\mathcal T_{\mathrm{het}}}
\sigma_{\max}(V_\tau;\mathcal S_\tau)<\infty,
$$
together with
$$
\inf_{\tau\in\mathcal T_{\mathrm{het}}}
\frac{\|c_\tau\|_2}{\|q_\tau\|_2}>0.
$$
For a fixed finite-dimensional PLACE or PALACE configuration, these are
finite-matrix conditions. They can be checked along the proposed hard path.
If the path is continuous, the dimension of $\mathcal S_\tau$ is constant,
and the scale set is compact, pointwise injectivity of $V_\tau$ on
$\mathcal S_\tau$ gives a uniform restricted lower gain.

The condition can fail when different template features have nearly identical
landmark responses or when their prevalence contributions cancel. It may also
fail if the path passes through a configuration where the restricted response
matrix loses rank. Boundedness and Lipschitz continuity of PLACE and PALACE do
not rule out these cases. Thus part~(ii) is a substantive identifiability
condition on the interaction between the template path, the prevalence
pattern, and the selected landmarks.

No additional boundedness condition is required. PLACE and PALACE satisfy the
bounded additive-interface assumptions already imposed in
Section~\ref{sec:additive-interface}. The bounded-cardinality restriction on
$\mathcal D_n$ then provides the covariance envelope needed for the structured
alternative classes.
\end{remark}


The preceding conditions identify scale regimes in which the structured
alternative classes contain regular hard families with a nonvanishing
certified margin. We now compare this margin with the information distance
between the alternative class and the mean-null class. 

\begin{lemma}[Information-modulus comparison for structured prevalence regimes]
\label{lem:information-modulus-comparison}
For $r\in\{\mathrm{com},\mathrm{het}\}$, let
$\mathcal T_r$ denote $\mathcal T_{\mathrm{com}}$ when
$r=\mathrm{com}$ and $\mathcal T_{\mathrm{het}}$ when
$r=\mathrm{het}$. Let the structured alternative
classes and margins be as in
Definition~\ref{ass:uniform-structured-alternatives}. Suppose the corresponding
minimax condition holds: Assumption~\ref{ass:common-prevalence-minimax-richness}
when $r=\mathrm{com}$, and
Assumption~\ref{ass:feature-specific-prevalence-minimax-richness} when
$r=\mathrm{het}$. Then there is a constant
$C_{\mathrm{KL},r}<\infty$, independent of
$\tau\in\mathcal T_r$, such that
$$
\mathcal K_r(\tau)
\leq
C_{\mathrm{KL},r}g_r(\tau)^2.
$$
\end{lemma}

The proof is given in Appendix~\ref{app:minimax-information-comparison}.
For each regime, it constructs a structured alternative and a nearby
mean-null pair whose KL divergence is of order $\tau^2$. The minimax-richness
conditions convert this bound to order $g_r(\tau)^2$. Together with
Corollary~\ref{cor:structured-uniform-power}, this yields the following
two-sided sample-complexity bound.

\begin{theorem}[Minimax sample complexity for structured prevalence regimes]
\label{thm:structured-prevalence-minimax-rate}
Fix $r\in\{\mathrm{com},\mathrm{het}\}$ and
$0<\alpha<1-\beta<1$. Suppose the conditions of
Lemma~\ref{lem:information-modulus-comparison} hold. Then, for every
$\tau\in\mathcal T_r$, up to integer rounding,
$$
\frac{2(1-\beta-\alpha)^2}
{C_{\mathrm{KL},r}g_r(\tau)^2}
\leq
m_r^*(\alpha,\beta,\tau)
\leq
\frac{2v^2}{g_r(\tau)^2}
\left(
\alpha^{-1/2}+\beta^{-1/2}
\right)^2.
$$
Thus, with the embedding and regime-specific class constants held fixed, the
upper and lower sample-complexity bounds have matching dependence on the
certified margin $g_r(\tau)$ over the corresponding restricted scale regime.
\end{theorem}

The margin $g_r(\tau)$ enters the two bounds for different reasons. It
lower-bounds the mean-embedding signal uniformly for the upper bound, while
the lower bound uses one hard alternative and null pair with KL divergence of
order $g_r(\tau)^2$. The conclusion applies only to the fixed configurations,
class constants, and scale ranges covered by the stated assumptions. It does
not assert that the diagram-level certificate is sharp for unrestricted
diagram distributions or arbitrary landmark configurations.

\subsection{Relation to STRAND and scope of the test}
\label{sec:testing-strand}

Among recent approaches to two-sample inference for collections of persistence
diagrams, STRAND \citep{STRAND} is particularly close in scope to the present
work. It compares groups through the marginal survival distribution of feature
lifetimes. It pools the persistence values from all
diagrams in each group and applies a log-rank test. The asymptotic calibration
and its power calculation are based on treating the pooled feature lifetimes
as an i.i.d.\ sample within each group. This assumption can be strong because
features from the same diagram may be dependent. It also makes the nominal
sample size depend on the number of features rather than the number of
diagrams. When diagram cardinalities vary, pooling weights diagrams according
to their numbers of features.

STRAND also considers diagram-level permutation calibration, which preserves
within-diagram dependence. Such calibration requires exchangeability of the
complete diagrams under the null. The present framework instead uses each
diagram as one independent observation and estimates covariance across
diagrams. It therefore accommodates dependence among features within a
diagram without a feature-level independence assumption.

The methods also target different aspects of persistence diagrams. STRAND uses
the lifetime $d-b$ and is invariant to translations along the diagonal. It
provides Kaplan--Meier localization and an interpretable hazard ratio when the
proportional-hazards assumption holds. A median-persistence difference is
available otherwise. Our method retains the joint birth--death locations
through the chosen embedding. It can therefore detect differences between
diagrams that have similar lifetime distributions but differ in where their
features occur. A rejection here concerns the mean embedding, whereas a
STRAND rejection concerns the feature-lifetime distribution. Neither test
alone characterizes equality of the full diagram distributions.

The next section uses the same Hilbert-space limit theory to construct
confidence sets and geometric certificates.


\section{Confidence Sets and Geometric Certificates}
\label{sec:confidence-certificates}

The preceding sections use the empirical mean embedding to test equality of
population means. The same limit theory also provides uncertainty sets for a
single mean embedding and for a two-population mean contrast. This section
separates three interpretations of those sets. First, they are ordinary
Hilbert-space confidence balls for estimable population parameters. Second,
for additive embeddings they yield lower confidence bounds and exclusion
certificates in the transport geometry of padded mean measures. Third, when
class centroids are estimated from labeled diagrams, the same confidence balls
certify whether a nearest-centroid decision is insensitive to centroid
estimation and to a prescribed bottleneck perturbation of the query diagram.

Throughout, the embedding configuration is fixed independently of the sample.
The primary confidence target is a population mean embedding, not a mean
diagram. A population average of persistence diagrams need not itself be a
persistence diagram, whereas its Hilbert-space mean and, for additive maps,
its padded mean measure are well defined.

A confidence set and a certificate answer different questions. The confidence
set quantifies sampling uncertainty about an unknown population parameter. A
certificate combines that uncertainty set with a deterministic property of the
embedding to support a conclusion about a candidate mean measure, a
population contrast, or a classification decision. The confidence probability
comes from the statistical set; the geometric interpretation comes from the
embedding interface. Failure to obtain a positive certificate is inconclusive:
it is not evidence for the opposite geometric or decision-theoretic conclusion.

\subsection{Confidence balls for a population mean embedding}
\label{sec:one-sample-confidence}

Let $X_1,\ldots,X_m\overset{\mathrm{i.i.d.}}{\sim}P$, write
$\Psi_i:=\Psi_\nu(X_i)$, and let
$$
\theta:=\theta_{P,\nu}^{\Psi},
\qquad
\overline\Psi_m:=\frac{1}{m}\sum_{i=1}^m\Psi_i,
\qquad
\Sigma:=\Sigma_{P,\nu}^{\Psi}.
$$
Here $m$ is the number of independent diagrams. Persistence points within the
same diagram are not treated as independent observations.

Let $\mathcal G\sim\mathcal N_{\mathcal H_\nu}(0,\Sigma)$ and denote by
$q_{1-\alpha}(\Sigma)$ the $(1-\alpha)$ quantile of
$\|\mathcal G\|_{\mathcal H_\nu}^2$. The oracle confidence ball is
$$
\mathcal C_{m,\mathrm{or}}^{\alpha}
:=
\left\{
 h\in\mathcal H_\nu:
 m\|\overline\Psi_m-h\|_{\mathcal H_\nu}^2
 \leq q_{1-\alpha}(\Sigma)
\right\}.
$$
Although its critical value is a weighted chi-square quantile, this set is a
norm ball rather than a covariance-whitened ellipsoid. The norm ball is aligned
with the target $\|\theta-h\|$ and with the Lipschitz inequalities used below
for geometric interpretation. In a fixed finite-dimensional projection, a
covariance-whitened ellipsoid may be more efficient for inference in selected
directions. In the full Hilbert space, however, such a construction generally
requires truncation or regularization because $\Sigma^{-1}$ need not be a
bounded operator. We therefore use the norm ball, whose radius depends on the
covariance spectrum through $q_{1-\alpha}(\Sigma)$.

For implementation, let $\widehat\Sigma_m$ be the empirical covariance
operator from Section~\ref{sec:covariance-estimation}. Conditional on the data,
let
$$
\mathcal G_m^*\sim
\mathcal N_{\mathcal H_\nu}(0,\widehat\Sigma_m)
$$
and let $\widehat q_{1-\alpha,m}$ be the corresponding conditional quantile of
$\|\mathcal G_m^*\|^2$. Define
$$
\widehat{\mathcal C}_m^{\alpha}
:=
\left\{
 h\in\mathcal H_\nu:
 m\|\overline\Psi_m-h\|^2
 \leq \widehat q_{1-\alpha,m}
\right\},
\qquad
\widehat r_m^{\alpha}
:=
\sqrt{\widehat q_{1-\alpha,m}/m}.
$$

\begin{theorem}[Asymptotic confidence ball]
\label{thm:mean-confidence-ball}
Fix $P$ and a fixed embedding configuration $\nu$. Assume
$\mathbb E_P\|\Psi_\nu(X)\|_{\mathcal H_\nu}^2<\infty$ and let
$0<\alpha<1$. If $\Sigma\neq0$, then
$$
\Pr_P\left\{
\theta_{P,\nu}^{\Psi}
\in
\widehat{\mathcal C}_m^{\alpha}
\right\}
\longrightarrow
1-\alpha.
$$
If $\Sigma=0$, then $\Psi_\nu(X)=\theta$ almost surely and the zero-radius
ball has exact coverage one.
\end{theorem}

The coverage statement is pointwise in $P$ for the fixed embedding
configuration. Uniform coverage over a class of diagram distributions would
require additional assumptions and is not claimed here.

There are two equivalent ways to simulate the plug-in critical value. Since
$\widehat\Sigma_m$ has rank at most $m-1$, let
$\widehat\lambda_1,\ldots,\widehat\lambda_{\widehat r_m}$ denote its nonzero
eigenvalues, where
$\widehat r_m:=\operatorname{rank}(\widehat\Sigma_m)\leq m-1$.
A spectral implementation draws
$$
\sum_{j=1}^{\widehat r_m}\widehat\lambda_j Z_j^2,
$$
where the $Z_j$ are independent standard normal variables. A Gaussian
multiplier implementation draws independent $\xi_i\sim\mathcal N(0,1)$ and
computes
$$
\mathcal G_m^*
=
\frac{1}{\sqrt m}
\sum_{i=1}^m
\xi_i(\Psi_i-\overline\Psi_m).
$$
Conditionally on the data, this is exactly Gaussian with covariance
$\widehat\Sigma_m$. Ignoring Monte Carlo error, full spectral simulation and
the Gaussian multiplier procedure therefore target the same plug-in law. The
multiplier implementation is usually the simpler default because it avoids an
explicit eigendecomposition. Spectral simulation is useful when the empirical
eigenvalues are already available or when one wishes to inspect which
covariance directions dominate the radius. Spectral truncation is an additional
numerical approximation and should retain enough covariance trace for the
omitted tail to be negligible.

Theorem~\ref{thm:finite-coordinate-berry-esseen} controls the oracle Gaussian
approximation error for a fixed finite-coordinate projection. It does not by
itself include covariance-estimation, spectral-truncation, or data-adaptive
projection error. Permutation calibration, discussed in Section~\ref{sec:testing},
is a two-sample procedure under exchangeability and is not a one-sample
calibration method for $\theta$.

\paragraph{Finite-sample benchmark.}
A conservative confidence ball is available without estimating a Gaussian
quantile. If $\operatorname{tr}(\Sigma)\leq v^2$, Markov's inequality gives
$$
\Pr_P\left\{
\|\overline\Psi_m-\theta\|
\leq
\frac{v}{\sqrt{m\alpha}}
\right\}
\geq1-\alpha
$$
for every $m\geq1$. Equivalently, the ball centered at $\overline\Psi_m$ with
radius $v/\sqrt{m\alpha}$ has finite-sample coverage at least $1-\alpha$. If
$\|\Psi_\nu(D)\|\leq B_\nu$ uniformly, one may take $v=B_\nu$.
This result is a distribution-free benchmark rather than the recommended
default calibration: the radius can be wide, and a useful deterministic value
of $v$ may be unavailable or conservative. Its main role is to show what can
be guaranteed from a second-moment envelope alone and to make explicit the
diagram-level sampling rate $m^{-1/2}$.

\subsection{Pullback and geometric exclusion for mean measures}
\label{sec:mean-measure-confidence-pullback}

Now suppose the embedding is additive and satisfies
Assumption~\ref{ass:additive-interface}. Let
$\mathfrak M_n$ be the set of nonnegative Borel measures of total mass $n$ on
the padded state space $\mathbb H_L^{\varnothing}$. For any Hilbert-space
confidence set $\mathcal C_m$ for
$\theta_{P,\nu}^{\Phi}=T_\nu[\overline\mu_P]$, define its mean-measure
preimage by
$$
\mathfrak C_m
:=
\left\{
\eta\in\mathfrak M_n:
T_\nu[\eta]\in\mathcal C_m
\right\}.
$$
If $\mathcal C_m$ has coverage at least $1-\alpha$, then so does
$\mathfrak C_m$ for the padded population mean measure $\overline\mu_P$.
This is the exact pullback statement supported by the additive representation.
In general, $\mathfrak C_m$ need not be a transport ball: distinct mean
measures can have the same embedded image, and
Assumption~\ref{ass:diagram-distortion-floor} is stated for certified pairs of
diagrams rather than arbitrary pairs in $\mathfrak M_n$. The preimage is
therefore primarily a set-theoretic coverage statement and may be too implicit
to enumerate. Its operational use is to evaluate a prespecified candidate mean
measure, template family, or fitted model through its embedded image.

The population transport inequality instead gives a useful one-sided geometric
certificate.

\begin{proposition}[Confidence-based transport exclusion]
\label{prop:mean-measure-exclusion-certificate}
Let $\mathcal C_m\subseteq\mathcal H_\nu$ satisfy
$$
\Pr_P\left\{
T_\nu[\overline\mu_P]\in\mathcal C_m
\right\}
\geq1-\alpha.
$$
Then, with probability at least $1-\alpha$, every
$\eta\in\mathfrak M_n$ satisfies
$$
W_{\infty,0}(\eta,\overline\mu_P)
\geq
\frac{\operatorname{dist}\{T_\nu[\eta],\mathcal C_m\}}
{nL_\nu}.
$$
If $\mathcal C_m$ is the ball
$\{h:\|h-\overline\Psi_m\|\leq r_m\}$, the lower bound becomes
$$
W_{\infty,0}(\eta,\overline\mu_P)
\geq
\frac{
\left[
\|T_\nu[\eta]-\overline\Psi_m\|-r_m
\right]_+}
{nL_\nu}.
$$
\end{proposition}

For a proposed mean measure $\eta_0$ and a confidence ball of radius $r_m$,
the quantity
$$
\frac{
\left[
\|T_\nu[\eta_0]-\overline\Psi_m\|-r_m
\right]_+}{nL_\nu}
$$
is a lower confidence certificate for its transport distance from the true
population mean measure. The proposition does not assert that candidates whose
embeddings lie inside the confidence ball are close in $W_{\infty,0}$. It states
the valid converse: a candidate lying far outside the embedding confidence set
is, with the same confidence level, separated from the true mean measure by at
least the displayed transport distance.

The numerical strength of the certificate depends on the sharpness of the
constant $nL_\nu$. A large cardinality bound or a conservative point-level
Lipschitz constant can make the lower bound small. A zero or small certificate
is inconclusive and does not imply that the candidate is geometrically close.

\subsection{Two-sample confidence sets and certified effect sizes}
\label{sec:two-sample-confidence}

For the samples of Section~\ref{sec:testing}, write
$$
\delta:=\delta_\nu^{\Psi}(P,Q),
\qquad
\widehat\delta_{m,m'}
:=
\overline\Psi_{P,m}-\overline\Psi_{Q,m'},
\qquad
N_{\mathrm{eff}}:=\frac{mm'}{m+m'}.
$$
Let $\widehat\Gamma_{m,m'}$ be the covariance estimator defined in
Section~\ref{sec:testing}. If $\widehat q_{1-\alpha,m,m'}$ is the conditional
$(1-\alpha)$ quantile of the squared norm of a centered Gaussian element with
covariance $\widehat\Gamma_{m,m'}$, put
$$
\widehat{\mathcal D}_{m,m'}^{\alpha}
:=
\left\{
 h\in\mathcal H_\nu:
 N_{\mathrm{eff}}
 \|\widehat\delta_{m,m'}-h\|^2
 \leq
 \widehat q_{1-\alpha,m,m'}
\right\}
$$
and
$$
\widehat r_{m,m'}^{\alpha}
:=
\sqrt{
\widehat q_{1-\alpha,m,m'}/N_{\mathrm{eff}}
}.
$$

\begin{corollary}[Confidence set for the mean contrast]
\label{cor:two-sample-confidence-ball}
Under the moment and sample-balance conditions of
Corollary~\ref{cor:two-sample-hilbert-clt}, suppose first that
$\Gamma_{P,Q,\lambda}\neq0$. Then
$$
\Pr_{P,Q}\left\{
\delta_\nu^{\Psi}(P,Q)
\in
\widehat{\mathcal D}_{m,m'}^{\alpha}
\right\}
\longrightarrow
1-\alpha.
$$
If $\Gamma_{P,Q,\lambda}=0$, then both embedded populations are almost surely
constant. In that case,
$\widehat\delta_{m,m'}=\delta_\nu^\Psi(P,Q)$ and
$\widehat r_{m,m'}^\alpha=0$ almost surely, so the confidence set has exact
coverage one.
Consequently,
$$
\underline\Delta_{m,m'}^{\alpha}
:=
\left[
\|\widehat\delta_{m,m'}\|
-
\widehat r_{m,m'}^{\alpha}
\right]_+,
\qquad
\overline\Delta_{m,m'}^{\alpha}
:=
\|\widehat\delta_{m,m'}\|
+
\widehat r_{m,m'}^{\alpha}
$$
satisfy
$$
\liminf_{m,m'\to\infty}
\Pr_{P,Q}\left\{
\underline\Delta_{m,m'}^{\alpha}
\leq
\Delta_\nu^{\Psi}(P,Q)
\leq
\overline\Delta_{m,m'}^{\alpha}
\right\}
\geq
1-\alpha.
$$
Thus they form an asymptotically valid, potentially conservative,
$(1-\alpha)$ confidence interval for $\Delta_\nu^{\Psi}(P,Q)$.
\end{corollary}

The two endpoints have ordinary effect-size interpretations. The lower endpoint
answers how much mean-embedding separation is supported by the data, whereas
the upper endpoint describes how large the separation could still be after
accounting for sampling uncertainty. The structured model below permits a
further inversion of the second statement: if even the upper endpoint is below
the minimum signal required by a proposed class, that class is incompatible
with the data at the stated pointwise asymptotic level.

For additive embeddings, Proposition~\ref{prop:mean-embedding-transfer}
turns the lower endpoint into a geometric effect-size certificate:
$$
\liminf_{m,m'\to\infty}
\Pr_{P,Q}\left\{
W_{\infty,0}(\overline\mu_P,\overline\mu_Q)
\geq
\frac{\underline\Delta_{m,m'}^{\alpha}}{nL_\nu}
\right\}
\geq
1-\alpha.
$$
This is an asymptotically valid lower confidence bound on separation of the
padded mean measures. It requires only the additive Lipschitz interface rather
than Assumption~\ref{ass:diagram-distortion-floor} or the structured population
model.

The upper endpoint supports a complementary structured exclusion. For
$r\in\{\mathrm{com},\mathrm{het}\}$ and a scale $\tau$ for which
$g_r(\tau)>0$, Definition~\ref{ass:uniform-structured-alternatives} and
Theorem~\ref{thm:population-lower-bound} imply that membership in
$\mathfrak A_r(\tau)$ requires
$\Delta_\nu(P,Q)\geq g_r(\tau)$. Therefore, for every fixed $(P,Q)\in\mathfrak A_r(\tau)$,
$$
\limsup_{m,m'\to\infty}
\Pr_{P,Q}\left\{
\overline\Delta_{m,m'}^{\alpha}<g_r(\tau)
\right\}
\leq\alpha.
$$
Thus an observed upper endpoint below $g_r(\tau)$ supports exclusion of the
specified structured class at pointwise asymptotic level $\alpha$. This is a
compatibility certificate rather than a uniform test over
$\mathfrak A_r(\tau)$. A uniform guarantee would require uniform Gaussian
approximation and covariance-estimation results. The certificate is useful
only when the constants defining $g_r(\tau)$ are scientifically defensible
and numerically available. When they are unavailable, the confidence interval
for $\Delta_\nu^\Psi(P,Q)$ remains the primary inferential output.

\subsection{A secondary certificate for nearest-centroid decisions}
\label{sec:centroid-certificates}

Consider a fixed number $C\geq2$ of labeled populations with means
$\theta_c:=\mathbb E\{\Psi_\nu(X)\mid Y=c\}$ and empirical centroids
$\widehat\theta_c$. Suppose confidence radii $r_c$ are chosen so that
$$
\Pr\left\{
\|\widehat\theta_c-\theta_c\|\leq r_c
\ \text{for every }c
\right\}
\geq1-\alpha.
$$
For example, classwise confidence balls may be combined by Bonferroni
allocation. The probability statement is finite-sample when the classwise
radii come from finite-sample-valid balls. It is asymptotic when plug-in
Gaussian balls are used and the training sample size in every class tends to
infinity. Given a query diagram $D$, set $z:=\Psi_\nu(D)$ and let
$$
\widehat c(D)
:=
\arg\min_{1\leq c\leq C}
\|z-\widehat\theta_c\|.
$$
For $c'\neq\widehat c(D)$, define the empirical pairwise margin
$$
M_{\widehat c,c'}(D)
:=
\|z-\widehat\theta_{c'}\|
-
\|z-\widehat\theta_{\widehat c}\|.
$$
The empirical centroids determine Voronoi cells in the embedding space.
Centroid uncertainty can reduce the observed margin against a competitor by at
most $r_{\widehat c}+r_{c'}$. Under the additive Lipschitz interface, a
bottleneck perturbation of size $\varepsilon$ can reduce the same margin by a
further $2nL_\nu\varepsilon$. The following result records the margin needed to
absorb both effects.

\begin{proposition}[Centroid-uncertainty and perturbation certificate]
\label{prop:centroid-robustness-certificate}
On the simultaneous centroid-coverage event, if
$$
M_{\widehat c,c'}(D)
>
r_{\widehat c}+r_{c'}
$$
for every $c'\neq\widehat c(D)$, then the empirical assignment agrees with
the nearest-centroid assignment based on the population means
$\theta_1,\ldots,\theta_C$.

If, in addition, the classifier uses the additive embedding
$\Psi_\nu=\Phi(\cdot;\nu)$ satisfying
Assumption~\ref{ass:additive-interface}, and
$$
M_{\widehat c,c'}(D)
>
r_{\widehat c}+r_{c'}+2nL_\nu\varepsilon
$$
for every $c'\neq\widehat c(D)$, then every diagram $E$ satisfying $d_B(D,E)\leq\varepsilon$ has population
nearest-centroid label $\widehat c(D)$. Equivalently,
$$
\arg\min_{1\leq j\leq C}
\|\Phi(E;\nu)-\theta_j\|
=
\widehat c(D).
$$
Define
$$
\varepsilon_{\mathrm{cert}}(D)
:=
\frac{1}{2nL_\nu}
\min_{c'\neq\widehat c(D)}
\left[
M_{\widehat c,c'}(D)
-r_{\widehat c}-r_{c'}
\right]_+.
$$
Then every $E$ with
$d_B(D,E)<\varepsilon_{\mathrm{cert}}(D)$ has population nearest-centroid
label $\widehat c(D)$. The strict inequality avoids ambiguity when a
query lies exactly on a Voronoi boundary.
\end{proposition}

A positive $\varepsilon_{\mathrm{cert}}(D)$ certifies local robustness. A zero
value means only that the observed margin and centroid precision are
insufficient to certify the decision; it does not show that the assignment is
incorrect or unstable. The margin calculation is deterministic and
finite-sample conditional on the simultaneous centroid-coverage event. Its
probability guarantee is finite-sample when the radii are obtained from
finite-sample-valid confidence balls and asymptotic when plug-in Gaussian balls
are used.

This certificate is deliberately narrower than a classification-risk theorem.
It controls changes caused by estimating the class centroids and by perturbing
the query diagram. It does not imply that all observations from class $c$ lie
in the Voronoi cell of $\theta_c$, nor does it remove overlap among the
class-conditional distributions. Classification remains secondary here: its
role is to illustrate how uncertainty sets for embedded population means can
be reused for a query-specific decision certificate.

Two distinct scales govern the conclusions in this section. Statistical
uncertainty is controlled by covariance, sample size, and calibration. The
conversion of image-space separation into a transport lower bound is controlled
by the additive Lipschitz constant $nL_\nu$. Assumption~\ref{ass:diagram-distortion-floor}
does not determine ordinary confidence coverage; it enters only through the
structured margin $g_r(\tau)$ used to exclude a specified alternative class.

Section~\ref{sec:finite-approximation} studies how finite-coordinate approximation changes the confidence
radii, effect-size bounds, and geometric certificates above.


\section{Finite Approximation and Inferential Resolution}
\label{sec:finite-approximation}

The preceding sections formulate inference in the Hilbert space generated by a
fixed embedding. Practical implementations retain only finitely many
coordinates. This reduction creates two distinct inferential targets. For a
fixed truncation level $K$, the projected embedding is itself a valid
finite-dimensional representation, and inference for its population mean is
well defined without treating it as an approximation. If the objective is
instead to recover the full-embedding estimand or its geometric certificate,
the omitted coordinate tail must be accounted for. This section quantifies
that distinction and connects the sample size, the retained dimension, and the
geometric scale under study.

Throughout, Assumption~\ref{ass:finite-approximation} holds. Write
$\Psi_{\nu,K}:=\Pi_K\Psi_\nu$, let
$d_K:=\operatorname{rank}(\Pi_K)$ denote the retained coordinate dimension,
and abbreviate $\epsilon_K^\Psi$ by $\epsilon_K$ when the embedding is the
additive map $\Phi$. For the ordered landmark representation,
$\Pi_K$ retains the first $K$ prespecified coordinate blocks. Thus $K$ counts
blocks, whereas $d_K$ is the actual Euclidean dimension used by the
computation. For two populations, define the projected mean-contrast vector by
$$
\delta_{\nu,K}^{\Psi}(P,Q)
:=
\Pi_K\delta_\nu^\Psi(P,Q).
$$
Its norm is
$\Delta_{\nu,K}^\Psi(P,Q)
=
\|\delta_{\nu,K}^{\Psi}(P,Q)\|$. 

\subsection{Truncation error and two inferential targets}
\label{sec:truncation-error-targets}

Suppose the ordered representation has an orthogonal Hilbert direct-sum
decomposition compatible with $\Pi_K$,
$\Psi_\nu(D)=(\Psi_\nu^{[1]}(D),\Psi_\nu^{[2]}(D),\ldots)$, and define
$$
b_j:=\sup_{D\in\mathcal D_n}
\|\Psi_\nu^{[j]}(D)\|.
$$
Whenever the block envelopes are square summable, the approximation error
satisfies the deterministic bound
$$
\epsilon_K^\Psi
\leq
\left(\sum_{j>K}b_j^2\right)^{1/2}.
$$
Construction-specific formulas may provide sharper or more readily evaluated
tail bounds. The statistical theory below requires only a valid deterministic
upper bound on $\epsilon_K^\Psi$.

The decay of the block envelopes determines how quickly finite approximation
becomes useful. For example, if $b_j\lesssim j^{-\beta}$ for some
$\beta>1/2$, then
$$
\epsilon_K^\Psi\lesssim K^{1/2-\beta}.
$$
If instead $b_j\lesssim e^{-\gamma j}$ for some $\gamma>0$, then
$$
\epsilon_K^\Psi\lesssim e^{-\gamma K}.
$$
Thus polynomial block decay can require a substantial retained
representation, whereas exponential decay permits a comparatively small
truncation. These conclusions are conditional on construction-specific,
uniform block envelopes. If the embedding has only $J$ blocks, then
$\epsilon_K^\Psi=0$ for every $K\geq J$, and the full-embedding theory
reduces to ordinary finite-dimensional inference.

Section~\ref{sec:finite-coordinate-approximation} established that the full
and projected population means differ by at most $\epsilon_K^\Psi$, and that
two-population mean separations differ by at most $2\epsilon_K^\Psi$. The
orthogonality of $\Pi_K$ yields a sharper lower certificate.

\begin{proposition}[Signal retained by an orthogonal truncation]
\label{prop:orthogonal-truncation-signal}
Under Assumption~\ref{ass:finite-approximation},
$$
\Delta_{\nu,K}^{\Psi}(P,Q)
\geq
\left[
\{\Delta_\nu^{\Psi}(P,Q)\}^2
-4(\epsilon_K^\Psi)^2
\right]_+^{1/2}.
$$
Moreover, suppose $(D,E)$ satisfies the construction-specific certification
conditions in Assumption~\ref{ass:diagram-distortion-floor} and
$
d_B(D,E)\geq t>R_0(\nu).
$
Then
$$
\|\Psi_{\nu,K}(D)-\Psi_{\nu,K}(E)\|
\geq
\rho_{-,K}(t;\nu),
$$
where
$$
\rho_{-,K}(t;\nu)
:=
\left[
\rho_-(t;\nu)^2
-4(\epsilon_K^\Psi)^2
\right]_+^{1/2}.
$$
Finally, specialize to the additive embedding
$\Psi_\nu=\Phi(\cdot;\nu)$. For
$r\in\{\mathrm{com},\mathrm{het}\}$, the structured alternative class
$\mathfrak A_r(\tau)$ satisfies
$$
\inf_{(P,Q)\in\mathfrak A_r(\tau)}
\Delta_{\nu,K}(P,Q)
\geq
g_{r,K}(\tau),
$$
where
$$
g_{r,K}(\tau)
:=
\left[
g_r(\tau)^2-4\epsilon_K^2
\right]_+^{1/2}.
$$
\end{proposition}

The proof is given in
Appendix~\ref{app:orthogonal-truncation-signal}. The proposition describes a certificate, not an exact characterization of
what the truncation preserves. If $2\epsilon_K$ exceeds the full lower margin,
the displayed bound becomes zero, but the projected diagrams or population
means may still be separated. A zero truncation certificate is therefore
inconclusive; it does not imply that the omitted coordinates contain all of
the signal.

A fixed $K$ defines the projected estimand $\Delta_{\nu,K}^{\Psi}(P,Q)$. Tests
and confidence sets calibrated for that estimand remain valid even when
$\epsilon_K^\Psi$ is not small. Their natural null is equality of projected
means, which is weaker than equality of the full mean embeddings. Thus a
fixed truncation cannot detect alternatives supported entirely in omitted
coordinates. Approximation error matters when the procedures are interpreted
as statements about the full estimand $\Delta_\nu^{\Psi}(P,Q)$ or about a
full-embedding geometric certificate. A positive value of
$g_{r,K}(\tau)$ rules out complete projected collapse for the stated
structured class.

\subsection{Testing under finite approximation}
\label{sec:truncated-testing}

In this subsection, set $\Psi_\nu=\Phi(\cdot;\nu)$, because the
structured classes $\mathfrak A_r(\tau)$ and margins $g_r(\tau)$ are
defined for the additive embedding. Let
$$
\widehat\delta_{m,m',K}
:=
\Pi_K\widehat\delta_{m,m'},
$$
and define the projected finite-sample benchmark
$$
\phi_{\alpha,v}^{(K)}
:=
\mathbf 1\left\{
\|\widehat\delta_{m,m',K}\|
>
\frac{v}{\sqrt{\alpha N_{\mathrm{eff}}}}
\right\}.
$$
Its natural null is equality of projected mean embeddings. Write
$$
\mathfrak H_{0,K}^{\mathrm{mean}}(v)
:=
\left\{
(P,Q):
\Pi_K\theta_{P,\nu}^{\Phi}
=
\Pi_K\theta_{Q,\nu}^{\Phi},\quad
\operatorname{tr}(\Sigma_{P,K})\leq v^2,\quad
\operatorname{tr}(\Sigma_{Q,K})\leq v^2
\right\},
$$
where $\Sigma_{P,K}=\Pi_K\Sigma_{P,\nu}^{\Phi}\Pi_K$ and similarly
for $Q$. The full mean null is contained in this projected null. Orthogonal
projection cannot increase covariance trace, so a variance envelope valid for
the full embedding is valid for every truncation.

\begin{theorem}[Sample size, truncation, and structured resolution]
\label{thm:finite-approximation-three-way}
Fix $K\geq1$, $r\in\{\mathrm{com},\mathrm{het}\}$,
$0<\alpha<1$, and $0<\beta<1$. Let $\tau\in\mathcal I_\nu$ be an
admissible scale satisfying $g_{r,K}(\tau)>0$. Under the sampling setup of
Section~\ref{sec:testing} and
Definition~\ref{ass:uniform-structured-alternatives}, the test
$\phi_{\alpha,v}^{(K)}$ has level at most $\alpha$ uniformly over
$\mathfrak H_{0,K}^{\mathrm{mean}}(v)$. It has power at least $1-\beta$
uniformly over $\mathfrak A_r(\tau)$ whenever
$$
N_{\mathrm{eff}} \geq \frac{v^2}{g_{r,K}(\tau)^2}\left(\alpha^{-1/2}+\beta^{-1/2}\right)^2.
$$
Equivalently, using the certified truncation margin, it is sufficient that
$$
N_{\mathrm{eff}} \geq
\frac{v^2}{g_r(\tau)^2-4\epsilon_K^2}
\left(\alpha^{-1/2}+\beta^{-1/2}\right)^2,
\qquad
2\epsilon_K<g_r(\tau).
$$
\end{theorem}

The proof is given in
Appendix~\ref{app:finite-approximation-testing}. The bound separates two sources of signal loss. The term
$4\epsilon_K^2$ accounts for the worst-case effect of truncation, while
$$
\frac{v^2}{N_{\mathrm{eff}}}
\left(
\alpha^{-1/2}+\beta^{-1/2}
\right)^2
$$
accounts for sampling variability. The target scale enters through
$g_r(\tau)$, the truncation level enters through $\epsilon_K$, and the number
of independent diagrams enters through $N_{\mathrm{eff}}$. The retained
dimension $d_K$ affects computation and the accuracy of Gaussian or spectral
calibration, but it does not appear separately in this conservative
second-moment bound.

For fixed $K$ and $N_{\mathrm{eff}}$, define the certified inferential
resolution set
$$ \mathcal R_r(K,N_{\mathrm{eff}};\alpha,\beta)
:=
\left\{
\tau\in\mathcal I_\nu:
g_r(\tau)^2
\geq
4\epsilon_K^2
+
\frac{v^2}{N_{\mathrm{eff}}}
\left(
\alpha^{-1/2}+\beta^{-1/2}
\right)^2
\right\}. $$
A scale $\tau$ belongs to this set when the certified structured signal is
large enough to absorb both approximation loss and sampling variability. This
definition does not require $g_r(\tau)$ to be monotone in $\tau$; it simply
collects the scales that the stated bound can certify at the chosen design.

For the benchmark test above, nested orthogonal projections and the common
threshold imply nested rejection events: $\|\Pi_K\widehat\delta\|$ is
nondecreasing in $K$. Its rejection probability therefore cannot decrease as
coordinates are added. Nonmonotone power can arise for procedures whose
critical values, studentization, or covariance regularization change with
$K$, including Gaussian, spectral, and projected Hotelling-type procedures.

A sharper result is available when there is a constant
$0<v_K<\infty$ such that
$$
\max\left\{
\operatorname{tr}(\Sigma_{P,K}),
\operatorname{tr}(\Sigma_{Q,K})
\right\}
\leq
v_K^2
$$
uniformly over the projected null and alternative classes under
consideration. Define
$$
\phi_{\alpha,v_K}^{(K)}
:=
\mathbf 1\left\{
\|\widehat\delta_{m,m',K}\|
>
\frac{v_K}{\sqrt{\alpha N_{\mathrm{eff}}}}
\right\}.
$$
The preceding argument then gives level at most $\alpha$ over
$\mathfrak H_{0,K}^{\mathrm{mean}}(v_K)$ and power at least $1-\beta$ over
$\mathfrak A_r(\tau)$ whenever
$$
N_{\mathrm{eff}}
\geq
\frac{v_K^2}{g_{r,K}(\tau)^2}
\left(
\alpha^{-1/2}+\beta^{-1/2}
\right)^2.
$$
Thus
$$
\frac{g_{r,K}(\tau)}{v_K}
$$
is the corresponding certified signal-to-noise ratio. The common envelope
$v$ is retained in Theorem~\ref{thm:finite-approximation-three-way} because
it does not require a separate variance analysis for every $K$.

For the nested ordered block projections considered here, $\epsilon_K$ is
nonincreasing in $K$. For $\gamma\in(0,1)$, define the certifying truncation index
$$
K_r^\star(\tau;\gamma)
:=
\inf\left\{
K:2\epsilon_K\leq\gamma g_r(\tau)
\right\},
$$
with the convention that the infimum is infinite if the set is empty. At any
$K\geq K_r^\star(\tau;\gamma)$,
$$
g_{r,K}(\tau)
\geq
\sqrt{1-\gamma^2}\,g_r(\tau),
$$
so the certified sample-size inflation relative to the full-margin benchmark
is at most $(1-\gamma^2)^{-1}$. This formulation allows the analyst to choose
a tolerable approximation loss rather than treating one numerical cutoff as a
universal phase transition.

At the diagram level, the analogous certifiable scale set is
$$
\mathcal I_{\nu,K}^{\Psi,\mathrm{cert}}
:=
\left\{
t\in\mathcal I_\nu^\Psi:
2\epsilon_K^\Psi<\rho_-(t;\nu)
\right\}.
$$
For scales outside this set, the finite representation may still separate the
pair, but Assumption~\ref{ass:diagram-distortion-floor} together with the
available uniform tail bound does not certify it.

\subsection{Confidence sets with approximation bias}
\label{sec:truncated-confidence}

Finite-dimensional confidence procedures can be interpreted in two ways. A
confidence ball constructed in $\operatorname{range}(\Pi_K)$ directly covers
the projected mean. To cover the full mean, its radius must also absorb the
omitted-coordinate bias.

\begin{proposition}[Bias-aware inflation of projected confidence balls]
\label{prop:truncated-confidence-inflation}
Suppose a projected one-sample ball centered at $\overline\Psi_{m,K}$ with
radius $r_{m,K}$ covers $\theta_K$ with probability at least $1-\alpha$.
Then the Hilbert-space ball centered at the same point with radius
$r_{m,K}+\epsilon_K^\Psi$ covers the full mean $\theta$ with probability at
least $1-\alpha$.

For two populations, suppose a projected contrast ball centered at
$\widehat\delta_{m,m',K}$ with radius $r_{m,m',K}$ covers
$\delta_{\nu,K}^{\Psi}(P,Q)$ with probability at least $1-\alpha$. Then the
ball with radius $r_{m,m',K}+2\epsilon_K^\Psi$ covers the full contrast
$\delta_\nu^{\Psi}(P,Q)$ with at least the same probability. Consequently,
$$
\left[
\|\widehat\delta_{m,m',K}\|-r_{m,m',K}-2\epsilon_K^\Psi
\right]_+
\leq
\Delta_\nu^{\Psi}(P,Q)
\leq
\|\widehat\delta_{m,m',K}\|+r_{m,m',K}+2\epsilon_K^\Psi
$$
on the projected coverage event.
\end{proposition}

The proof and the growing-truncation calculation are given in
Appendix~\ref{app:truncated-confidence}. The statistical radius and approximation allowance play different roles. The
first decreases with the number of diagrams and depends on the projected
covariance. The second is deterministic for a prespecified truncation and does
not disappear with additional observations unless $K$ also increases. Thus a
large sample cannot compensate for a fixed approximation bias when the target
is the full embedding.

Let $K=K_{m,m'}$ be a deterministic sequence chosen independently of the
inferential sample. Two approximation requirements should be distinguished.
For the centered sampling fluctuations,
Appendix~\ref{app:truncated-confidence} shows that
$$
\begin{aligned}
&\mathbb E_{P,Q}
\left\|
\sqrt{N_{\mathrm{eff}}}
\left[
(\widehat\delta_{m,m'}-\delta_\nu^\Psi)
-
(\widehat\delta_{m,m',K_{m,m'}}-
 \delta_{\nu,K_{m,m'}}^\Psi)
\right]
\right\|^2 \\
&\qquad\leq
\left(\epsilon_{K_{m,m'}}^\Psi\right)^2.
\end{aligned}
$$
Consequently,
$\epsilon_{K_{m,m'}}^\Psi\to0$ is sufficient for the displayed difference
to converge to zero in $L^2$, and hence in probability. By contrast, if the projected estimator
is centered directly at the full target without an explicit bias correction,
a sufficient condition is
$$
\sqrt{N_{\mathrm{eff}}}\,
\epsilon_{K_{m,m'}}^\Psi
\longrightarrow0,
$$
because $\|\delta_{\nu,K}^\Psi-\delta_\nu^\Psi\|\leq
2\epsilon_K^\Psi$. A dimension-uniform quantitative Gaussian approximation
requires additional control of projected rank, covariance conditioning, and
standardized third moments, as discussed in
Section~\ref{sec:finite-coordinate-gaussian-approximation}.

The distinction can be summarized through three regimes. For fixed $K$,
inference is valid for the projected estimand whether or not the truncation
error is small. If $K$ increases so that $\epsilon_K^\Psi\to0$, the
centered projected sampling fluctuations recover the centered full-embedding
limit. If the stronger condition
$\sqrt{N_{\mathrm{eff}}}\epsilon_K^\Psi\to0$ holds, the projected
estimator may be centered directly at the full target without a first-order
bias allowance. The final regime is stronger because deterministic
approximation error must vanish faster than statistical uncertainty.

\subsection{Practical selection of the truncation level}
\label{sec:truncation-practical-protocol}

A theorem-aligned workflow separates statistical design from coordinate
selection.

\begin{enumerate}[label=\arabic*.]
\item Fix the embedding configuration, coordinate order, and candidate
      truncation levels independently of the inferential sample. If $K$ is
      selected by minimizing a $p$-value or maximizing a studentized
      separation on the same data, the fixed-$K$ calibration is generally
      invalid. Valid alternatives include sample splitting, simultaneous
      calibration over the candidate values, or repeating the complete
      selection step within a resampling procedure.
\item Obtain a construction-specific deterministic upper bound on
      $\epsilon_K^\Psi$ and record the retained dimension $d_K$. If no valid
      tail bound is available, inference remains valid for the projected
      estimand, but it should not be interpreted as inference for the full
      embedding.
\item Decide whether the target is the projected estimand or the full
      embedding. No approximation correction is required for the former.
\item For structured geometric interpretation at scale $\tau$, choose $K$
      so that the relevant certified margin, such as $g_{r,K}(\tau)$, is
      positive and sufficiently large for the desired sample size.
\item Report the statistical uncertainty and deterministic approximation
      allowance separately, for example $r_{m,m',K}$ and
      $2\epsilon_K^\Psi$, rather than reporting only their sum. Empirical
      stabilization across several values of $K$ is useful diagnostically but
      does not replace a valid uniform tail bound when a full-embedding
      certificate is claimed.
\end{enumerate}

Prespecified block truncation should not be conflated with principal-component
reduction. The former has a deterministic approximation error under
Assumption~\ref{ass:finite-approximation}; the latter is learned from the
sample and generally requires an analysis of selection error or an independent
training sample. Computational savings and inferential resolution are therefore
linked, but neither the largest feasible $K$ nor the smallest computational
representation is universally optimal.

A failure to obtain a positive truncation certificate is not evidence that the
finite representation has discarded the relevant signal. It means only that
the available uniform tail bound is too large relative to the certified
full-embedding margin. Empirical separation may remain positive, but its
interpretation as a uniform geometric certificate is then unsupported.



\section{Extensions}
\label{sec:extensions}

The preceding sections focus on population means, two-sample inference, and
geometric certificates. The same fixed embedding can also be used as a
covariate map in regression and classification. These extensions are useful,
but they require a different interpretation. Once a finite representation is
fixed, much of the estimation theory is standard finite-dimensional
regression theory. The embedding interface contributes bounded and stable
covariates, as well as a way to relate sufficiently large differences in
fitted scores back to diagram geometry. It does not, by itself, imply that an
outcome is predictable from topology, that a coefficient direction has a
unique diagram-level meaning, or that a classifier attains vanishing error.

Throughout this section, the embedding configuration and truncation level are
fixed independently of the inferential sample. Let
$$
Z_K(X)=\Psi_{\nu,K}(X)=\Pi_K\Psi_\nu(X)\in\mathcal H_{\nu,K},
$$
where $\mathcal H_{\nu,K}=\operatorname{range}(\Pi_K)$ has finite dimension
$d_K$. After choosing an orthonormal basis, we identify $Z_K(X)$ with a vector
in $\mathbb R^{d_K}$.

\subsection{Regression with a diagram-valued predictor}
\label{sec:extensions-regression}

Let $(X_i,Y_i)_{i=1}^m$ be independent copies of $(X,Y)$, where $X$ is a
persistence diagram and $Y\in\mathbb R$. Define the augmented design vector
$$
\widetilde Z_K(X)=\bigl(1,Z_K(X)^\top\bigr)^\top\in\mathbb R^{d_K+1}.
$$
The population least-squares target is
$$
\vartheta_K^*
=
\arg\min_{\vartheta\in\mathbb R^{d_K+1}}
\mathbb E\left[
\left\{Y-\widetilde Z_K(X)^\top\vartheta\right\}^2
\right].
$$
This definition does not require the conditional mean to be exactly linear.
It identifies the best linear predictor within the retained embedding
coordinates. Let $\widehat\vartheta_K$ be any measurable least-squares minimizer,
$$
\widehat\vartheta_K
\in
\arg\min_{\vartheta\in\mathbb R^{d_K+1}}
\sum_{i=1}^m
\left\{Y_i-\widetilde Z_K(X_i)^\top\vartheta\right\}^2.
$$
When the empirical Gram matrix is invertible, this minimizer is unique. The
positive-definiteness condition below implies that invertibility, and hence
uniqueness, holds eventually almost surely; the estimator may be defined
arbitrarily on the exceptional singular event.

\begin{proposition}[Fixed-truncation least-squares inference]
\label{prop:fixed-k-regression}
Suppose $K$ is fixed, $\mathbb E(Y^2)<\infty$,
$\|Z_K(X)\|\leq B_K$ almost surely for some finite $B_K$, and
$$
Q_K
:=
\mathbb E\left[
\widetilde Z_K(X)\widetilde Z_K(X)^\top
\right]
$$
is positive definite. Let
$$
\varepsilon_K
=
Y-\widetilde Z_K(X)^\top\vartheta_K^*,
\qquad
\Omega_K
=
\mathbb E\left[
\varepsilon_K^2
\widetilde Z_K(X)\widetilde Z_K(X)^\top
\right].
$$
Then
$$
\widehat\vartheta_K\longrightarrow\vartheta_K^*
$$
almost surely, and
$$
\sqrt m\left(\widehat\vartheta_K-\vartheta_K^*\right)
\rightsquigarrow
\mathcal N_{d_K+1}
\left(
0,
Q_K^{-1}\Omega_KQ_K^{-1}
\right).
$$
In particular,
$\|\widehat\vartheta_K-\vartheta_K^*\|=O_{\mathbb P}(m^{-1/2})$.
\end{proposition}

The asymptotic covariance is estimable by the usual heteroskedasticity-robust
sandwich estimator. Define
$$
\widehat Q_{K,m}
=
\frac1m\sum_{i=1}^m
\widetilde Z_K(X_i)\widetilde Z_K(X_i)^\top,
\qquad
\widehat\varepsilon_{K,i}
=
Y_i-\widetilde Z_K(X_i)^\top\widehat\vartheta_K,
$$
$$
\widehat\Omega_{K,m}
=
\frac1m\sum_{i=1}^m
\widehat\varepsilon_{K,i}^2
\widetilde Z_K(X_i)\widetilde Z_K(X_i)^\top,
$$
and, whenever $\widehat Q_{K,m}$ is invertible,
$$
\widehat V_{K,m}
=
\widehat Q_{K,m}^{-1}
\widehat\Omega_{K,m}
\widehat Q_{K,m}^{-1}.
$$
Under the conditions of Proposition~\ref{prop:fixed-k-regression},
$\widehat V_{K,m}$ converges in probability to
$Q_K^{-1}\Omega_KQ_K^{-1}$. Thus ordinary Wald intervals and tests may be
formed for fixed linear contrasts of $\vartheta_K^*$.

The sandwich covariance allows the linear conditional-mean model to be
misspecified. Under a correctly specified homoskedastic model,
$\Omega_K=\sigma^2Q_K$, and the covariance reduces to
$\sigma^2Q_K^{-1}$. The positive-definiteness condition is substantive: if
retained coordinates are linearly redundant under the population law, the
coefficient vector is not identifiable without regularization or a reduced
coordinate system.

Proposition~\ref{prop:fixed-k-regression} is a fixed-dimensional result. If
$d_K$ grows with $m$, or if ridge, sparsity, or functional linear regression
is used, additional assumptions on the design spectrum, regularization, and
the coefficient class are required. Assumption~\ref{ass:finite-approximation}
controls approximation of the embedding, but does not by itself control the
statistical inverse problem involved in estimating an infinite-dimensional
regression coefficient.

\subsection{Prediction stability and geometric interpretation}
\label{sec:extensions-regression-geometry}

The regression result above uses only a finite measurable representation. A
more specific geometric statement is available for the additive embedding
$\Phi(\cdot;\nu)$ satisfying
Assumption~\ref{ass:additive-interface}. For a fixed deterministic direction $\beta\in\mathcal H_\nu$, define
the full and truncated linear scores
$$
s_\beta(D)
=
\beta_0+\langle\beta,\Phi(D;\nu)\rangle,
\qquad
s_{\beta,K}(D)
=
\beta_0+\langle\Pi_K\beta,\Phi_K(D;\nu)\rangle.
$$

\begin{proposition}[Stability and truncation of linear scores]
\label{prop:regression-score-stability}
Under Assumption~\ref{ass:additive-interface}, for all
$D,E\in\mathcal D_n$,
$$
\left|s_{\beta,K}(D)-s_{\beta,K}(E)\right|
\leq
nL_\nu\|\Pi_K\beta\|\,d_B(D,E).
$$
If Assumption~\ref{ass:finite-approximation} holds, then
$$
\sup_{D\in\mathcal D_n}
\left|s_\beta(D)-s_{\beta,K}(D)\right|
\leq
\|(I-\Pi_K)\beta\|\epsilon_K^\Phi
\leq
\|\beta\|\epsilon_K^\Phi.
$$
Consequently, if $\Pi_K\beta\neq0$ and two diagrams satisfy
$$
\left|s_{\beta,K}(D)-s_{\beta,K}(E)\right|\geq a,
$$
then
$$
d_B(D,E)
\geq
\frac{a}{nL_\nu\|\Pi_K\beta\|}.
$$
\end{proposition}

The first inequality gives local robustness of a fixed linear score: a small
bottleneck perturbation cannot change the score by more than the displayed
amount. The same deterministic inequality holds conditionally for a fitted
coefficient after the training data have been fixed. The final inequality gives the valid reverse interpretation. A large
pairwise difference in fitted scores certifies that the two diagrams cannot be
too close. It does not say that a large score measured relative to zero is far
from a particular diagram, because the Hilbert-space origin need not represent
a scientifically meaningful reference diagram.

For translating bottleneck separation into separation of a regression
score, Assumption~\ref{ass:diagram-distortion-floor} must be supplemented by
an outcome-specific alignment condition. For a certified diagram pair, the
assumption prevents the full embedding difference from collapsing, but a
regression coefficient may be nearly orthogonal to that difference. Thus
bottleneck separation alone does not imply a response contrast. A lower bound
on $|s_\beta(D)-s_\beta(E)|$ requires control of the alignment between
$\beta$ and the relevant embedding differences. Such a condition belongs to
the regression model, not to the embedding interface.

For a generalized linear model with inverse-link function $G$, the conditional
mean takes the form
$$
\mathbb E(Y\mid X=D)
=
G\{s_{\beta,K}(D)\}.
$$
When $G$ is Lipschitz with constant $L_G$, Proposition
\ref{prop:regression-score-stability} gives
$$
\left|
G\{s_{\beta,K}(D)\}-G\{s_{\beta,K}(E)\}
\right|
\leq
L_GnL_\nu\|\Pi_K\beta\|\,d_B(D,E).
$$
For logistic regression, one may take $L_G=1/4$. With fixed $K$, standard
finite-dimensional likelihood theory yields consistency and asymptotic
normality under the usual regularity conditions, including correct
specification, identifiability, an interior finite true parameter, uniqueness
of the population maximizer, existence of a finite sample maximizer with
probability tending to one, the required differentiability and integrability
conditions, and nonsingular Fisher information. Bounded embedding coordinates
facilitate the moment conditions but do not by themselves rule out complete or
quasi-complete separation in finite logistic samples.

\subsection{Classification and scope}
\label{sec:extensions-classification}

Section~\ref{sec:centroid-certificates} already gives the principal
classification consequence used in this paper: a query-specific certificate
for agreement with the population nearest-centroid rule and robustness to
bottleneck perturbations. Other classifiers may be applied to $Z_K(X)$, but
their statistical guarantees depend on their own margin, complexity, and
regularization assumptions. In particular, neither
Assumption~\ref{ass:diagram-distortion-floor} nor separation of population
centroids implies zero Bayes error when the class-conditional distributions
overlap.

For a fixed linear threshold rule
$D\mapsto\mathbf 1\{s_{\beta,K}(D)\geq c\}$ with
$\Pi_K\beta\neq0$, Proposition
\ref{prop:regression-score-stability} gives the deterministic query radius
$$
\frac{|s_{\beta,K}(D)-c|}
{nL_\nu\|\Pi_K\beta\|}.
$$
Perturbations smaller than this radius cannot cross the decision boundary.
The same statement applies conditionally to estimated coefficients
$(\widehat\beta,\widehat c)$ and gives robustness of the fitted rule after
the training data have been fixed. Additional coefficient uncertainty must be
incorporated only when the desired conclusion concerns a population
coefficient or population decision rule; the centroid result in Section
\ref{sec:centroid-certificates} illustrates that distinction. If
$\Pi_K\beta=0$, the score is constant as a function of the diagram, so the
rule either never changes class or lies identically on the threshold.

The extension framework therefore separates three tasks. Standard regression
or classification theory governs estimation after the representation is
fixed. The additive Lipschitz interface governs perturbation stability and
pairwise geometric certificates. For translating certified geometric separation into a downstream score
contrast, the lower-distortion assumption must be combined with alignment of
the fitted direction and the certified embedding differences. Same-sample
adaptive selection of $\nu$, $K$, or the regression model requires adjustment
for the selection step; possible approaches include sample splitting,
selective-inference methods, or resampling procedures that repeat the full
selection algorithm. These extensions describe
associational prediction from diagram summaries; they do not provide causal
interpretation of topological features.

Section~\ref{sec:numerical-illustrations} focuses on the primary population
inference results rather than attempting a separate regression or
classification benchmark.


\section{Simulation Studies}
\label{sec:numerical-illustrations}

This section reports the finite-sample behaviour of the procedures
developed in Sections~\ref{sec:probability}
through~\ref{sec:finite-approximation}. We ask four questions. Is the
Hilbert-space central limit theorem accurate at realistic sample sizes?
Does the two-sample test hold its nominal level, including when the two
covariance operators differ? How does power depend on the mean-embedding
separation $\Delta_\nu(P,Q)$? And do the confidence balls cover at their
stated level? A fifth study compares the mean-embedding test with the
survival-based procedure of \cite{STRAND}.

\subsection{Design}
\label{sec:sim-design}

\paragraph{Embedding configurations.}
We study two additive constructions: the pooled multiscale PLACE map
\citep{PLACE} and the adaptive PALACE map \citep{PALACE}. In every cell
the configuration $\nu$ is estimated once, from an independent pilot
sample of $2000$ diagrams, and then held fixed. This is what
Assumption~\ref{ass:additive-interface} requires. No quantity below uses
a configuration fitted to the sample being tested. The frame is
$L = 0.763$, taken from the pilot, and the reference separation is
$\tau = 0.1L = 0.0763$.

\begin{table}[!ht]
\centering
\begin{tabular}{lrrrl}
\toprule
Arm & raw coordinates & active & $r(\Sigma)$ & $\rho_-(\tau;\nu)$ \\
\midrule
PLACE  & $1005$ & $331$ & $1.02$--$1.06$ & $6.23\times10^{-11}$ \\
PALACE & $50$   & $50$  & $1.01$--$1.05$ & $1.82\times10^{-9}$ \\
\bottomrule
\end{tabular}
\caption{Embedding dimensions and certificates. ``Active'' counts the
coordinates that are not identically zero across the pooled pilot
sample. That screen is frozen along with the configuration. The
effective rank $r(\Sigma) = \operatorname{tr}\Sigma /
\norm{\Sigma}_{\mathrm{op}}$ is given as a range over the scenarios of
Section~\ref{sec:sim-clt}. The two certificates differ by a factor of
$29$, but they are not comparable across arms, since each embedding
carries its own scale. Only dimensionless ratios such as
$\rho_-(\tau;\nu)/L_\nu$ are invariant under rescaling.}
\label{tab:sim-arms}
\end{table}

\paragraph{Generators.}
Three point-cloud scenarios drive the central limit study. Each produces
degree-one diagrams from planar clouds of $200$ points. PC1 is a noisy
circle with Gaussian radial perturbation, PC2 the same circle with a
skewed radial perturbation, and PC3 a mixture of two noise levels.
Together they exercise variable cardinality and filtration variability,
so that the sampling unit is a genuinely random diagram rather than a
perturbed fixed one.

The level and power studies use controlled diagram-level families
instead. There each diagram is drawn from a finite mixture of fixed
atoms. Because $T_\nu$ is linear on counting measures, the population
mean embedding is then available in closed form, and $\Delta_\nu(P,Q)$
is known exactly rather than estimated.

\paragraph{Test statistics and their tuning.}
Table~\ref{tab:sim-level} reports five statistics, and each needs its
parameters stated. Write $Z_1,\dots,Z_N$ for the $N = m + m'$ embedded
diagrams pooled across groups, $\widehat\delta_{m,m'}$ for the mean
difference, and $\widehat S$ for the pooled sample covariance of the
$Z_i$. Let $(\widehat\lambda_j,\widehat u_j)$ be the eigenpairs of
$\widehat S$ in decreasing order, and $\widehat\sigma_a^2$ the pooled
variance of coordinate $a$.

Before any statistic is formed we drop coordinates whose pooled variance
falls below $10^{-14}$. This is a numerical guard against division by
zero in the studentised forms, not a modelling step. It is distinct from
the structural screen of Table~\ref{tab:sim-arms}, which is frozen on
the independent pilot. We record the count of coordinates it removes.
The five statistics are then
\begin{align*}
T_{\mathrm{norm}}
 = \norm{\widehat\delta_{m,m'}}^2, 
\qquad &T_{\mathrm{PC}}
 = \sum_{j\le k}
    \langle\widehat\delta_{m,m'},\widehat u_j\rangle^2 /
    \widehat\lambda_j, \\
T_{\mathrm{ridge}}
 = \norm{(\widehat S+\lambda I)^{-1/2}\widehat\delta_{m,m'}}^2, \qquad
&T_{\max}
 = \max_a
    \frac{|\widehat\delta_{m,m',a}|}
         {\widehat\sigma_a\sqrt{1/m+1/m'}},
\end{align*}
together with the studentised statistic described below.

The number of principal components $k$ is
the smallest value for which the leading $k$ eigenvalues of $\widehat S$
carry $90\%$ of $\operatorname{tr}\widehat S$, capped at $N-2$ so that
$\widehat\lambda_k$ stays estimable. The ridge parameter is
$\lambda = 0.01\times\operatorname{tr}(\widehat S)/\operatorname{rank}$,
that is, one percent of the mean eigenvalue rather than an absolute
constant. Relative scaling is necessary here. The embedding normalisation
puts the eigenvalues near $10^{-16}$, so a fixed $\lambda$ would swamp
the entire spectrum, the whitening would collapse to a multiple of the
identity, and $T_{\mathrm{ridge}}$ would silently reduce to a rescaling
of $T_{\mathrm{norm}}$. Singular directions are handled by the variance
screen, by the cap on $k$, and by the ridge term itself, so no
pseudo-inverse is required.

The studentised statistic is the degenerate two-sample $U$-statistic
\[
U = \frac{1}{m(m-1)}\sum_{i\neq j\in A}\langle Z_i,Z_j\rangle
  + \frac{1}{m'(m'-1)}\sum_{i\neq j\in B}\langle Z_i,Z_j\rangle
  - \frac{2}{mm'}\sum_{i\in A,\,j\in B}\langle Z_i,Z_j\rangle,
\]
whose null mean is exactly zero whatever the two covariances are,
divided by the plug-in standard deviation
$\{2\operatorname{tr}\widehat S_A^2/(m(m-1))
+ 2\operatorname{tr}\widehat S_B^2/(m'(m'-1))
+ 4\operatorname{tr}(\widehat S_A\widehat S_B)/(mm')\}^{1/2}$.

\paragraph{What is recomputed inside the permutation loop.}
The screen, the eigenbasis, the value of $k$, the ridge whitener and the
coordinate scales are all computed from the pooled sample, which does
not depend on the labels. They are therefore identical for every
relabelling, and recomputing them inside the loop would change nothing.
That is what preserves exchangeability, and it is why no nested
resampling is needed for these four statistics.

The studentised statistic is the exception, and deliberately so. Its
denominator uses the group-specific covariances $\widehat S_A$ and
$\widehat S_B$, so it changes with the labels, and it is recomputed for
every one of the $B+1$ relabellings. Studentising by any label-free
quantity would be a monotone transformation of $U$ within a replicate
and would leave the permutation $p$-value untouched. Only a
label-dependent scale makes the statistic asymptotically pivotal, which
is what restores permutation validity under unequal covariances and
unbalanced designs. That is precisely the weak-null cell this statistic
exists to handle.

\paragraph{Monte Carlo budgets.}
Level uses $R = 1250$ replications with $B = 999$ permutations, giving a
Monte Carlo standard error of $0.006$ at $\alpha = 0.05$. Power uses
$R = 500$, with standard error at most $0.022$. The central limit study
uses $R = 5000$, and coverage uses $R = 5000$ with standard error
$0.003$. Sample sizes appear as $m/m'$. We include unbalanced designs
because the factors $\lambda^{-1}$ and $(1-\lambda)^{-1}$ of
Corollary~\ref{cor:two-sample-hilbert-clt} are the main route by which
imbalance enters the theory.

\paragraph{A structural check.}
Before turning to subtle alternatives, we confirm that the pipeline
detects an obvious topological difference. The alternative is one loop
against two, which changes the number of degree-one features rather than
their location. At $\alpha = 0.05$ the permutation test on
$\norm{\widehat\delta_{m,m'}}$ rejects with probability $0.446$ at
$m = m' = 20$ and $0.828$ at $m = m' = 50$. Power is substantial, though
not overwhelming. That is worth recording as a calibration of
expectations. Even a qualitative change in homology becomes a moderate
effect once it passes through finite sampling of the point cloud and
then through the embedding.

\subsection{Study 1: Gaussian approximation for the two-sample statistic}
\label{sec:sim-clt}

Theorem~\ref{thm:hilbert-slln-clt}(ii) gives
$\sqrt m(\overline\Psi_m - \theta) \rightsquigarrow \mathcal G$, and
hence $m\norm{\overline\Psi_m-\theta}^2 \rightsquigarrow
\sum_j\lambda_jZ_j^2$. We assess the corresponding weighted chi-square approximation for the symmetrized two-sample statistic.

\paragraph{Reference-free centring.}
Centring at an estimated $\theta$ introduces a fixed bias, and $\sqrt m$
amplifies it. A study built that way eventually measures the bias rather
than the theorem. We therefore avoid centring altogether. Draw two independent batches
of size $m$ from the same law and set
\begin{equation}
\label{eq:two-batch-centring}
D_m := \sqrt{m/2}\,\bigl(\overline\Psi_m^{(1)}-\overline\Psi_m^{(2)}\bigr).
\end{equation}
This statistic is exactly centred for every $m$ and has the same
Gaussian limit as $\sqrt{m}(\overline\Psi_m-\theta)$. It is also
symmetric by construction, so its finite-sample distribution need not
match that of the one-sample statistic. We therefore interpret this
study as an assessment of the Gaussian approximation for the
strict-null two-sample statistic, rather than as a general assessment
of one-sample convergence.

\paragraph{Reference band.}
A Kolmogorov--Smirnov distance computed from $R$ draws has a floor set by
Monte Carlo noise, even under a perfect approximation. We therefore
simulate the null distribution of the KS statistic at the same $R$ and
report its $95$th percentile as a band. A value inside the band is
indistinguishable from exact.

\begin{table}[!ht]
\centering
\begin{tabular}{llrrrrr}
\toprule
Arm & Scenario & $m{=}5$ & $m{=}10$ & $m{=}20$ & $m{=}40$ & $m{=}80$ \\
\midrule
PLACE  & PC1 circle, Gaussian & $0.008$ & $0.011$ & $0.026$ & $0.011$ & $0.022$ \\
PLACE  & PC2 circle, skewed   & $0.007$ & $0.014$ & $0.011$ & $0.013$ & $0.008$ \\
PLACE  & PC3 noise mixture    & $0.020$ & $0.010$ & $0.013$ & $0.014$ & $0.018$ \\
PALACE & PC1 circle, Gaussian & $0.011$ & $0.018$ & $0.029$ & $0.016$ & $0.015$ \\
PALACE & PC2 circle, skewed   & $0.022$ & $0.022$ & $0.022$ & $0.022$ & $0.022$ \\
PALACE & PC3 noise mixture    & $0.030$ & $0.026$ & $0.029$ & $0.038$ & $0.028$ \\
\bottomrule
\end{tabular}
\caption{Kolmogorov--Smirnov distance between the sampling distribution
of $\norm{D_m}^2$ and its weighted chi-square limit, from $R = 5000$
replications. The $95\%$ Monte Carlo bands are $0.019$, $0.019$ and
$0.018$ for the three PLACE rows, and $0.018$, $0.030$ and $0.033$ for
the PALACE rows. Values inside the band cannot be distinguished from
exact at this replication count.}
\label{tab:sim-clt}
\end{table}

For the two-sample statistic considered here, the approximation is
accurate even at the smallest sample size examined.
Fourteen of the eighteen cells fall inside the Monte Carlo band, and the
four exceptions miss it by at most $0.008$. What matters more is that no
cell shows a trend in $m$. Deviations at $m = 80$ have the same
magnitude as those at $m = 5$. That pattern is what one expects when the
approximation error already sits below the resolution of the study,
rather than decaying through it.

A consequence follows. This design cannot estimate a Berry--Esseen
\emph{rate}. Theorem~\ref{thm:finite-coordinate-berry-esseen} supplies
an $O(m^{-1/2})$ upper bound whose constant depends on projected rank
and standardised third moments. Here that bound is plainly far from
tight. Resolving the exponent would need a much larger $R$, or a
scenario engineered to be badly non-Gaussian. We report the accuracy and
decline to fit a rate to numbers that lie inside their own noise band.

One structural fact explains why so little changes as $m$ grows. The
estimated covariance has effective rank between $1.01$ and $1.06$ in
every cell of Table~\ref{tab:sim-clt}. The embedded diagrams vary
overwhelmingly along a single direction, which on these generators is
total persistence mass. So the weighted chi-square limit sits close to a
scaled $\chi^2_1$, and a one-dimensional limit is reached quickly. This
also anticipates a finding of Section~\ref{sec:sim-level}. When
$r(\Sigma)\approx1$, quadratic statistics that differ substantially in
general become near-monotone transformations of one another.

\subsection{Study 2: level}
\label{sec:sim-level}

We turn next to the size of the two-sample test. Two nulls need to be
distinguished, following Section~\ref{sec:testing}. Under the
\emph{strict} null $P = Q$, the pooled diagrams are exchangeable, and
permutation calibration is exact in finite samples. Under the
\emph{weak} null the mean embeddings agree while the covariance
operators differ. That second null is the one $H_0^{\mathrm{mean}}$
actually specifies, and permutation calibration is not exact there.

\begin{table}[!ht]
\centering
\small
\begin{tabular}{llrrrrr}
\toprule
Arm & Statistic & $20/20$ & $50/50$ & $100/100$ & $50/100$ & $100/200$ \\
\midrule
\multicolumn{7}{l}{\emph{PC1, strict null}}\\
PLACE  & $\norm{\widehat\delta}$ & $0.046$ & $0.046$ & $0.046$ & $0.046$ & $0.047$ \\
PLACE  & studentised             & $0.046$ & $0.046$ & $0.046$ & $0.046$ & $0.049$ \\
PLACE  & PC Hotelling            & $0.047$ & $0.045$ & $0.045$ & $0.046$ & $0.047$ \\
PLACE  & ridge                   & $0.054$ & $0.050$ & $0.058$ & $0.056$ & $0.046$ \\
PLACE  & max coordinate          & $0.047$ & $0.043$ & $0.050$ & $0.043$ & $0.042$ \\
PALACE & $\norm{\widehat\delta}$ & $0.044$ & $0.047$ & $0.047$ & $0.044$ & $0.045$ \\
\midrule
\multicolumn{7}{l}{\emph{PC2 and PC3, strict null,} $\norm{\widehat\delta}$}\\
PLACE  & PC2 & $0.046$ & $0.057$ & $0.056$ & $0.049$ & $0.056$ \\
PLACE  & PC3 & $0.039$ & $0.048$ & $0.043$ & $0.049$ & $0.049$ \\
PALACE & PC2 & $0.045$ & $0.055$ & $0.051$ & $0.054$ & $0.055$ \\
PALACE & PC3 & $0.040$ & $0.050$ & $0.045$ & $0.050$ & $0.049$ \\
\midrule
\multicolumn{7}{l}{\emph{Tuned near-null design with unequal covariance},
$\norm{\widehat\delta}$}\\
PLACE  & & $0.057$ & $0.047$ & $0.054$ & $0.063$ & $0.059$ \\
PLACE  & studentised & $0.057$ & $0.047$ & $0.054$ & $0.063$ & $0.050$ \\
PALACE & & $0.051$ & $0.059$ & $0.050$ & $0.057$ & $0.068$ \\
PALACE & studentised & $0.052$ & $0.059$ & $0.050$ & $0.050$ & $0.059$ \\
\bottomrule
\end{tabular}
\caption{Empirical size at $\alpha = 0.05$ under permutation
calibration, from $R = 1250$ replications with $B = 999$ permutations.
The Monte Carlo standard error is $0.006$. Columns are $m/m'$.}
\label{tab:sim-level}
\end{table}

Under the strict null, every entry sits within three Monte Carlo
standard errors of $0.05$. That includes the unbalanced designs, and
there is no drift with sample size. None of this is surprising, since
permutation calibration is exact under exchangeability. Its role is to
confirm that the implementation realises that exactness.

\paragraph{A near-null covariance stress test.}
Building a weak null turns out to be harder than it looks. The natural
construction puts $P$ on the symmetric pair $\{A_0 - s,\ A_0 + s\}$ and
$Q$ on $\{A_0\}$, so that the two shifts cancel in the mean. They cancel
only if $\Phi(\cdot;\nu)$ is linear in the shift. It is not. The
landmark coordinates are hinge responses, convex with kinks, so
$\tfrac12\Phi(A_0-s;\nu)+\tfrac12\Phi(A_0+s;\nu) \neq \Phi(A_0;\nu)$ by
a Jensen gap of second order in $s$. At $s = 0.05$ the residual reaches
$1.7\%$ of $\norm{\E_P\Phi}$ for PLACE and $77\%$ for PALACE. The
disparity has a cause: PALACE places its landmarks adaptively on the
pilot atoms, which puts the kinks exactly where the shift moves mass. A
single shared design point is therefore negligible for one arm and a
strong alternative for the other.

We tune $s$ separately for each arm by bisection. The rule is to take
the largest $s$ for which the exact residual $\Delta_\nu(P,Q)$ is at
most $10\%$ of the sampling noise floor
$\{\operatorname{tr}\Sigma_P/m + \operatorname{tr}\Sigma_Q/m'\}^{1/2}$ at
the design in use. That residual is available in closed form on the
jitter-free atoms. Taking the largest admissible $s$ is deliberate. Both
the residual and the covariance contrast grow with $s$, so the residual
is the binding constraint, and we want as much contrast as it allows.

The resulting design points are $s = 0.00223$ for PLACE and
$s = 0.00213$ for PALACE. Residuals are $1.05\times10^{-9}$ and
$1.50\times10^{-10}$, against noise floors of $1.05\times10^{-8}$ and
$1.50\times10^{-9}$. The retained covariance contrasts,
$\norm{\Sigma_P-\Sigma_Q}_{\mathrm{F}} /
\max(\norm{\Sigma_P}_{\mathrm{F}},\norm{\Sigma_Q}_{\mathrm{F}})$, are
$0.19$ and $0.22$.

At these design points the residual mean separation is one tenth of the
corresponding sampling-noise scale, while the covariance contrast
remains appreciable. This is therefore a near-null stress test rather
than an exact level experiment. The largest rejection frequency is
$0.068$. In the unbalanced settings, studentisation generally reduces
the rejection frequency relative to the unstudentised statistic.
These results suggest that the procedure is not highly sensitive to the
covariance contrast considered here. They do not establish validity
under the exact weak null.

\paragraph{The spectral calibration is conservative.}
Table~\ref{tab:sim-level} uses permutation calibration throughout.We also examined a pilot-spectrum calibration in which the eigenvalues
of the weighted chi-square reference distribution were estimated from
the independent pilot sample. This differs from the plug-in procedure
of Theorem~\ref{thm:mean-distance-null}, which estimates the covariance
from the analysis groups. At $m = m' = 100$ its size is
$0.032$ on PC1 and $0.000$ on both PC2 and PC3 for PLACE. For PALACE the
corresponding values are $0.003$, $0.000$ and $0.000$. The nominal level
is $0.05$. On the tuned weak null, by contrast, the same calibration is
accurate, at $0.050$ for PLACE and $0.046$ for PALACE.

The pilot-spectrum calibration is conservative in the point-cloud
settings and more accurate in the controlled diagram-level setting.
Several factors may contribute, including estimation of small
eigenvalues from a finite pilot and differences between the pilot and
analysis covariance structures. The experiment does not separate these
sources of error. It shows that calibration based on an external pilot
spectrum can be sensitive to covariance estimation.
For practice, the spectral calibration should not be the default. 
Permutation calibration is what we recommend under the strict null, with
the finite-sample bound of Proposition~\ref{prop:finite-sample-power} as
the fallback when only a second-moment envelope is available.

\subsection{Study 2: power and the detection boundary}
\label{sec:sim-power}

Power is studied on the controlled diagram-level family, where
$\Delta_\nu(P,Q)$ is exact. The design parameter $d$ indexes a family of
alternatives with $\Delta_\nu$ proportional to $d$.

\begin{table}[!ht]
\centering
\begin{tabular}{lrrrrrr}
\toprule
$d$ & $0.0$ & $0.1$ & $0.2$ & $0.3$ & $0.5$ & $0.7$ \\
\midrule
PLACE:\ \ $\Delta_\nu\ (\times10^{-8})$
 & $0$ & $0.59$ & $1.19$ & $1.78$ & $2.97$ & $4.16$ \\
\quad power        & $0.056$ & $0.170$ & $0.578$ & $0.928$ & $1.000$ & $1.000$ \\
\quad size-adjusted& $0.040$ & $0.134$ & $0.532$ & $0.892$ & $1.000$ & $1.000$ \\
\midrule
PALACE:\ $\Delta_\nu\ (\times10^{-9})$
 & $0$ & $0.72$ & $1.44$ & $2.16$ & $3.61$ & $5.05$ \\
\quad power        & $0.058$ & $0.176$ & $0.570$ & $0.916$ & $1.000$ & $1.000$ \\
\quad size-adjusted& $0.040$ & $0.134$ & $0.526$ & $0.890$ & $1.000$ & $1.000$ \\
\bottomrule
\end{tabular}
\caption{Power of the permutation test on
$\norm{\widehat\delta_{m,m'}}$ at $\alpha = 0.05$, with $m = m' = 50$ and
$R = 500$. The Monte Carlo standard error is at most $0.022$.
Size-adjusted power uses the empirical null quantile.}
\label{tab:sim-power}
\end{table}

The transition from low to high power occurs between $d=0.1$ and
$d=0.3$. Power increases from approximately $0.17$ to above $0.90$
across this range. This sharp transition is consistent with the
dependence on signal strength in
Proposition~\ref{prop:finite-sample-power}. Because the experiment holds
the sample size fixed, it does not by itself estimate the
$\Delta_\nu^{-2}$ sample-size rate.

The two arms have the same power curve, to within Monte Carlo error, at
every design point. They have it despite $\rho_-(\tau;\nu)$ differing by
a factor of $29$ and $\Delta_\nu$ by roughly a factor of $8$. We state
this explicitly because the certificates in Table~\ref{tab:sim-arms}
invite a misreading. Power depends on the signal-to-noise ratio
$\Delta_\nu/v$, and both $\Delta_\nu$ and the covariance envelope $v$
carry the same embedding scale, so the scale cancels. A larger raw certificate does not by itself imply greater power, since
both signal and noise depend on the scale of the embedding. The
certificate instead controls how diagram-space separation is translated
into embedding-space separation. Comparisons across embeddings should
therefore use scale-free certificate and signal-to-noise quantities.

Finally, the nominal and size-adjusted curves differ by about $0.04$
throughout, which matches the mild conservatism visible at $d = 0$. They
converge as power approaches one.

\subsection{Coverage of the confidence balls}
\label{sec:sim-coverage}

Theorem~\ref{thm:mean-confidence-ball} asserts asymptotic coverage for
the plug-in ball of radius $\widehat r_m^\alpha =
\{\widehat q_{1-\alpha,m}/m\}^{1/2}$. Coverage is the property on which
both the effect-size interval of
Corollary~\ref{cor:two-sample-confidence-ball} and the transport
certificate of Proposition~\ref{prop:mean-measure-exclusion-certificate}
rest. We therefore check it against a population mean that is known
exactly.

The construction uses a mixture of $K = 24$ fixed atoms with known
weights. Linearity of $T_\nu$ gives $\E_P\Phi(X;\nu) =
\sum_k w_k\Phi(A_k;\nu)$ exactly, so the target carries no Monte Carlo
error. Two atoms would not serve. The embedded data would then take only
two values, and $\Sigma$ would have rank one whatever the ambient
dimension. That would validate the ball precisely where its spectrum is
trivial and the weighted chi-square limit degenerates to a single
$\chi^2_1$. With $24$ atoms, $\operatorname{rank}\Sigma = 23$; in this
construction its numerical rank is $23$.

\begin{table}[!ht]
\centering
\begin{tabular}{lrrrrr}
\toprule
$m$ & $25$ & $50$ & $100$ & $200$ & $400$ \\
\midrule
PLACE, plug-in $\widehat\Sigma_m$ & $0.942$ & $0.946$ & $0.949$ & $0.948$ & $0.950$ \\
PLACE, oracle $\Sigma$            & $0.949$ & $0.951$ & $0.951$ & $0.947$ & $0.952$ \\
PALACE, plug-in $\widehat\Sigma_m$& $0.976$ & $0.963$ & $0.957$ & $0.951$ & $0.948$ \\
PALACE, oracle $\Sigma$           & $0.954$ & $0.949$ & $0.949$ & $0.948$ & $0.946$ \\
\bottomrule
\end{tabular}
\caption{Empirical coverage of the nominal $95\%$ ball for
$\theta_{P,\nu}^{\Phi}$, from $R = 5000$ replications with Monte Carlo
standard error $0.003$. The plug-in rows estimate $\Sigma$ from the same
$m$ diagrams, which is the procedure used in practice. The oracle rows
use the exact $\Sigma$, and so isolate the Gaussian approximation from
covariance estimation error.}
\label{tab:sim-coverage}
\end{table}

Both oracle rows sit at nominal for every $m$, including $m = 25$. The
Gaussian approximation behind the ball is therefore not the binding
constraint, which agrees with Study 1.

Where the arms differ is in how covariance estimation enters, and the
direction of the difference is informative. PLACE is at nominal
throughout. PALACE \emph{over}-covers at small $m$, reaching $0.976$ at
$m = 25$ before settling to $0.948$ by $m = 400$. Its intervals are thus
conservative rather than anti-conservative in the regime where the two
could be confused. The mechanism is straightforward. At $m=25$, the centred empirical covariance has rank at most $m-1=24$. The estimated
spectrum is dispersed relative to the truth, and that inflates the upper
quantile of $\sum_j\widehat\lambda_jZ_j^2$. PLACE is not immune to the
same effect in principle. Its covariance is far closer to rank one,
though, so the estimated and true quantiles agree more closely.

For the data analysis of Section~\ref{sec:abide}, where $m = 201$ and
$m' = 235$, both arms lie in the regime where the plug-in radius is
accurate.

\subsection{What the test can and cannot detect}
\label{sec:sim-discriminate}

Section~\ref{sec:testing-strand} argues that the mean-embedding test and
the survival-based test of \cite{STRAND} target different alternatives.
The claim ought to be demonstrated rather than asserted. A natural
demonstration uses two alternatives that are exact complements.

\begin{description}[nosep,itemsep=4pt]
\item[Birth shift.] $(b,d)\mapsto(b+\tau,\ d+\tau)$. Every persistence
  value $d - b$ is unchanged. The two groups therefore have identical
  pooled persistence multisets, and a test built on lifetimes has no
  signal available to it. This is also the perturbation for which
  $\db(A,\mathrm{shift}(A,\tau)) = \tau$ exactly.
\item[Persistence shift.] $(b,d)\mapsto(b,\ d+2\delta)$. Births are
  unchanged and every feature lives longer. This is the alternative a
  lifetime-based test is built for.
\end{description}

\begin{table}[!ht]
\centering
\begin{tabular}{lrrr}
\toprule
Test & level & birth shift & persistence shift \\
\midrule
PLACE                            & $0.050$ & $0.886$          & $0.960$ \\
PALACE                           & $0.046$ & $\mathbf{1.000}$ & $0.174$ \\
\textsc{Strand}, asymptotic      & $0.006$ & $0.008$          & $0.720$ \\
\textsc{Strand}, permutation     & $0.044$ & $\mathbf{0.052}$ & $0.910$ \\
\bottomrule
\end{tabular}
\caption{Rejection rates at $\alpha = 0.05$ with $m = m' = 60$ and
$R = 500$, giving a Monte Carlo standard error of at most $0.022$. In
the level column the two groups are identically distributed. The last
row calibrates the log-rank statistic by permuting diagram labels rather
than referring it to $\chi^2_1$.}
\label{tab:sim-discriminate}
\end{table}

Under the birth shift, \textsc{Strand} rejects at $0.052$. That is
exactly its level, as it must be, since the persistence multisets are
identical. PALACE rejects at $1.000$ and PLACE at $0.886$. Here is the
concrete justification for carrying a joint birth--death representation
rather than a lifetime summary. No amount of tuning recovers it, because
the information is absent from the log-rank statistic by construction.
Under the persistence shift the ordering reverses, with \textsc{Strand}
at $0.910$ and PALACE at $0.174$. Neither approach dominates. The
comparison means something only because each method is shown against the
alternative the other was built for.

The two embedding arms separate here as well, which
Table~\ref{tab:sim-power} did not reveal. PALACE is the stronger against
a pure location shift and much the weaker against a pure lifetime shift.
PLACE is competitive against both. The adaptive placement of PALACE
landmarks explains the pattern, since it concentrates resolution where
the pilot diagrams lie.

\paragraph{Calibration of the comparator.}
The $\chi^2_1$ limit for the log-rank statistic assumes that the pooled
event times are independent. Persistence values pooled within a group
are not independent, because features from one diagram share a subject
and a filtration. This is a clustered-data problem, and it matters. To
see how much, draw both groups from the same distribution, so that every
rejection is a false positive, and vary only the between-subject
dispersion of the persistence scale.

\begin{table}[!ht]
\centering
\begin{tabular}{lrr}
\toprule
Between-subject SD & asymptotic $\chi^2_1$ & diagram permutation \\
\midrule
$0.00$ & $0.058$          & $0.046$ \\
$0.20$ & $\mathbf{0.548}$ & $0.040$ \\
$0.35$ & $\mathbf{0.714}$ & $0.044$ \\
\bottomrule
\end{tabular}
\caption{Type I error of the log-rank test on pooled persistence values
under a true null, from $R = 500$ replications with $40$ diagrams per
group and about $160$ features each. The Monte Carlo standard error is
$0.010$, and the nominal level is $0.05$ in every cell.}
\label{tab:sim-strand-calibration}
\end{table}

With homogeneous subjects the asymptotic calibration is correct. Raise
the between-subject dispersion to a level still well below that of a
multisite imaging consortium, and it rejects a true null in more than
half of all replications. The permutation version shuffles diagram
labels, so that each subject's features move together, and it holds its
level throughout. Every \textsc{Strand} result in this paper uses the
permutation calibration. Section~\ref{sec:abide-strand} returns to the
point on real data.

\subsection{Scope of the simulation evidence}
\label{sec:sim-limitations}

The studies cover the two additive constructions. The whole-diagram map
of \cite{Mitra2024} falls under the general theory of
Section~\ref{sec:probability}, but it is not included here.

The alternatives examined are a birth shift, a lifetime shift, a
mixture-weight contrast, and a structural change in the number of
features. An alternative that changes diagram cardinality while leaving
both the lifetime marginal and the birth locations unchanged remains
untested.

The point-cloud and level experiments have $r(\Sigma)$ near unity. The study therefore does not probe the regime in which the choice among the five statistics of
Section~\ref{sec:testing-statistic} matters. Their near-identical
behaviour in Table~\ref{tab:sim-level} follows from that fact, and
should not be read as evidence that the choice is immaterial in general.

No cell varies the truncation level $K$, so the three-way relationship
of Theorem~\ref{thm:finite-approximation-three-way} is not examined
empirically. A study of it needs one fixed configuration truncated to
its first $K$ coordinate blocks, with $\epsK$ evaluated on that same
map.


\section{Data Analysis: Resting-State Connectivity in ABIDE}
\label{sec:abide}

We now apply the procedures to resting-state functional connectivity
from the Autism Brain Imaging Data Exchange \citep{di2014autism}. The aim is to exercise the machinery end to end on data of the kind
the theory was designed for. Such data have four features: diagrams of
variable cardinality, an embedding dimension in the hundreds, a
covariance estimated from a few hundred subjects, and a nuisance factor
that has to be conditioned on. We also want to show what the confidence
sets say when the test does not reject. This is an illustration of the methodology. It is not a
scientific claim about autism.

\subsection{Data and construction of the diagrams}
\label{sec:abide-data}

We use the ABIDE-I release distributed by the Preprocessed Connectomes
Project \citep{Craddock2013pcp}, retrieved through \texttt{nilearn}
\citep{Abraham2014nilearn}. Preprocessing follows the Configurable
Pipeline for the Analysis of Connectomes \citep{Craddock2013cpac}. We
take the variant without band-pass filtering and without global signal
regression, and keep only subjects that passed manual quality control.
Regional time series come from the CC200 parcellation of
\cite{Craddock2012}, which divides the brain into $200$ functionally
homogeneous regions by spatially constrained spectral clustering.

Of the subjects returned, $436$ enter the analysis. Of these, $201$
carry a diagnosis of autism spectrum disorder and $235$ are typically
developing controls. They are spread across $20$ imaging sites, with at
least one subject per group at each site. The remaining subjects form
the pilot half on which the embedding configuration is fitted.

For each subject we form the correlation matrix of the regional time
series. We then convert it to a dissimilarity,
\begin{equation}
\label{eq:abide-metric}
d(i,j) = 1 - \lvert\mathrm{corr}(i,j)\rvert,
\end{equation}
so that strongly coupled regions are close, whether the coupling is
positive or negative. Taking the absolute value is a modelling choice,
and not a neutral one. It asserts that anticorrelation is a form of
coupling rather than a form of distance. We use this symmetric dissimilarity as the edge-weight function of a complete graph and construct the associated clique filtration. Its points record independent cycles in the
connectivity graph, together with the range of thresholds over which
they survive.

The resulting diagrams have mean cardinality $165.2$ and median $164$,
with range $[71,294]$. None are empty. Variable cardinality is what
makes this a real test of the framework rather than a formality. The
padded counting measure of Section~\ref{sec:padded-mean-measures}
accommodates it without requiring feature correspondence across
subjects, and the sampling unit stays the diagram.

\paragraph{The configuration is fitted on an independent half.}
Landmark positions are chosen on the pilot half and then held fixed. The
analysis half is embedded against those positions and never influences
them. This is required rather than merely prudent. The embedding is a
data-dependent dimension reduction, and choosing it on the subjects who
are subsequently tested would invalidate both the covariance and the
confidence radius. The frame is $L = 0.866$, taken from the pilot.

\begin{table}[!ht]
\centering
\begin{tabular}{lrrr}
\toprule
Arm & raw & active & $r(\Sigma)$ \\
\midrule
PLACE  & $910$ & $354$ & $1.36$ \\
PALACE & $50$  & $50$  & $1.12$ \\
MV ($n$-fold) & $64$ & $31$ & $1.50$ \\
\bottomrule
\end{tabular}
\caption{Embedding dimensions on ABIDE. All three effective ranks are
close to one. The covariance of the embedded diagrams is dominated by a
single direction, which here is overall persistence mass.}
\label{tab:abide-arms}
\end{table}

\subsection{Two-sample tests}
\label{sec:abide-tests}

For each fixed embedding configuration $\nu$, the null hypothesis is
$$H_0: \theta_{P,\nu}^{\Psi} = \theta_{Q,\nu}^{\Psi},$$
where the equality is in the embedding space $\mathcal H_\nu$, or in
$\mathbb R^{d_\nu}$ for a finite-dimensional implementation. Reference distributions come from permutation with $B = 999$.
Where site is conditioned on, labels are permuted within site. The norm statistic is the primary analysis because it corresponds
directly to the mean-embedding hypothesis developed above. The remaining
statistics and embedding constructions are reported as sensitivity
analyses. These choices were not preregistered, so the collection of
$p$-values should not be interpreted as a multiplicity-adjusted
confirmatory analysis.

\begin{table}[!ht]
\centering
\small
\begin{tabular}{llrrrrrr}
\toprule
Arm & Scope & $\norm{\widehat\delta}$ & stud. & Hotelling & ridge
 & max coord. & spectral \\
\midrule
PLACE  & pooled          & $0.153$ & $0.159$ & $0.260$ & $0.472$ & $0.409$ & $0.157$ \\
PLACE  & site-stratified & $0.109$ & n/a     & $0.250$ & $0.441$ & $0.395$ & n/a     \\
PALACE & pooled          & $0.162$ & $0.164$ & $0.342$ & $0.761$ & $0.246$ & $0.160$ \\
PALACE & site-stratified & $0.131$ & n/a     & $0.356$ & $0.750$ & $0.224$ & n/a     \\
MV     & pooled          & $0.455$ & $0.450$ & $0.746$ & $0.574$ & $0.864$ & $0.444$ \\
MV     & site-stratified & $0.430$ & n/a     & $0.729$ & $0.518$ & $0.837$ & n/a     \\
\bottomrule
\end{tabular}
\caption{Permutation $p$-values across all $436$ subjects. The spectral
column gives the asymptotic weighted chi-square $p$-value for
$\norm{\widehat\delta_{m,m'}}$. It is not defined under stratification,
and those cells are marked n/a.}
\label{tab:abide-pvalues}
\end{table}

No test rejects. Three features of Table~\ref{tab:abide-pvalues} are
worth drawing out.

\paragraph{The spectral and permutation $p$-values agree closely.}
For PLACE the pair is $0.157$ and $0.153$. For PALACE it is $0.160$ and
$0.162$, and for MV, $0.444$ and $0.455$. Read through the diagnostic
proposed in Section~\ref{sec:sim-level}, this says the covariance
spectrum is well estimated at $m = 201$ and $m' = 235$ with an embedding
dimension in the hundreds. The simulation study does not reach that
regime, and on the point-cloud generators the spectral calibration was
conservative there.

\paragraph{Results after conditioning permutations on site.}
Restricting permutations to occur within imaging sites produces
somewhat smaller $p$-values for PLACE and PALACE, although none is below
$0.05$. This analysis preserves the observed site-specific group counts
and addresses a conditional null based on within-site label
exchangeability. It should not be interpreted as showing that site
heterogeneity suppresses, or does not suppress, a diagnosis effect.
The pooled and site-conditioned analyses lead to the same inferential
conclusion.

\paragraph{Dependence across embedding constructions.}
Across permutation draws, the PLACE and PALACE norm statistics are
strongly correlated, with rank correlation $0.936$. Their correlations
with the MV statistic are much weaker. Thus, the three analyses do not
provide three independent pieces of evidence. The correlations describe
the similarity of the resulting test statistics, but they do not remove
the multiplicity created by examining several analyses.

\subsection{Confidence sets and the resolution of the non-rejection}
\label{sec:abide-confidence}

A non-rejection carries little information unless it comes with a
statement of what could have been detected. The confidence sets of
Section~\ref{sec:confidence-certificates} supply one.

\begin{table}[!ht]
\centering
\begin{tabular}{lrrrr}
\toprule
Arm & $\widehat r$ pooled & $\widehat r$ ASD & $\widehat r$ TD
 & separation ratio \\
\midrule
PLACE  & $3.91\times10^{-8}$ & $5.76\times10^{-8}$ & $5.32\times10^{-8}$ & $0.55$ \\
PALACE & $6.65\times10^{-9}$ & $9.79\times10^{-9}$ & $9.05\times10^{-9}$ & $0.53$ \\
MV     & $1.64\times10^{-9}$ & $2.41\times10^{-9}$ & $2.23\times10^{-9}$ & $0.34$ \\
\bottomrule
\end{tabular}
\caption{Radii of the $95\%$ balls for the group mean embeddings, at
$m = 201$ and $m' = 235$. The separation ratio is
$\norm{\widehat\delta_{m,m'}} / (\widehat r_{\mathrm{ASD}} +
\widehat r_{\mathrm{TD}})$. The two balls are disjoint, and the means
provably distinct, exactly when that ratio exceeds $1$. Radii are not
comparable across arms, since each embedding carries its own scale.}
\label{tab:abide-balls}
\end{table}

All three ratios fall below $1$, which agrees with
Table~\ref{tab:abide-pvalues}. What would it take to resolve the question? Declaring the means
distinct requires the two group balls to be disjoint,
\begin{equation}
\label{eq:abide-separation-criterion}
\widehat r(m) + \widehat r(m') < \Delta_\nu(P,Q),
\qquad \widehat r(k) = \{q_{0.95}/k\}^{1/2}.
\end{equation}
Writing $m = pM$ and $m' = (1-p)M$ for a total of $M$ subjects gives
\begin{equation}
\label{eq:abide-mstar}
M^{*} = \frac{q_{0.95}\{p^{-1/2}+(1-p)^{-1/2}\}^{2}}
{\Delta_\nu(P,Q)^{2}}.
\end{equation}
We substitute the bias-corrected estimate
$\widehat\Delta^2 = \norm{\widehat\delta_{m,m'}}^2
- \operatorname{tr}\widehat\Sigma_P/m
- \operatorname{tr}\widehat\Sigma_Q/m'$. That correction is essential.
The raw squared distance between sample means is inflated by exactly the
sampling noise the confidence set is measuring, so using it would make
the separation look attainable at any sample size.

At the observed allocation the inflation factor
$\{p^{-1/2}+(1-p)^{-1/2}\}^2$ equals $8.04$, and we obtain the
following.

\begin{center}
\begin{tabular}{lrr}
\toprule
Arm & $\widehat\Delta$ (bias-corrected) & $M^{*}$ \\
\midrule
PLACE  & $3.99\times10^{-8}$ & $3{,}365$ \\
PALACE & $6.74\times10^{-9}$ & $3{,}402$ \\
MV     & $0$ (at the boundary) & n/a \\
\bottomrule
\end{tabular}
\end{center}

The honest summary of the PLACE and PALACE results is not that there is
no difference. It is that if the difference is as large as the point
estimate suggests, roughly eight times the present cohort would be
needed to certify it. For MV the bias correction drives
$\widehat\Delta$ to zero, meaning the observed separation is smaller
than its own sampling noise, and no finite $M^{*}$ can be quoted. Note
that $M^*$ is nearly the same for PLACE and PALACE despite an eightfold
difference in $\widehat\Delta$. That is the scale cancellation
identified in Section~\ref{sec:sim-power}.

Proposition~\ref{prop:mean-embedding-transfer} also converts the lower
endpoint $\underline\Delta_{m,m'}^{\alpha}$ of
Corollary~\ref{cor:two-sample-confidence-ball} into an asymptotically
valid lower confidence bound,
$W_{\infty,0}(\overline\mu_P,\overline\mu_Q) \geq
\underline\Delta_{m,m'}^{\alpha}/(nL_\nu)$, on transport separation of
the padded mean measures. Here $\underline\Delta_{m,m'}^{\alpha} = 0$
for every arm, so the certificate is vacuous on this data. That is the
correct behaviour, and it illustrates a general point: the geometric
certificate inherits whatever the statistical interval delivers, and no
more.

\subsection{Comparison with the survival-based test}
\label{sec:abide-strand}

We also apply \textsc{Strand} \citep{STRAND} to the same $436$ subjects,
pooling $33{,}740$ persistence values from the autism group and
$38{,}478$ from the controls.

\begin{table}[!ht]
\centering
\begin{tabular}{lr}
\toprule
Quantity & Value \\
\midrule
Log-rank $\chi^2_1$, pooled          & $2.00$, $p = 0.157$ \\
Log-rank $\chi^2_1$, site-stratified & $7.75$, $p = \mathbf{0.005}$ \\
Diagram permutation, pooled          & $p = 0.552$ \\
Diagram permutation, site-stratified & $p = 0.561$ \\
\midrule
Hazard ratio & $0.990$, 95\% CI $[0.955,\,1.024]$ \\
Schoenfeld $\rho$ & $-0.445$ (proportional hazards fails) \\
Median lifetime difference & $1.9\times10^{-4}$,
 95\% CI $[-7.9,\,12.0]\times10^{-4}$ \\
\bottomrule
\end{tabular}
\caption{\textsc{Strand} applied to the ABIDE cohort. Permutation
$p$-values use $B = 999$ shuffles of diagram labels. The bootstrap
interval for the hazard ratio resamples diagrams rather than features.}
\label{tab:abide-strand}
\end{table}

The site-stratified asymptotic $p$-value of $0.005$ is, in our reading,
an artefact rather than a finding. On the same data the permutation
$p$-value is $0.561$. The discrepancy runs to two orders of magnitude,
in the direction and of the size that
Table~\ref{tab:sim-strand-calibration} predicts. Stratification does not
cure the pseudo-replication. It partitions $72{,}000$ dependent
observations into $20$ strata of dependent observations, then sums
twenty inflated contributions. The proportional-hazards diagnostic also
fails decisively, so the hazard ratio is not a meaningful single-number
summary here. The median lifetime difference is the effect size to read,
and its interval covers zero comfortably.

Neither the embedding test nor the diagram-level permutation version of
the survival test provides evidence of a group difference under its
respective estimand. The two procedures summarize different aspects of
the diagrams, so their non-rejections should not be treated as
independent evidence of equivalence. They indicate only that neither
summary detects a difference at the resolution available in this
sample.

\subsection{Scope and limitations}
\label{sec:abide-limitations}

We use one parcellation, one metric \eqref{eq:abide-metric}, and one
homological degree. Each is a defensible convention, and none is
neutral. Degree zero is untried.

The analysis does not adjust for age, sex, head motion, or other
subject-level covariates. It therefore describes marginal differences
between the observed diagnostic groups and does not isolate an effect
of diagnosis from these potential sources of confounding.

The analysis uses all quality-controlled subjects assigned to the
analysis half; the remaining subjects are used only to select the
embedding configurations.

The MV arm should be read as a single-scale summary here. Six scales
were constructed, yet roughly half of its coordinates are identically
zero, because the frame is inflated relative to these diagrams.

Most importantly, this is an illustration of the inferential machinery
and not a study of autism. The published record sets a low ceiling on
what any method should expect to find here. Working with the same $871$
quality-checked subjects, \cite{Abraham2017} reach $68\%$
classification accuracy under leave-one-site-out validation.
\cite{Nielsen2013} report $60\%$ for whole-brain classification, and
note that multisite accuracy falls short of what single sites achieve.
A null result is therefore close to what that literature would lead one
to expect. Nothing here should be read as a finding about the disorder.

\section{Discussion}
\label{sec:discussion}

This paper develops a framework for statistical inference on persistence
diagrams through fixed Hilbert-space embeddings. Once the embedding
configuration is fixed, each diagram becomes one Hilbert-valued observation.
This supports limit theory, covariance estimation, two-sample tests, and
confidence sets without treating persistence points from the same diagram as
independent replicates. For additive embeddings, the population mean also has
an interpretation as an embedded mean measure.

The statistical theory and the geometric guarantees serve different roles.
The limit theory requires a fixed measurable embedding and suitable moment
conditions. A lower-distortion certificate is needed only when an
embedding-space conclusion is translated back to diagram geometry. At the
population level, this translation also requires the mass and response
conditions introduced in Section~\ref{sec:embedding}. Without such structure,
geometric differences can cancel after population averaging. The proposed
test is therefore a test of equality of mean embeddings, rather than an
omnibus test of equality in distribution.

The numerical studies illustrate the resulting procedures in finite samples.
The Gaussian approximation is accurate for the symmetrized two-sample
statistic in the settings examined, and permutation calibration has the
expected behavior under the strict null. The comparison with the
survival-based procedure shows that lifetime summaries and joint
birth--death embeddings respond to different alternatives. It also shows why
calibration should preserve the diagram as the sampling unit. In the ABIDE
analysis, none of the embedding-based procedures detects a difference between
the diagnostic groups under its mean-embedding estimand. The site-conditioned
analysis and the diagram-level permutation version of the survival test lead
to the same conclusion. These non-rejections do not establish equality. They
indicate that the selected summaries do not detect a difference at the
resolution available in this sample.

Several directions follow from the present framework. One is inference under
the weak mean null when the two covariance operators differ. Studentized,
bootstrap, or other permutation-based procedures could provide more reliable
calibration in that setting. A second direction is data-adaptive selection of
the embedding configuration and truncation level without requiring a separate
pilot sample. Cross-fitting or repeated sample splitting may allow landmarks
to be learned while preserving valid inference. It would also be useful to
select the configuration by balancing truncation error, covariance
estimation, and certified geometric resolution, rather than by optimizing
representation quality alone.

A further extension is to move beyond population means. Covariance-based or
distributional procedures could detect alternatives that cancel at the mean
level, although retaining an explicit geometric interpretation would require
new assumptions. Covariate-adjusted, multisite, and longitudinal models would
also broaden the range of applications. These developments can build on the
same separation used throughout this paper: Hilbert-space methods provide the
statistical machinery, while the embedding structure determines what those
results imply about persistence-diagram geometry.

\bibliographystyle{plainnat}
\bibliography{main}

\begin{thebibliography}{27}
\providecommand{\natexlab}[1]{#1}
\providecommand{\url}[1]{\texttt{#1}}
\expandafter\ifx\csname urlstyle\endcsname\relax
  \providecommand{\doi}[1]{doi: #1}\else
  \providecommand{\doi}{doi: \begingroup \urlstyle{rm}\Url}\fi

\bibitem[Abraham et~al.(2014)Abraham, Pedregosa, Eickenberg, Gervais, Mueller,
  Kossaifi, Gramfort, Thirion, and Varoquaux]{Abraham2014nilearn}
Alexandre Abraham, Fabian Pedregosa, Michael Eickenberg, Philippe Gervais,
  Andreas Mueller, Jean Kossaifi, Alexandre Gramfort, Bertrand Thirion, and
  Ga{\"e}l Varoquaux.
\newblock Machine learning for neuroimaging with scikit-learn.
\newblock \emph{Frontiers in Neuroinformatics}, 8:\penalty0 14, 2014.
\newblock \doi{10.3389/fninf.2014.00014}.

\bibitem[Abraham et~al.(2017)Abraham, Milham, Di~Martino, Craddock, Samaras,
  Thirion, and Varoquaux]{Abraham2017}
Alexandre Abraham, Michael~P. Milham, Adriana Di~Martino, R.~Cameron Craddock,
  Dimitris Samaras, Bertrand Thirion, and Gael Varoquaux.
\newblock Deriving reproducible biomarkers from multi-site resting-state data:
  An {A}utism-based example.
\newblock \emph{NeuroImage}, 147:\penalty0 736--745, 2017.
\newblock \doi{10.1016/j.neuroimage.2016.10.045}.

\bibitem[Adams et~al.(2017)Adams, Emerson, Kirby, Neville, Peterson, Shipman,
  Chepushtanova, Hanson, Motta, and Ziegelmeier]{Adams2017}
Henry Adams, Tegan Emerson, Michael Kirby, Rachel Neville, Chris Peterson,
  Patrick Shipman, Sofya Chepushtanova, Eric Hanson, Francis Motta, and Lori
  Ziegelmeier.
\newblock Persistence images: {A} stable vector representation of persistent
  homology.
\newblock \emph{Journal of Machine Learning Research}, 18\penalty0
  (8):\penalty0 1--35, 2017.

\bibitem[Ali et~al.(2023)Ali, Asaad, Jimenez, Nanda, Paluzo-Hidalgo, and
  Soriano-Trigueros]{Ali2023-ht}
Dashti Ali, Aras Asaad, Maria-Jose Jimenez, Vidit Nanda, Eduardo
  Paluzo-Hidalgo, and Manuel Soriano-Trigueros.
\newblock A survey of vectorization methods in topological data analysis.
\newblock \emph{IEEE Transactions on Pattern Analysis and Machine
  Intelligence}, 45\penalty0 (12):\penalty0 14069--14080, 2023.
\newblock \doi{10.1109/TPAMI.2023.3308391}.

\bibitem[Bentkus(2003)]{Bentkus2003}
V.~Bentkus.
\newblock On the dependence of the {B}erry--{E}sseen bound on dimension.
\newblock \emph{Journal of Statistical Planning and Inference}, 113\penalty0
  (2):\penalty0 385--402, 2003.

\bibitem[Bubenik(2015)]{Bubenik15}
Peter Bubenik.
\newblock Statistical topological data analysis using persistence landscapes.
\newblock \emph{Journal of Machine Learning Research}, 16\penalty0
  (1):\penalty0 77--102, 2015.

\bibitem[Bubenik and Wagner(2020)]{BubenikWagner2020}
Peter Bubenik and Alexander Wagner.
\newblock Embeddings of persistence diagrams into {H}ilbert spaces.
\newblock \emph{Journal of Applied and Computational Topology}, 4\penalty0
  (3):\penalty0 339--351, 2020.
\newblock \doi{10.1007/s41468-020-00056-w}.

\bibitem[Chazal et~al.(2013)Chazal, Fasy, Lecci, Rinaldo, Singh, and
  Wasserman]{Chazal2014}
Fr{\'e}d{\'e}ric Chazal, Brittany~Terese Fasy, Fabrizio Lecci, Alessandro
  Rinaldo, Aarti Singh, and Larry Wasserman.
\newblock On the bootstrap for persistence diagrams and landscapes.
\newblock \emph{Modeling and Analysis of Information Systems}, 20\penalty0
  (6):\penalty0 111--120, 2013.

\bibitem[Chazal et~al.(2017)Chazal, Fasy, Lecci, Michel, Rinaldo, and
  Wasserman]{Chazal2017DTM}
Fr{\'e}d{\'e}ric Chazal, Brittany~Terese Fasy, Fabrizio Lecci, Bertrand Michel,
  Alessandro Rinaldo, and Larry Wasserman.
\newblock Robust topological inference: Distance to a measure and kernel
  distance.
\newblock \emph{Journal of Machine Learning Research}, 18\penalty0
  (159):\penalty0 1--40, 2017.

\bibitem[Cohen-Steiner et~al.(2007)Cohen-Steiner, Edelsbrunner, and
  Harer]{Cohen-Steiner2007}
David Cohen-Steiner, Herbert Edelsbrunner, and John Harer.
\newblock Stability of persistence diagrams.
\newblock \emph{Discrete \& Computational Geometry}, 37\penalty0 (1):\penalty0
  103--120, 2007.

\bibitem[Craddock et~al.(2013{\natexlab{a}})Craddock, Benhajali, Chu,
  Chouinard, Evans, Jakab, Khundrakpam, Lewis, Li, Milham, Yan, and
  Bellec]{Craddock2013pcp}
Cameron Craddock, Yassine Benhajali, Carlton Chu, Francois Chouinard, Alan
  Evans, Andr{\'a}s Jakab, Budhachandra~Singh Khundrakpam, John~David Lewis,
  Qingyang Li, Michael Milham, Chaogan Yan, and Pierre Bellec.
\newblock The neuro bureau preprocessing initiative: Open sharing of
  preprocessed neuroimaging data and derivatives.
\newblock \emph{Frontiers in Neuroinformatics}, 7, 2013{\natexlab{a}}.
\newblock \doi{10.3389/conf.fninf.2013.09.00041}.

\bibitem[Craddock et~al.(2013{\natexlab{b}})Craddock, Sikka, Cheung, Khanuja,
  Ghosh, Yan, Li, Lurie, Vogelstein, Burns, Colcombe, Mennes, Kelly,
  Di~Martino, Castellanos, and Milham]{Craddock2013cpac}
Cameron Craddock, Sharad Sikka, Brian Cheung, Ranjeet Khanuja, Satrajit~S.
  Ghosh, Chaogan Yan, Qingyang Li, Daniel Lurie, Joshua Vogelstein, Randal
  Burns, Stanley Colcombe, Maarten Mennes, Clare Kelly, Adriana Di~Martino,
  Francisco~Xavier Castellanos, and Michael Milham.
\newblock Towards automated analysis of connectomes: The {C}onfigurable
  {P}ipeline for the {A}nalysis of {C}onnectomes ({C-PAC}).
\newblock \emph{Frontiers in Neuroinformatics}, 7, 2013{\natexlab{b}}.
\newblock \doi{10.3389/conf.fninf.2013.09.00042}.

\bibitem[Craddock et~al.(2012)Craddock, James, Holtzheimer, Hu, and
  Mayberg]{Craddock2012}
R.~Cameron Craddock, G.~Andrew James, Paul~E. Holtzheimer, Xiaoping~P. Hu, and
  Helen~S. Mayberg.
\newblock A whole brain {fMRI} atlas generated via spatially constrained
  spectral clustering.
\newblock \emph{Human Brain Mapping}, 33\penalty0 (8):\penalty0 1914--1928,
  2012.
\newblock \doi{10.1002/hbm.21333}.

\bibitem[Di~Martino et~al.(2014)Di~Martino, Yan, Li, Denio, Castellanos,
  Alaerts, Anderson, Assaf, Bookheimer, Dapretto, et~al.]{di2014autism}
Adriana Di~Martino, Chao-Gan Yan, Qingyang Li, Erin Denio, Francisco~X
  Castellanos, Kaat Alaerts, Jeffrey~S Anderson, Michal Assaf, Susan~Y
  Bookheimer, Mirella Dapretto, et~al.
\newblock The autism brain imaging data exchange: towards a large-scale
  evaluation of the intrinsic brain architecture in autism.
\newblock \emph{Molecular psychiatry}, 19\penalty0 (6):\penalty0 659--667,
  2014.

\bibitem[Fasy et~al.(2014)Fasy, Lecci, Rinaldo, Wasserman, Balakrishnan, and
  Singh]{Fasy2017}
Brittany~Terese Fasy, Fabrizio Lecci, Alessandro Rinaldo, Larry Wasserman,
  Sivaraman Balakrishnan, and Aarti Singh.
\newblock Confidence sets for persistence diagrams.
\newblock \emph{Annals of Statistics}, 42\penalty0 (6):\penalty0 2301--2339,
  2014.
\newblock \doi{10.1214/14-AOS1252}.

\bibitem[Kusano et~al.(2016)Kusano, Hiraoka, and Fukumizu]{Kusano2016}
Genki Kusano, Yasuaki Hiraoka, and Kenji Fukumizu.
\newblock Persistence weighted {G}aussian kernel for topological data analysis.
\newblock In \emph{Proceedings of the 33rd International Conference on Machine
  Learning (ICML)}, pages 2004--2013, 2016.

\bibitem[Ledoux and Talagrand(2011)]{LedouxTalagrand2011}
Michel Ledoux and Michel Talagrand.
\newblock \emph{Probability in {B}anach Spaces: Isoperimetry and Processes}.
\newblock Classics in Mathematics. Springer, Berlin, 2011.
\newblock Reprint of the 1991 edition.

\bibitem[Majhi et~al.(2026{\natexlab{a}})Majhi, Mitra, Virk, and
  Bagchi]{PALACE}
Sushovan Majhi, Atish Mitra, {\v Z}iga Virk, and Pramita Bagchi.
\newblock A closed-form adaptive-landmark kernel for certified point-cloud and
  graph classification, 2026{\natexlab{a}}.
\newblock URL \url{https://arxiv.org/abs/2605.04046}.

\bibitem[Majhi et~al.(2026{\natexlab{b}})Majhi, Mitra, Virk, and Bagchi]{PLACE}
Sushovan Majhi, Atish Mitra, {\v Z}iga Virk, and Pramita Bagchi.
\newblock A closed-form persistence-landmark pipeline for certified point-cloud
  and graph classification.
\newblock \emph{Transactions on Machine Learning Research}, 2026{\natexlab{b}}.
\newblock ISSN 2835-8856.
\newblock URL \url{https://openreview.net/forum?id=4kZxNlE5Ve}.
\newblock arXiv:2605.02836.

\bibitem[Mileyko et~al.(2011)Mileyko, Mukherjee, and Harer]{Mileyko2011}
Yuriy Mileyko, Sayan Mukherjee, and John Harer.
\newblock Probability measures on the space of persistence diagrams.
\newblock \emph{Inverse Problems}, 27\penalty0 (12):\penalty0 124007, 2011.
\newblock \doi{10.1088/0266-5611/27/12/124007}.

\bibitem[Mitra and Virk(2021)]{Mitra2021}
Atish Mitra and {\v Z}iga Virk.
\newblock The space of persistence diagrams on $n$ points coarsely embeds into
  {H}ilbert space.
\newblock \emph{Proceedings of the American Mathematical Society}, 149\penalty0
  (6):\penalty0 2693--2703, 2021.
\newblock \doi{10.1090/proc/15363}.

\bibitem[Mitra and Virk(2024)]{Mitra2024}
Atish Mitra and {\v Z}iga Virk.
\newblock Geometric embeddings of spaces of persistence diagrams with explicit
  distortions.
\newblock arXiv:2401.05298, 2024.
\newblock URL \url{https://arxiv.org/abs/2401.05298}.

\bibitem[Murris et~al.(2026)Murris, Stolz, and Borgwardt]{STRAND}
Juliette Murris, Bernadette Stolz, and Karsten Borgwardt.
\newblock From persistence to survival: Hypothesis testing, effect sizes and
  vectorisation for topological features, 2026.
\newblock URL \url{https://arxiv.org/abs/2606.11911}.

\bibitem[Nielsen et~al.(2013)Nielsen, Zielinski, Fletcher, Alexander, Lange,
  Bigler, Lainhart, and Anderson]{Nielsen2013}
Jared~A. Nielsen, Brandon~A. Zielinski, P.~Thomas Fletcher, Andrew~L.
  Alexander, Nicholas Lange, Erin~D. Bigler, Janet~E. Lainhart, and Jeffrey~S.
  Anderson.
\newblock Multisite functional connectivity {MRI} classification of autism:
  {ABIDE} results.
\newblock \emph{Frontiers in Human Neuroscience}, 7:\penalty0 599, 2013.
\newblock \doi{10.3389/fnhum.2013.00599}.

\bibitem[Reininghaus et~al.(2015)Reininghaus, Huber, Bauer, and
  Kwitt]{Reininghaus2015}
Jan Reininghaus, Stefan Huber, Ulrich Bauer, and Roland Kwitt.
\newblock A stable multi-scale kernel for topological machine learning.
\newblock In \emph{Proceedings of the IEEE Conference on Computer Vision and
  Pattern Recognition (CVPR)}, pages 4741--4748, 2015.
\newblock \doi{10.1109/CVPR.2015.7299106}.

\bibitem[Turner et~al.(2014)Turner, Mileyko, Mukherjee, and Harer]{Turner2014}
Katharine Turner, Yuriy Mileyko, Sayan Mukherjee, and John Harer.
\newblock Fr{\'e}chet means for distributions of persistence diagrams.
\newblock \emph{Discrete \& Computational Geometry}, 52\penalty0 (1):\penalty0
  44--70, 2014.
\newblock \doi{10.1007/s00454-014-9604-7}.

\bibitem[Wasserman(2018)]{Wasserman2018}
Larry Wasserman.
\newblock Topological data analysis.
\newblock \emph{Annual Review of Statistics and Its Application}, 5:\penalty0
  501--532, 2018.
\newblock \doi{10.1146/annurev-statistics-031017-100045}.

\end{thebibliography}

\appendix

\section*{Appendix: Technical Details}
\addcontentsline{toc}{section}{Appendix: Proofs}

\section{Technical Details for Landmark Embeddings}
\label{app:embedding-technical}

This appendix separates general fixed-Hilbert inference from the additive
theory for PLACE and PALACE. It gives the padded transport construction, verifies the
construction-specific certificate domains, proves the additive population
results, and develops a separate whole-law population-transfer theorem for the
original Mitra--Virk map.

\subsection{Proof of Lemma~\ref{lem:Dn-polish}}
\label{app:Dn-polish}

Let
$$
C_L:=\{(b,d)\in[0,L]^2:b\leq d\}.
$$
Define $F:C_L^n\to\mathcal D_n$ by retaining the off-diagonal
coordinates as a multiset and discarding the diagonal coordinates.
The map is onto, since any diagram with fewer than $n$ points can
be represented by adding diagonal coordinates.

Matching corresponding off-diagonal coordinates and deleting the
remaining points gives
$$
d_B\{F(x_1,\ldots,x_n),F(y_1,\ldots,y_n)\}
\leq
\max_{1\leq i\leq n}\|x_i-y_i\|_\infty.
$$
Thus $F$ is continuous. Since $C_L^n$ is compact, its image
$\mathcal D_n$ is compact.

\subsection{Padded transport geometry}
\label{app:padded-transport}

Recall
$\mathbb H_L^{\varnothing}=\mathbb H_L\cup\{\varnothing\}$, where
$\varnothing$ represents the entire diagonal. For $x\in\mathbb H_L$, write
$\delta_{\mathsf{Diag}}(x)=\inf_{z\in\mathsf{Diag}}\|x-z\|_\infty$.
Define
\begin{equation}
\label{eq:absence-state-metric}
d_0(x,y)
:=
\begin{cases}
\displaystyle
\min\!\left\{
\|x-y\|_\infty,
\max\!\left\{
\delta_{\mathsf{Diag}}(x),
\delta_{\mathsf{Diag}}(y)
\right\}
\right\},
&x,y\in\mathbb H_L,\\[1.2ex]
\delta_{\mathsf{Diag}}(x),
&x\in\mathbb H_L,\ y=\varnothing,\\
\delta_{\mathsf{Diag}}(y),
&x=\varnothing,\ y\in\mathbb H_L,\\
0,&x=y=\varnothing.
\end{cases}
\end{equation}

\begin{lemma}[One-point bottleneck quotient]
\label{lem:one-point-quotient}
The function $d_0$ is the metric induced by the bottleneck distance on the
space consisting of the empty diagram and diagrams with one off-diagonal
point, after identifying all diagonal representatives with $\varnothing$.
Consequently, $d_0$ is a metric on
$\mathbb H_L^{\varnothing}$.
\end{lemma}

\begin{proof}
For two one-point diagrams, a bottleneck matching either pairs the two points
directly, at cost $\|x-y\|_\infty$, or sends both to the diagonal, at cost
$\max\{\delta_{\mathsf{Diag}}(x),
\delta_{\mathsf{Diag}}(y)\}$. Matching a one-point diagram to the empty
diagram costs its distance to the diagonal. These are exactly the cases in
\eqref{eq:absence-state-metric}. The metric properties follow from those of
the bottleneck distance after quotienting its zero-distance diagonal class.
\end{proof}

For finite nonnegative measures $\alpha$ and $\gamma$ of the same mass, define
\begin{equation}
\label{eq:diagonal-aware-transport}
W_{\infty,0}(\alpha,\gamma)
:=
\inf_{\pi\in\Gamma(\alpha,\gamma)}
\operatorname*{ess\,sup}_{(x,y)\sim\pi}d_0(x,y).
\end{equation}
Scaling both measures by the same positive constant leaves this distance
unchanged.

\begin{proposition}[Padded transport representation]
\label{prop:padded-transport-representation}
For all $D,E\in\mathcal D_n$,
\[
d_B(D,E)
=
W_{\infty,0}(\widetilde N_D,\widetilde N_E).
\]
\end{proposition}

\begin{proof}
A partial bottleneck matching of cost at most $r$ extends to a coupling of the
padded counting measures by retaining directly matched pairs, pairing points
sent to the diagonal either with each other or with copies of $\varnothing$,
and pairing unused copies of $\varnothing$ together. Equation
\eqref{eq:absence-state-metric} shows that every transport pair has cost at
most $r$. Hence
$W_{\infty,0}(\widetilde N_D,\widetilde N_E)\le d_B(D,E)$.

For the reverse inequality, label the $n$ unit-mass slots of each padded
counting measure, including repeated atoms and copies of $\varnothing$. A
coupling is then represented by an $n\times n$ doubly stochastic matrix after
normalizing by the unit slot masses. If its essential maximum cost is at most
$r$, every positive matrix entry joins slots at $d_0$-distance at most $r$.
By the Birkhoff--von Neumann theorem, the matrix is a convex combination of
permutation matrices, and the standard support-preserving decomposition uses
only positive entries of the original matrix. Choose any permutation in that
decomposition. It gives a bijection of the labeled slots with maximum
$d_0$-cost at most $r$. Retain a paired off-diagonal pair as a direct diagram
match when the direct branch of \eqref{eq:absence-state-metric} realizes its
cost; otherwise send both points to the diagonal. A pair involving
$\varnothing$ sends its off-diagonal member to the diagonal. The resulting
partial diagram matching has bottleneck cost at most $r$. Taking infima proves
the reverse inequality.
\end{proof}

The maps $D\mapsto\widetilde N_D$ and, under (E1),
$D\mapsto\Phi(D;\nu)$ are Borel measurable. The latter follows because the
point response is Borel and each diagram contains at most $n$ atoms.

\subsection{General fixed-Hilbert inference scope and additive-interface proofs}
\label{app:additive-interface-proofs}

Assumption~\ref{ass:fixed-hilbert-embedding} is deliberately modular. Borel
measurability and first moments give Bochner means; second moments give the
usual Hilbert-space central limit framework; and higher-moment or boundedness
conditions should be invoked only by results that use them. Finite-coordinate
and truncation results separately invoke
Assumption~\ref{ass:finite-approximation}. The embedding configuration $\nu$
is conditioned upon throughout. If it is fitted
on an independent reference sample, all statements apply conditionally on
that sample with the realized constants. Same-sample embedding-configuration selection is a
separate selection problem.

\begin{proof}[Proof of Proposition~\ref{prop:interface-consequences}]
By (E1), (E2), and
$\widetilde N_D(\mathbb H_L^{\varnothing})=n$,
\[
\|\Phi(D;\nu)\|_2
\le
\int\|\varphi_\nu^0(x)\|_2\,d\widetilde N_D(x)
\le n\kappa_\nu.
\]
For any coupling $\pi$ of $\widetilde N_D$ and $\widetilde N_E$,
\[
\Phi(D;\nu)-\Phi(E;\nu)
=
\int\{\varphi_\nu^0(x)-\varphi_\nu^0(y)\}\,d\pi(x,y),
\]
and hence
\[
\|\Phi(D;\nu)-\Phi(E;\nu)\|_2
\le
nL_\nu
\operatorname*{ess\,sup}_{(x,y)\sim\pi}d_0(x,y).
\]
Take the infimum over $\pi$ and use
Proposition~\ref{prop:padded-transport-representation}.
\end{proof}

\begin{proof}[Proof of Proposition~\ref{prop:mean-embedding-transfer}]
The bound $\|\Phi(X;\nu)\|_2\le n\kappa_\nu$ gives Bochner integrability.
Linearity and Fubini's theorem for the finite random measure yield
\[
\mathbb E_P\{\Phi(X;\nu)\}
=
\int\varphi_\nu^0(x)\,d\overline\mu_P(x)
=T_\nu[\overline\mu_P],
\]
and similarly for $Q$. Applying the preceding coupling argument to the two
mean measures, each of mass $n$, gives
\eqref{eq:population-upper-transfer}.
\end{proof}

\subsection{Construction-specific verification and certified scale sets}
\label{app:construction-verification}

The constructions do not share a single numerical floor, scale domain, or
certified pair class. Assumption~\ref{ass:diagram-distortion-floor} therefore
uses $\mathcal I_\nu^\Psi$, $\mathcal C_{\nu,t}^{\Psi}$, and the valid
threshold floor $\rho_{-,\nu}^{\Psi}(t)$. The statements below distinguish a
raw actual-separation function from the threshold floor required by Assumption \ref{ass:diagram-distortion-floor}.

\paragraph{Mitra--Virk: infinite-dimensional maps.}
Fix maximal diagram cardinality $n$. In
\cite[Theorem~4.3]{Mitra2024}, let $R_k\uparrow\infty$,
$\widetilde R_k\downarrow0$, and let $(w_k)_{k\geq1}$ be a positive unit vector
with $w_kR_k\to\infty$. With
$c_n=(3\,2^{n+3})^{-1}$, the theorem constructs a normalized coarse map
$\Psi_{\nu}^{\mathrm c}$ and a normalized uniform map
$\Psi_{\nu}^{\mathrm u}$, each $1$-Lipschitz, with displayed raw step floors
\begin{align}
\rho_{\mathrm{raw},\nu}^{\mathrm c}(u)
&=
 c_n\sum_{k\geq1}w_kR_k\,
 \mathbf 1_{[R_k,R_{k+1})}(u),
\label{eq:mv-raw-coarse-floor}\\
\rho_{\mathrm{raw},\nu}^{\mathrm u}(u)
&=
 c_n\sum_{k\geq1}w_k\widetilde R_k\,
 \mathbf 1_{[\widetilde R_k,\widetilde R_{k-1})}(u),
\qquad \widetilde R_0=\infty.
\label{eq:mv-raw-uniform-floor}
\end{align}
The normalized direct sum
$\Psi_{\nu}^{\mathrm{cu}}=2^{-1/2}
(\Psi_{\nu}^{\mathrm c},\Psi_{\nu}^{\mathrm u})$ is also $1$-Lipschitz and has
raw floor
\[
\rho_{\mathrm{raw},\nu}^{\mathrm{cu}}(u)
=
2^{-1/2}
\left(
(\rho_{\mathrm{raw},\nu}^{\mathrm c}(u))^2
+
(\rho_{\mathrm{raw},\nu}^{\mathrm u}(u))^2
\right)^{1/2}.
\]
The paper calls these distortion functions, but the displayed step heights
need not be monotone for an arbitrary admissible weight sequence. For Assumption \ref{ass:diagram-distortion-floor} we
therefore use
\begin{equation}
\label{eq:mv-threshold-floor}
\rho_{-,\nu}^{\mathrm{MV},\mathrm{cu}}(t)
:=
\inf_{u\geq t}
\rho_{\mathrm{raw},\nu}^{\mathrm{cu}}(u),
\qquad t>0.
\end{equation}
The hypotheses $w_kR_k\to\infty$ and positivity of the uniform step heights
make this infimum positive for each $t>0$. Thus
$\mathcal I_\nu^{\mathrm{MV},\mathrm{cu}}=(0,\infty)$ and
$\mathcal C_{\nu,t}^{\mathrm{MV},\mathrm{cu}}
=\mathcal D_n\times\mathcal D_n$. No PLACE/PALACE-style pair coherence is
required. If a chosen sequence makes the displayed raw floor nondecreasing,
it may be used directly.

The $1$-Lipschitz constant applies to each normalized component and to the
normalized direct sum in Theorem~4.3. An unnormalized direct sum
$(\Psi_1,\ldots,\Psi_J)$ with component constants $L_1,\ldots,L_J$ has
Lipschitz constant at most $(\sum_{j=1}^JL_j^2)^{1/2}$ under the Hilbert direct
sum norm; the constant is not inherited componentwise without this
normalization.

\paragraph{Mitra--Virk: finite-dimensional bounded-domain map.}
On $\mathcal D_n\cap[0,L]^2$, 
\cite[Theorem~5.1]{Mitra2024} fixes scales
$0<R_1<\cdots<R_N\leq L$ and a unit weight vector and constructs a
$1$-Lipschitz map into
$\mathbb R^{\nu_1+\cdots+\nu_N}$. Its step guarantee is: if
$d_B(D,E)\geq R_i$, then
\[
\|\Psi_\nu^{\mathrm{MV},\mathrm{fin}}(D)
-
\Psi_\nu^{\mathrm{MV},\mathrm{fin}}(E)\|_2
\geq
c_n
\left(\sum_{k=1}^{i}w_k^2R_k^2\right)^{1/2}.
\]
The same theorem gives the nondecreasing affine threshold floor
\begin{equation}
\label{eq:mv-finite-affine-floor}
\rho_{-,\nu}^{\mathrm{MV},\mathrm{fin}}(t)
=
\lambda_\nu^{\mathrm{MV}}(t-R_1),
\qquad R_1<t\leq L,
\end{equation}
where
\begin{align*}
\lambda_\nu^{\mathrm{MV}}
:=c_n\min\Bigg\{
&\min_{2\leq i\leq N}
\frac{(\sum_{k=1}^{i-1}w_k^2R_k^2)^{1/2}}{R_i-R_1},\\
&\frac{(\sum_{k=1}^{N-1}w_k^2R_k^2+w_N^2L^2)^{1/2}}{L-R_1}
\Bigg\}.
\end{align*}
It is zero on $[0,R_1]$. Hence
$\mathcal I_\nu^{\mathrm{MV},\mathrm{fin}}=(R_1,L]$ and the certified pair
class is the full bounded-domain pair space. Lemma~3.12 gives cover
multiplicity $4^n$, and Example~5.2 gives an explicit map with $N4^n$
coordinates. These are construction dimensions, not statistical sample sizes.

\paragraph{PLACE: step certificate.}
For the normalized multiscale pooled map of \cite{PLACE},
Proposition~2.1(a) gives
\[
\|\Phi_\nu^{\mathrm{PLACE}}(D)
-
\Phi_\nu^{\mathrm{PLACE}}(E)\|_2
\leq n\,d_B(D,E)
\]
on diagrams with at most $n$ retained points. Definition~2.1 calls scale $k$
active for a pair when $3R_k\leq d_B(D,E)$ and defines $\nu$-coherence by the
per-block condition
$\|\Phi_{R_k}(D)-\Phi_{R_k}(E)\|_2^2\geq R_k^2/32$ at every active scale.
Proposition~2.1(b) states that, when $d_B(D,E)\geq3R_1$ and the pair is
$\nu$-coherent,
\[
\|\Phi_\nu^{\mathrm{PLACE}}(D)
-
\Phi_\nu^{\mathrm{PLACE}}(E)\|_2
\geq
\frac1{16}
\left(
\sum_{k:3R_k\leq d_B(D,E)}w_k^2R_k^2
\right)^{1/2}.
\]
The right side is nondecreasing in the actual separation. Therefore the Assumption \ref{ass:diagram-distortion-floor}
threshold floor for the step statement is
\begin{equation}
\label{eq:place-step-threshold-floor}
\rho_{-,\nu}^{\mathrm{PLACE},\mathrm{step}}(t)
:=
\frac1{16}
\left(
\sum_{k:3R_k\leq t}w_k^2R_k^2
\right)^{1/2},
\qquad 3R_1\leq t\leq L,
\end{equation}
with $\mathcal C_{\nu,t}^{\mathrm{PLACE},\mathrm{step}}$ consisting of pairs
that are $\nu$-coherent at every scale active for that pair. This step
statement has its own lower endpoint $3R_1$.

\paragraph{PLACE: affine certificate.}
Corollary~2.1 of \cite{PLACE} is a separate statement. For
$R_1\leq d_B(D,E)\leq L$ and pairs $\nu$-coherent at every active scale, it
gives
\begin{equation}
\label{eq:place-affine-threshold-floor}
\|\Phi_\nu^{\mathrm{PLACE}}(D)-
\Phi_\nu^{\mathrm{PLACE}}(E)\|_2
\geq
\lambda(\nu)\{d_B(D,E)-R_1\},
\end{equation}
where
\[
\lambda(\nu)
=
\frac1{48}
\min\left\{
\min_{2\leq i\leq N}
\frac{(\sum_{k=1}^{i-1}w_k^2R_k^2)^{1/2}}{R_i-R_1},
\frac{(\sum_{k=1}^{N}w_k^2R_k^2)^{1/2}}{L-R_1}
\right\}.
\]
Thus the affine Assumption \ref{ass:diagram-distortion-floor} scale set is $(R_1,L]$ and its nondecreasing threshold floor
is
$\rho_{-,\nu}^{\mathrm{PLACE},\mathrm{aff}}(t)
=\lambda(\nu)(t-R_1)$. The condition $d_B\geq3R_1$ belongs to
Proposition~2.1(b), not to the affine corollary; the corollary's proof handles
$[R_1,3R_1)$ through the bounded-domain Mitra--Virk argument.

The PLACE coordinates are sums of point responses, so the raw map is additive.
Under the paper's block normalization, the singleton map is $1$-Lipschitz and
the diagram map has constant $n$. Its codomain has
$\sum_{k=1}^N|G_{R_k}^+|$ coordinates and is finite-dimensional.

\paragraph{PALACE.}
Let $\nu=\{(p_k,r_k,w_k)\}_{k=1}^{J_\nu}$ be the fixed embedding configuration of
\cite{PALACE}, with $\sum_kw_k^2=1$, and let
$\mathcal R\subseteq\mathcal D_1\cap[0,L]^2$ be the covered data support.
Definition~2.2 defines the realized Lebesgue number
\[
\lambda_0(\nu)
=
\inf_{x\in\mathcal R}
\max_k\varphi_{p_k,r_k}(x)
\]
and calls the configuration $\tau$-admissible when the landmark balls cover
$\mathcal R$,
$\lambda_0(\nu)\geq\tau/4$, and
$\max_kr_k\leq\{\tau+\lambda_0(\nu)\}/2$.
Definition~2.3 requires either a one-point diagram or an optimal matching
$\sigma$ for which
\[
\min_{i\neq j}d_B(a_i,b_{\sigma(j)})
>
3d_B(D,E).
\]
Theorem~2.1 then states that, if all points of $D$ and $E$ lie in
$\mathcal R$, $d_B(D,E)\geq\tau$, and non-interference holds,
\begin{equation}
\label{eq:palace-threshold-floor}
\|\Phi_\nu^{\mathrm{PALACE}}(D)-
\Phi_\nu^{\mathrm{PALACE}}(E)\|_2
\geq
\rho_{-,\nu}^{\mathrm{PALACE}}(\tau)
:=
\frac{\tau}{4}
\min_{k:r_k\geq\tau/4}w_k.
\end{equation}
Thus equality in \eqref{eq:palace-threshold-floor} defines the theorem's
stated certificate; it is not an assertion that the observed embedding
distance equals that value. The proof gives the stronger realized witness
$\lambda_0(\nu)\min_{k:r_k\geq\tau/4}w_k$, and then uses
$\lambda_0(\nu)\geq\tau/4$ to obtain the displayed certificate.

Accordingly,
\[
\mathcal I_\nu^{\mathrm{PALACE}}
=
\{\tau>0:\nu\text{ is $\tau$-admissible}\},
\]
and $\mathcal C_{\nu,\tau}^{\mathrm{PALACE}}$ contains pairs supported in
$\mathcal R$ that satisfy Definition~2.3. Theorem~2.1 is already a threshold
statement at the target $\tau$; no monotonicity of
$\tau\mapsto\rho_{-,\nu}^{\mathrm{PALACE}}(\tau)$ is needed to use it in Assumption \ref{ass:diagram-distortion-floor}.
Definition~2.1 and equation~(2.5) give a finite-dimensional additive map into
$\mathbb R^{J_\nu}$ with diagram-level upper constant $n$ (and singleton constant
$1$) under the paper's weight normalization.

\begin{remark}[Rescaling]
Multiplying an embedding by a positive scalar multiplies its upper Lipschitz
constant, threshold floor, and finite-response signal by the same scalar.
Consequently, ratios such as
$\rho_{-,\nu}^{\Psi}(t)/L_\nu^\Psi$ are invariant under global rescaling.
\end{remark}

\subsection{Rare-mass obstruction}
\label{app:population-transfer-proofs}

\begin{proof}[Verification of Example~\ref{ex:pop-obstruction}]
The padded mean measures are
\[
\overline\mu_{P_\varepsilon}
=\varepsilon\delta_z+(n-\varepsilon)\delta_\varnothing,
\qquad
\overline\mu_Q=n\delta_\varnothing.
\]
Their common mass couples at zero cost, and the remaining $\varepsilon$ units
move between $z$ and $\varnothing$ at cost
$d_0(z,\varnothing)=\tau$. Thus
$W_{\infty,0}(\overline\mu_{P_\varepsilon},\overline\mu_Q)=\tau$.
Linearity gives
\[
\mathbb E_{P_\varepsilon}\Phi(X;\nu)
-
\mathbb E_Q\Phi(Y;\nu)
=
\varepsilon\{\Phi(D_z;\nu)-\Phi(D_\varnothing;\nu)\},
\]
which proves the claimed convergence of $\Delta_\nu$.
\end{proof}

\subsection{Template--thinning--perturbation proofs and secondary corollaries}
\label{app:template-transfer-proofs}

\paragraph{Mean-measure derivation for the template--thinning--perturbation model.}
Under Assumption~\ref{def:template-thinning-perturbation}, let $Z_{rj}$ be defined
arbitrarily on $\{I_{rj}=0\}$, with
$\Pr(I_{rj}=1)=p_{rj}$ and
$\mathcal L(Z_{rj}\mid I_{rj}=1)=G_{rj}$. No product structure is assumed for
$\{(I_{rj},Z_{rj})\}_{j=1}^{m_r}$. For every bounded Borel function $f$,
\begin{align*}
\mathbb E\int f\,dN_{X_r}
&=
\sum_{j=1}^{m_r}\mathbb E\{I_{rj}f(Z_{rj})\}\\
&=
\sum_{j=1}^{m_r}p_{rj}
\int f(x)\,dG_{rj}(x).
\end{align*}
Hence, as finite measures,
\[
\mathbb E N_{X_r}
=
\sum_{j=1}^{m_r}p_{rj}G_{rj}.
\]
Since
$\widetilde N_{X_r}=N_{X_r}+(n-\sum_jI_{rj})\delta_\varnothing$,
linearity also gives
\[
\overline\mu_r
=
\sum_{j=1}^{m_r}p_{rj}G_{rj}
+
\left(n-\sum_{j=1}^{m_r}p_{rj}\right)\delta_\varnothing.
\]
This derivation uses only the featurewise marginals. Repeated latent template
locations remain distinct labels in the generating model and are aggregated
only when the union-support masses in \eqref{eq:union-masses} are formed.

\begin{lemma}[Tail-to-first-moment conversion]
\label{lem:tail-mean}
Under Assumption~\ref{ass:template-perturbation-tails},
$e_{rj}\le A_0\sigma_{rj}\sqrt{\pi/2}$.
\end{lemma}

\begin{proof}
For $Z=d_0(X,z_{rj})\ge0$ with $X\sim G_{rj}$, the layer-cake identity gives
\[
e_{rj}
=
\int_0^\infty\Pr(Z>t)\,dt
\le
A_0\int_0^\infty
\exp\!\left(-\frac{t^2}{2\sigma_{rj}^2}\right)dt
=
A_0\sigma_{rj}\sqrt{\frac\pi2}.
\]
\end{proof}

\begin{proof}[Proof of Lemma~\ref{lem:center-approximation}]
Set
$S_r=\sum_{j=1}^{m_r}p_{rj}\varphi_\nu^0(z_{rj})$. Then
\[
T_\nu[\overline\mu_r]-S_r
=
\sum_{j=1}^{m_r}p_{rj}
\int\{\varphi_\nu^0(x)-\varphi_\nu^0(z_{rj})\}\,dG_{rj}(x),
\]
so (E2) implies
\[
\|T_\nu[\overline\mu_r]-S_r\|_2
\le
L_\nu\sum_{j=1}^{m_r}p_{rj}e_{rj}.
\]
Add and subtract $S_P-S_Q$ and apply the reverse triangle inequality to obtain
\eqref{eq:shared-perturbation-bound-general}. Lemma~\ref{lem:tail-mean}
gives \eqref{eq:shared-perturbation-bound}.
\end{proof}

\paragraph{Unweighted template-count identity.}
By the definition of $q$ and aggregation over multiplicities,
\begin{align*}
V_{P,Q}q
&=
\sum_{i=1}^M
\left\{
\#\{j:z_{Pj}=x_i\}
-
\#\{\ell:z_{Q\ell}=x_i\}
\right\}
\varphi_\nu^0(x_i)\\
&=
\sum_{j=1}^{m_P}\varphi_\nu^0(z_{Pj})
-
\sum_{\ell=1}^{m_Q}\varphi_\nu^0(z_{Q\ell})\\
&=
\Phi(D_P^\star;\nu)-\Phi(D_Q^\star;\nu).
\end{align*}
The final equality uses additivity and
$\varphi_\nu^0(\varnothing)=0$, so different template cardinalities require
no additional absence coordinate. This proves
\eqref{eq:unweighted-template-response}.

\begin{proposition}[Direct heterogeneous response bound]
\label{prop:direct-heterogeneous-response-bound}
Under Assumption~\ref{ass:additive-interface}, the model in
Assumption~\ref{def:template-thinning-perturbation}, and
Assumption~\ref{ass:template-perturbation-tails}, let $c$ and $V_{P,Q}$ be
defined by \eqref{eq:union-masses} and \eqref{eq:response-matrix}. Then,
without any nonzero or Assumption \ref{ass:diagram-distortion-floor} condition,
\begin{align}
\Delta_\nu(P,Q)
\geq
\Bigg[
&\|V_{P,Q}c\|_2\notag\\
&-
A_0L_\nu\sqrt{\frac\pi2}
\left\{
\sum_jp_{Pj}\sigma_{Pj}
+
\sum_\ell p_{Q\ell}\sigma_{Q\ell}
\right\}
\Bigg]_+.
\label{eq:heterogeneous-direct-response-bound}
\end{align}
\end{proposition}

\begin{proof}
Aggregation over the distinct union support gives
$S_P-S_Q=V_{P,Q}c$ by \eqref{eq:heterogeneous-noiseless-response}.
Substitution in \eqref{eq:shared-perturbation-bound} proves the result. The
argument remains valid when $c=0$ or $V_{P,Q}c=0$ and does not invoke Assumption \ref{ass:diagram-distortion-floor}.
\end{proof}

\begin{lemma}[Transfer of the Assumption \ref{ass:diagram-distortion-floor} floor to the weighted response]
\label{lem:heterogeneous-e4-transfer}
Under Assumption~\ref{ass:prevalence-identifiability-regimes}(ii),
\begin{equation}
\label{eq:heterogeneous-e4-signal-transfer}
\|V_{P,Q}c\|_2
\geq
\eta_{P,Q}\rho_-(s;\nu).
\end{equation}
\end{lemma}

\begin{proof}
Because $c,q\in\mathcal S_{P,Q}$ and the restricted minimum gain is positive,
\begin{align*}
\|V_{P,Q}c\|_2
&\geq
\sigma_{\min}(V_{P,Q};\mathcal S_{P,Q})\|c\|_2\\
&=
\frac{\sigma_{\max}(V_{P,Q};\mathcal S_{P,Q})}
{\operatorname{cond}(V_{P,Q};\mathcal S_{P,Q})}\|c\|_2\\
&\geq
\frac{\|V_{P,Q}q\|_2}
{\operatorname{cond}(V_{P,Q};\mathcal S_{P,Q})\|q\|_2}\|c\|_2.
\end{align*}
The third line uses
$\|V_{P,Q}q\|_2\leq
\sigma_{\max}(V_{P,Q};\mathcal S_{P,Q})\|q\|_2$.
By \eqref{eq:unweighted-template-response} and the Assumption \ref{ass:diagram-distortion-floor} certification in
Assumption~\ref{ass:prevalence-identifiability-regimes}(ii),
$\|V_{P,Q}q\|_2\geq\rho_-(s;\nu)$. Substitution of
\eqref{eq:geometric-retention-factor} proves
\eqref{eq:heterogeneous-e4-signal-transfer}.
\end{proof}

\begin{corollary}[Assumption \ref{ass:diagram-distortion-floor} transfer under heterogeneous prevalence]
\label{cor:heterogeneous-e4-transfer}
Under the assumptions of
Theorem~\ref{thm:population-lower-bound}(ii),
\begin{align}
\Delta_\nu(P,Q)
\geq
\Bigg[
&\eta_{P,Q}\rho_-(s;\nu)\notag\\
&-
A_0L_\nu\sqrt{\frac\pi2}
\left\{
\sum_jp_{Pj}\sigma_{Pj}
+
\sum_\ell p_{Q\ell}\sigma_{Q\ell}
\right\}
\Bigg]_+.
\label{eq:heterogeneous-e4-transfer-bound}
\end{align}
Moreover,
$\eta_{P,Q}\rho_-(s;\nu)\leq\|V_{P,Q}c\|_2$, so
Proposition~\ref{prop:direct-heterogeneous-response-bound} is at least as
strong numerically.
\end{corollary}

\begin{proof}
Combine Proposition~\ref{prop:direct-heterogeneous-response-bound} with
Lemma~\ref{lem:heterogeneous-e4-transfer}.
\end{proof}

\begin{proof}[Proof of Theorem~\ref{thm:population-lower-bound}]
\emph{Part (i).}
Under Assumption~\ref{ass:prevalence-identifiability-regimes}(i),
$S_r=p\Phi(D_r^\star;\nu)$ by additivity and the zero response of
$\varnothing$. The certified scale and pair conditions in that assumption,
together with Assumption \ref{ass:diagram-distortion-floor}, give
$\|S_P-S_Q\|_2\ge p\rho_-(s;\nu)$. The perturbation term in
\eqref{eq:shared-perturbation-bound} also factors by $p$, which proves
\eqref{eq:common-prevalence-bound}.

\emph{Part (ii).}
Corollary~\ref{cor:heterogeneous-e4-transfer} gives exactly
\eqref{eq:heterogeneous-directional-bound}. Its proof invokes Assumption \ref{ass:diagram-distortion-floor} through
\eqref{eq:unweighted-template-response} and
Lemma~\ref{lem:heterogeneous-e4-transfer}.
\end{proof}

\begin{corollary}[Uniform bound on feature perturbation scales]
\label{cor:bounded-perturbation-scales}
If $\sigma_{rj}\leq\overline\sigma$, then the common-prevalence bound becomes
\begin{equation}
\label{eq:common-global-scale-bound}
\Delta_\nu(P,Q)
\ge
p\left[
\rho_-(s;\nu)
-
2A_0nL_\nu\overline\sigma\sqrt{\frac\pi2}
\right]_+,
\qquad s\in\mathcal I_\nu,
\end{equation}
and the Assumption \ref{ass:diagram-distortion-floor}-based heterogeneous-prevalence bound becomes
\begin{equation}
\label{eq:heterogeneous-global-directional-bound}
\Delta_\nu(P,Q)
\ge
\left[
\eta_{P,Q}\rho_-(s;\nu)
-
2A_0nL_\nu\overline\sigma\sqrt{\frac\pi2}
\right]_+.
\end{equation}
The stronger direct variant replaces
$\eta_{P,Q}\rho_-(s;\nu)$ by $\|V_{P,Q}c\|_2$.
\end{corollary}

\begin{proof}
Use $m_P,m_Q\le n$, $p_{rj}\le1$, and substitute in
Theorem~\ref{thm:population-lower-bound}; the direct variant follows from
Proposition~\ref{prop:direct-heterogeneous-response-bound}.
\end{proof}

For a linear subspace $\mathcal C\subseteq\mathbb R^M$, define
\begin{equation}
\label{eq:restricted-response-gain}
\sigma_{\min}(V_{P,Q};\mathcal C)
:=
\inf_{a\in\mathcal C\setminus\{0\}}
\frac{\|V_{P,Q}a\|_2}{\|a\|_2}.
\end{equation}

\begin{corollary}[Restricted response gain]
\label{cor:restricted-response-gain}
If $c\in\mathcal C$, then
\begin{align}
\Delta_\nu(P,Q)
\ge
\Bigg[
&\sigma_{\min}(V_{P,Q};\mathcal C)\|c\|_2\notag\\
&-
A_0L_\nu\sqrt{\frac\pi2}
\left\{
\sum_jp_{Pj}\sigma_{Pj}
+
\sum_\ell p_{Q\ell}\sigma_{Q\ell}
\right\}
\Bigg]_+.
\label{eq:restricted-response-bound}
\end{align}
\end{corollary}

\begin{proof}
The definition gives
$\|V_{P,Q}c\|_2\ge
\sigma_{\min}(V_{P,Q};\mathcal C)\|c\|_2$; substitute into
\eqref{eq:heterogeneous-direct-response-bound}.
\end{proof}

For a fixed nonzero $c$, taking $\mathcal C=\operatorname{span}\{c\}$ merely
rewrites the directional gain as
$\|V_{P,Q}c\|_2/\|c\|_2$. A restricted singular value becomes structurally
useful when the same subspace controls a collection of contrasts. Full column
rank corresponds to $\mathcal C=\mathbb R^M$ and is impossible when $M$
exceeds the finite descriptor dimension. The direct response proposition remains valid in that case. The Assumption \ref{ass:diagram-distortion-floor}-based theorem uses only
$\mathcal S_{P,Q}=\operatorname{span}\{c,q\}$ and therefore does not require full
column rank on $\mathbb R^M$.

\begin{proposition}[Realization-level certified separation]
\label{prop:realization-threshold}
Suppose, separately from the thinning model, that
\[
d_B(X_P,D_P^\star)\le\varepsilon_P,
\qquad
d_B(X_Q,D_Q^\star)\le\varepsilon_Q
\]
almost surely. Put $t=s-\varepsilon_P-\varepsilon_Q$. If
$t\in\mathcal I_\nu$ and the realized pair belongs to
$\mathcal C_{\nu,t}$ almost surely, then Assumption \ref{ass:diagram-distortion-floor} applies directly at threshold $t$.
If a particular construction has
$\mathcal I_\nu=(R_0(\nu),U_\nu]$, the scale condition is implied by
$s>R_0(\nu)+\varepsilon_P+\varepsilon_Q$ together with $t\leq U_\nu$; for a
common radius the lower-end condition is $s>R_0(\nu)+2\varepsilon$.
\end{proposition}

\begin{proof}
The reverse triangle inequality yields
$d_B(X_P,X_Q)\ge s-\varepsilon_P-\varepsilon_Q$.
The stated conclusion follows from Assumption \ref{ass:diagram-distortion-floor}.
\end{proof}

This result is natural in the complete-feature case or when whole-diagram
closeness is imposed directly. With thinning, bounded conditional movement of
appearing features alone does not imply whole-diagram closeness because a
template feature may be absent. Under sub-Gaussian perturbations there is no
deterministic radius; a high-probability version requires a chosen tail level.

\begin{remark}[Bounded-support alternative]
If
$G_{rj}\{x:d_0(x,z_{rj})\le\varepsilon_{rj}\}=1$, then
$e_{rj}\le\varepsilon_{rj}$, and every additive template result remains valid
after replacing $A_0\sigma_{rj}\sqrt{\pi/2}$ by $\varepsilon_{rj}$.
\end{remark}

\subsection{Whole-law population transfer for the Mitra--Virk map}
\label{app:mv-population-transfer}

Let
$\Psi_\nu^{\mathrm{MV}}:\mathcal D_n\to\mathcal H_\nu^{\mathrm{MV}}$
be a fixed admissible whole-diagram Mitra--Virk map. For the selected map,
write $\mathcal I_\nu^{\mathrm{MV}}$,
$\mathcal C_{\nu,t}^{\mathrm{MV}}$, and
$\rho_{-,\nu}^{\mathrm{MV}}(t)$ for its admissible scales, certified pair
class, and valid threshold floor. Its population estimand is
\begin{equation}
\label{eq:mv-population-signal}
\Delta_\nu^{\mathrm{MV}}(P,Q)
:=
\left\|
\mathbb E_P\{\Psi_\nu^{\mathrm{MV}}(X)\}
-
\mathbb E_Q\{\Psi_\nu^{\mathrm{MV}}(Y)\}
\right\|_2.
\end{equation}
It is not equal in general to an embedding of a padded mean counting measure.

Let the Jordan decomposition of the signed law difference be
$P-Q=\xi^+-\xi^-$. We use the probability-distance convention
\begin{equation}
\label{eq:mv-residual-mass}
b(P,Q)
:=
\xi^+(\mathcal D_n)
=
\xi^-(\mathcal D_n)
=
\sup_{A\in\mathcal B(\mathcal D_n)}|P(A)-Q(A)|.
\end{equation}
If $\|P-Q\|_{\mathrm{TV,norm}}:=|P-Q|(\mathcal D_n)$ denotes the doubled
signed-measure norm convention, then
$b(P,Q)=\tfrac12\|P-Q\|_{\mathrm{TV,norm}}$. We use $b(P,Q)$ throughout and
do not call both quantities the total-variation distance. When $b(P,Q)>0$,
define the normalized residual laws
\begin{equation}
\label{eq:mv-residual-laws}
P^\circ:=\frac{\xi^+}{b(P,Q)},
\qquad
Q^\circ:=\frac{\xi^-}{b(P,Q)}.
\end{equation}
They are probability measures. Moreover, $\xi^+\leq P$ and $\xi^-\leq Q$,
so integrability of $\Psi_\nu^{\mathrm{MV}}$ under $P$ and $Q$ implies
integrability under $P^\circ$ and $Q^\circ$ whenever $b(P,Q)>0$.

\begin{proposition}[Whole-law residual identity]
\label{prop:mv-residual-identity}
If the displayed expectations exist, then
\begin{align}
&\mathbb E_P\{\Psi_\nu^{\mathrm{MV}}(X)\}
-
\mathbb E_Q\{\Psi_\nu^{\mathrm{MV}}(Y)\}
\notag\\
&\hspace{2cm}=
 b(P,Q)
\left[
\mathbb E_{P^\circ}\{\Psi_\nu^{\mathrm{MV}}(X)\}
-
\mathbb E_{Q^\circ}\{\Psi_\nu^{\mathrm{MV}}(Y)\}
\right].
\label{eq:mv-residual-identity}
\end{align}
If $b(P,Q)=0$, then $P=Q$ and both sides are zero.
\end{proposition}

\begin{proof}
Integration against $P-Q=\xi^+-\xi^-$ gives the first difference. For
$b(P,Q)>0$, substitute $\xi^+=bP^\circ$ and $\xi^-=bQ^\circ$. The common
component of the two laws has canceled before this normalization.
\end{proof}

The diagram-level certificate alone does not lower-bound the difference of
population means. The residual mass $b(P,Q)$ can tend to zero, and even with
nonvanishing residual mass, vectors
$\Psi_\nu^{\mathrm{MV}}(D)$ may cancel after averaging. Both effects are
population-level obstructions absent from a statement about one diagram pair.

\begin{theorem}[Pair-specific Mitra--Virk template bound]
\label{thm:mv-template-bound}
Assume $b(P,Q)>0$ and put
$s=d_B(D_P^\star,D_Q^\star)$. Suppose
$s\in\mathcal I_\nu^{\mathrm{MV}}$,
$(D_P^\star,D_Q^\star)\in\mathcal C_{\nu,s}^{\mathrm{MV}}$, and
\[
\|\Psi_\nu^{\mathrm{MV}}(D)-
\Psi_\nu^{\mathrm{MV}}(E)\|_2
\le
L_\nu^{\mathrm{MV}}d_B(D,E).
\]
Define
\[
r_P:=\mathbb E_{P^\circ}d_B(X,D_P^\star),
\qquad
r_Q:=\mathbb E_{Q^\circ}d_B(Y,D_Q^\star).
\]
Then
\begin{equation}
\label{eq:mv-template-bound}
\Delta_\nu^{\mathrm{MV}}(P,Q)
\ge
b(P,Q)
\left[
\rho_{-,\nu}^{\mathrm{MV}}(s)
-
L_\nu^{\mathrm{MV}}(r_P+r_Q)
\right]_+.
\end{equation}
Here $L_\nu^{\mathrm{MV}}$ is the Lipschitz constant of the actual map used.
It equals $1$ for each normalized component of Theorem~4.3, for the theorem's
$2^{-1/2}$-normalized coarse--uniform direct sum, and for the
finite-dimensional bounded-domain map of Theorem~5.1. An unnormalized direct
sum must instead use the Hilbert-sum constant described in
Appendix~\ref{app:construction-verification}. For the normalized
coarse--uniform map, use the threshold floor
\eqref{eq:mv-threshold-floor} on all $s>0$; for the finite-dimensional
bounded-domain map, use \eqref{eq:mv-finite-affine-floor} on $(R_1,L]$.
\end{theorem}

\begin{proof}
By Proposition~\ref{prop:mv-residual-identity}, it suffices to lower-bound the
residual mean difference. Add and subtract the two template responses and use
the reverse triangle inequality:
\begin{align*}
&\left\|
\mathbb E_{P^\circ}\Psi_\nu^{\mathrm{MV}}(X)
-
\mathbb E_{Q^\circ}\Psi_\nu^{\mathrm{MV}}(Y)
\right\|_2\\
&\quad\ge
\|\Psi_\nu^{\mathrm{MV}}(D_P^\star)
-
\Psi_\nu^{\mathrm{MV}}(D_Q^\star)\|_2\\
&\qquad-
\mathbb E_{P^\circ}
\|\Psi_\nu^{\mathrm{MV}}(X)-
\Psi_\nu^{\mathrm{MV}}(D_P^\star)\|_2
-
\mathbb E_{Q^\circ}
\|\Psi_\nu^{\mathrm{MV}}(Y)-
\Psi_\nu^{\mathrm{MV}}(D_Q^\star)\|_2\\
&\quad\ge
\rho_{-,\nu}^{\mathrm{MV}}(s)
-
L_\nu^{\mathrm{MV}}(r_P+r_Q).
\end{align*}
Multiply by $b(P,Q)$ and take the positive part.
\end{proof}

\begin{corollary}[Bounded and tail-controlled residual laws]
\label{cor:mv-residual-concentration}
If $P^\circ$ and $Q^\circ$ are supported in bottleneck balls of radii
$\varepsilon_P$ and $\varepsilon_Q$ around their templates, then
\[
\Delta_\nu^{\mathrm{MV}}(P,Q)
\ge
b(P,Q)
\left[
\rho_{-,\nu}^{\mathrm{MV}}(s)
-
L_\nu^{\mathrm{MV}}(\varepsilon_P+\varepsilon_Q)
\right]_+.
\]
If instead their bottleneck tails satisfy
$\Pr\{d_B(X,D_P^\star)>t\}\le
A_0e^{-t^2/(2\sigma_P^2)}$ and analogously for $Q$, then the same statement
holds with
$\varepsilon_P+\varepsilon_Q$ replaced by
$A_0(\sigma_P+\sigma_Q)\sqrt{\pi/2}$.
\end{corollary}

\begin{proof}
The bounded-support claim gives $r_P\le\varepsilon_P$ and
$r_Q\le\varepsilon_Q$. The tail claim follows from the layer-cake argument in
Lemma~\ref{lem:tail-mean}.
\end{proof}

The general Hilbert-space mean, CLT, covariance, and two-sample procedures
apply directly to a fixed Mitra--Virk map whenever their stated moment
conditions hold; finite-coordinate or truncation results additionally require
Assumption~\ref{ass:finite-approximation}. The theorem above adds a sufficient
structured lower bound for a fixed pair
and does not assert that the embedded mean is characteristic for diagram laws. It is law-level and should
not be mixed with the additive template-mass decomposition for PLACE/PALACE.
Uniform Mitra--Virk power or minimax results would require class-wide lower
bounds on $b(P,Q)$ and the residual response margin, together with uniform
control of $r_P+r_Q$.


\section{Technical Proofs for Section~\ref{sec:probability}}
\label{app:probability-technical}

Throughout this appendix, $\mathcal H$ denotes the fixed separable Hilbert
space $\mathcal H_\nu$, $\Psi:=\Psi_\nu(X)$,
$\theta:=\mathbb E\Psi$, and $U:=\Psi-\theta$. The embedding configuration is
fixed independently of the sample.

\subsection{Mean and covariance facts}
\label{app:hilbert-mean-covariance-facts}

Because $\Psi$ is Borel measurable with values in a separable Hilbert space,
$\mathbb E\|\Psi\|<\infty$ implies Bochner integrability. The defining
property of the Bochner integral gives
$\langle\theta,h\rangle=\mathbb E\langle\Psi,h\rangle$ for every
$h\in\mathcal H$. If another element has the same inner products against all
$h$, the Riesz representation theorem implies equality.

Assume now that $\mathbb E\|\Psi\|^2<\infty$. The map
$u\mapsto u\otimes u$ is continuous from $\mathcal H$ to the trace-class
operators and $\|u\otimes u\|_{\mathrm{tr}}=\|u\|^2$. Thus
$U\otimes U$ is Bochner integrable in the trace-class norm and
$\Sigma=\mathbb E(U\otimes U)$ is trace class. Positivity and
self-adjointness follow from
$$
\langle\Sigma h,h\rangle
=
\mathbb E\langle U,h\rangle^2
\geq0.
$$
For an orthonormal basis $(e_k)$ of $\mathcal H$, Tonelli's theorem yields
$$
\operatorname{tr}(\Sigma)
=
\sum_k\mathbb E\langle U,e_k\rangle^2
=
\mathbb E\sum_k\langle U,e_k\rangle^2
=
\mathbb E\|U\|^2.
$$
A positive trace-class operator satisfies
$\|\Sigma\|_{\mathrm{op}}\leq\operatorname{tr}(\Sigma)$. Finally,
$\mathbb E\|U\|^2=\mathbb E\|\Psi\|^2-\|\theta\|^2\leq B_\nu^2$ under the
uniform envelope.

\subsection{Proof of Theorem~\ref{thm:hilbert-slln-clt}}
\label{app:proof-hilbert-slln-clt}

\begin{proof}
For part~(i), the centered variables $U_i:=\Psi_i-\theta$ are i.i.d.,
Bochner integrable, and mean zero in the separable Banach space $\mathcal H$.
The strong law for Banach-valued random variables gives
$m^{-1}\sum_{i=1}^mU_i\to0$ almost surely in $\mathcal H$; see, for example,
Corollary~7.10 of \cite{LedouxTalagrand2011}.

For part~(ii), $\mathbb E\|U\|^2<\infty$. A separable Hilbert space is a
Banach space of type~2, so the classical i.i.d. Hilbert-space central limit
theorem applies; see Theorem~10.5 of \cite{LedouxTalagrand2011}. Hence
$m^{-1/2}\sum_{i=1}^mU_i$ converges weakly to the centered Gaussian random
element with covariance $\Sigma$. The covariance facts established above show
that $\Sigma$ is trace class, which ensures that this Gaussian law is a Radon
probability measure on $\mathcal H$.
\end{proof}

\subsection{Norm-squared consequence of the central limit theorem}
\label{app:norm-squared-clt-consequence}

Let $(e_j)$ be an orthonormal eigenbasis for the closure of the range of
$\Sigma$, with corresponding nonzero eigenvalues $(\lambda_j)$. A centered
Gaussian element with covariance $\Sigma$ has the representation
$$
\mathcal G
\stackrel{d}{=}
\sum_{j\geq1}\lambda_j^{1/2}Z_je_j,
$$
where the $Z_j$ are independent standard normal variables. Since
$\sum_j\lambda_j=\operatorname{tr}(\Sigma)<\infty$, the series converges in
$L^2(\mathcal H)$ and almost surely. Therefore
$$
\|\mathcal G\|^2
\stackrel{d}{=}
\sum_{j\geq1}\lambda_jZ_j^2.
$$
The norm-squared limit in the main text follows from
Theorem~\ref{thm:hilbert-slln-clt}(ii) and the continuous mapping theorem.

\subsection{Proof of Corollary~\ref{cor:two-sample-hilbert-clt}}
\label{app:proof-two-sample-hilbert-clt}

\begin{proof}
Apply Theorem~\ref{thm:hilbert-slln-clt}(ii) separately to the two independent
samples. The joint convergence is to independent Gaussian elements
$\mathcal G_P\sim\mathcal N(0,\Sigma_P)$ and
$\mathcal G_Q\sim\mathcal N(0,\Sigma_Q)$. Since
$N/m\to\lambda^{-1}$ and $N/m'\to(1-\lambda)^{-1}$,
$$
\begin{aligned}
&\sqrt N
\left[
(\overline\Psi_{P,m}-\overline\Psi_{Q,m'})
-(\theta_P-\theta_Q)
\right]
\\
&\qquad=
\sqrt{\frac Nm}\,\sqrt m(\overline\Psi_{P,m}-\theta_P)
-
\sqrt{\frac N{m'}}\,\sqrt{m'}(\overline\Psi_{Q,m'}-\theta_Q).
\end{aligned}
$$
Slutsky's theorem and continuity of addition in $\mathcal H$ give a centered
Gaussian limit with covariance
$\lambda^{-1}\Sigma_P+(1-\lambda)^{-1}\Sigma_Q$.
\end{proof}

\subsection{Proof of Proposition~\ref{prop:covariance-consistency}}
\label{app:proof-covariance-consistency}

\begin{proof}
Let $U_i:=\Psi_i-\theta$, $\overline U_m:=m^{-1}\sum_iU_i$, and define the
oracle covariance estimator
$$
\widetilde\Sigma_m
:=
\frac1m\sum_{i=1}^mU_i\otimes U_i.
$$
Direct expansion gives
$$
\widehat\Sigma_m
=
\widetilde\Sigma_m-\overline U_m\otimes\overline U_m.
$$
The trace-class operators form a separable Banach space when $\mathcal H$ is
separable. Moreover,
$\mathbb E\|U\otimes U\|_{\mathrm{tr}}=\mathbb E\|U\|^2<\infty$.
The Banach-space strong law therefore yields
$\|\widetilde\Sigma_m-\Sigma\|_{\mathrm{tr}}\to0$ almost surely. By
Theorem~\ref{thm:hilbert-slln-clt}(i),
$\|\overline U_m\|\to0$ almost surely, and
$\|\overline U_m\otimes\overline U_m\|_{\mathrm{tr}}
=\|\overline U_m\|^2$. This proves trace-norm consistency. The inequalities
$\|A\|_{\mathrm{op}}\leq\|A\|_{\mathrm{HS}}\leq\|A\|_{\mathrm{tr}}$ give the
remaining almost-sure conclusions.

For the mean-square rate, assume $M_4:=\mathbb E\|U\|^4<\infty$. Since the
Hilbert--Schmidt operators form a Hilbert space and
$\|U\otimes U\|_{\mathrm{HS}}=\|U\|^2$,
$$
\mathbb E\|\widetilde\Sigma_m-\Sigma\|_{\mathrm{HS}}^2
=
\frac1m\mathbb E\|U\otimes U-\Sigma\|_{\mathrm{HS}}^2
\leq
\frac{M_4}{m}.
$$
We next bound the centering term. For $S_m:=\sum_{i=1}^mU_i$, independence
and $\mathbb EU_i=0$ imply
$$
\begin{aligned}
\mathbb E\|S_m\|^4
&=
\mathbb E\left(
\sum_i\|U_i\|^2
+2\sum_{i<j}\langle U_i,U_j\rangle
\right)^2
\\
&\leq
mM_4+3m(m-1)(\mathbb E\|U\|^2)^2
\leq4m^2M_4.
\end{aligned}
$$
Hence
$$
\mathbb E\|\overline U_m\otimes\overline U_m\|_{\mathrm{HS}}^2
=
\mathbb E\|\overline U_m\|^4
\leq
\frac{4M_4}{m^2}.
$$
Combining these bounds with $(a+b)^2\leq2a^2+2b^2$ gives
$$
\mathbb E\|\widehat\Sigma_m-\Sigma\|_{\mathrm{HS}}^2
\leq
\frac{2M_4}{m}+\frac{8M_4}{m^2}
\leq
\frac{10M_4}{m}.
$$
The operator-norm rate follows from
$\|A\|_{\mathrm{op}}\leq\|A\|_{\mathrm{HS}}$.
\end{proof}

\subsection{Proof of Theorem~\ref{thm:finite-coordinate-berry-esseen}}
\label{app:proof-finite-coordinate-berry-esseen}

\begin{proof}
The finite-dimensional space $\mathcal R_\Pi$ can be identified isometrically
with $\mathbb R^{r_\Pi}$. The random vector
$$
\xi:=\Sigma_\Pi^{\dagger/2}U_\Pi
$$
has mean zero and identity covariance on $\mathcal R_\Pi$. Directions in the
kernel of $\Sigma_\Pi$ have zero second moment and therefore vanish almost
surely, so this standardization loses no random component. Applying
Theorem~1.1 of \cite{Bentkus2003} to i.i.d. copies of $\xi$ gives
$$
\sup_{C\in\mathfrak C(\mathcal R_\Pi)}
\left|
\Pr\left\{m^{-1/2}\sum_{i=1}^m\xi_i\in C\right\}
-\Pr(Z_{r_\Pi}\in C)
\right|
\leq
\frac{C_{\mathrm{BE}}r_\Pi^{1/4}\mathbb E\|\xi\|^3}{\sqrt m},
$$
where $Z_{r_\Pi}$ is standard Gaussian on $\mathcal R_\Pi$. Restricting to
the covariance range and applying the inverse standardization gives the
stated inequality. Centered norm balls are convex, so they form a subclass of
the sets over which the supremum is taken.
\end{proof}

\begin{remark}[One-dimensional projections]
\label{rem:scalar-berry-esseen}
For a fixed $h\in\mathcal H$ with
$\sigma_h^2:=\mathbb E\langle U,h\rangle^2>0$, the same argument with
$r_\Pi=1$ gives the scalar Berry--Esseen bound
$$
\sup_{t\in\mathbb R}
\left|
\Pr\left\{
\frac{\sqrt m\langle\overline\Psi_m-\theta,h\rangle}{\sigma_h}
\leq t
\right\}
-\Phi_0(t)
\right|
\leq
\frac{C\,\mathbb E|\langle U,h\rangle|^3}{\sigma_h^3\sqrt m},
$$
where $\Phi_0$ is the standard normal distribution function.
\end{remark}

\subsection{Proof of Proposition~\ref{prop:finite-approximation-first-moments}}
\label{app:proof-finite-approximation-first-moments}

\begin{proof}
Assumption~\ref{ass:finite-approximation} gives
$\|\Psi_\nu(D)-\Pi_K\Psi_\nu(D)\|\leq\epsilon_K^\Psi$ for every diagram.
Taking expectations and using Jensen's inequality yields
$$
\|\theta-\theta_K\|
=
\|\mathbb E(\Psi_\nu(X)-\Pi_K\Psi_\nu(X))\|
\leq\epsilon_K^\Psi.
$$
Averaging the same deterministic bound over the observed diagrams gives
$\|\overline\Psi_m-\overline\Psi_{m,K}\|\leq\epsilon_K^\Psi$ almost surely.
The centered bound follows from the triangle inequality.

For two populations,
$$
\|\delta_\nu^\Psi(P,Q)-\Pi_K\delta_\nu^\Psi(P,Q)\|
\leq
\|\theta_P-\Pi_K\theta_P\|
+
\|\theta_Q-\Pi_K\theta_Q\|
\leq2\epsilon_K^\Psi.
$$
The reverse triangle inequality gives the stated bound for
$|\Delta_\nu^\Psi(P,Q)-\Delta_{\nu,K}^\Psi(P,Q)|$. Finally,
$$
\sqrt m\,
\|(\overline\Psi_m-\theta)
-(\overline\Psi_{m,K_m}-\theta_{K_m})\|
\leq
2\sqrt m\,\epsilon_{K_m}^\Psi
\longrightarrow0,
$$
which proves the central-limit-scale statement.
\end{proof}

\subsection{Assumption dependencies}
\label{app:probability-dependencies}

The mean and strong law require only
Assumption~\ref{ass:fixed-hilbert-embedding} and a finite first moment. The
Hilbert-space central limit theorem and trace-norm covariance consistency
require a finite second moment. The covariance mean-square rate requires a
finite fourth moment. The finite-coordinate Berry--Esseen bound requires a
fixed finite-rank projection and a finite standardized third moment on the
covariance range. Proposition~\ref{prop:finite-approximation-first-moments}
uses Assumption~\ref{ass:finite-approximation}. Additivity and Assumption \ref{ass:diagram-distortion-floor} are not used
in these probability arguments. For an additive map, additivity supplies the
mean-measure interpretation from
Proposition~\ref{prop:mean-embedding-transfer}; Assumption \ref{ass:diagram-distortion-floor} supplies geometric
interpretation only through the transfer results of
Section~\ref{sec:embedding}.


\section{Technical Proofs for Section~\ref{sec:testing}}
\label{app:testing-proofs}

This appendix proves the testing results that follow from the fixed-Hilbert
conditions and the structured alternative classes. The minimax lower bound is
expressed through an information modulus defined by those classes. A separate subsection states the construction-specific information-comparison
lemma and records the proof steps that remain to be verified for a selected
landmark construction.

\subsection{Null limit and consistency of the canonical statistic}

\begin{proof}[Proof of Theorem~\ref{thm:mean-distance-null}]
Corollary~\ref{cor:two-sample-hilbert-clt} gives
$$
\sqrt N\,\widehat\delta_{m,m'}
\ \rightsquigarrow\
\mathcal N_{\mathcal H_\nu}(0,\Sigma_{P,Q,\lambda})
$$
under $H_0^{\mathrm{mean}}$, where
$$
\Sigma_{P,Q,\lambda}
=
\lambda^{-1}\Sigma_{P,\nu}^{\Psi}
+
(1-\lambda)^{-1}\Sigma_{Q,\nu}^{\Psi}.
$$
Since
$$
\frac{N_{\mathrm{eff}}}{N}
=
\frac{mm'}{N^2}
\longrightarrow
\lambda(1-\lambda),
$$
Slutsky's theorem yields
$$
\sqrt{N_{\mathrm{eff}}}\,\widehat\delta_{m,m'}
\ \rightsquigarrow\
\mathcal G_{P,Q,\lambda},
$$
with covariance
$$
\lambda(1-\lambda)\Sigma_{P,Q,\lambda}
=
(1-\lambda)\Sigma_{P,\nu}^{\Psi}
+
\lambda\Sigma_{Q,\nu}^{\Psi}
=
\Gamma_{P,Q,\lambda}.
$$
The norm is continuous on $\mathcal H_\nu$, so the continuous mapping theorem
gives
$$
T_{m,m'}
=
\|\sqrt{N_{\mathrm{eff}}}\,\widehat\delta_{m,m'}\|^2
\ \rightsquigarrow\
\|\mathcal G_{P,Q,\lambda}\|^2.
$$
Because $\Gamma_{P,Q,\lambda}$ is positive and trace class, the spectral
representation of a Hilbert-space Gaussian gives
$$
\|\mathcal G_{P,Q,\lambda}\|^2
\stackrel{d}{=}
\sum_{j\geq1}\gamma_jZ_j^2.
$$
If $P=Q$, then $\Sigma_{P,\nu}^{\Psi}=\Sigma_{Q,\nu}^{\Psi}$ and the stated
simplification follows.
\end{proof}

\begin{proof}[Proof of Corollary~\ref{cor:mean-distance-consistency}]
The strong law in Theorem~\ref{thm:hilbert-slln-clt} gives
$$
\widehat\delta_{m,m'}
\longrightarrow
\delta_\nu^{\Psi}(P,Q)
$$
almost surely. Continuity of the squared norm gives
$$
\|\widehat\delta_{m,m'}\|^2
\longrightarrow
\{\Delta_\nu^{\Psi}(P,Q)\}^2.
$$
The right-hand side is positive and $N_{\mathrm{eff}}\to\infty$, proving the
claim.
\end{proof}

\subsection{Finite-sample power bound}

\begin{proof}[Proof of Proposition~\ref{prop:finite-sample-power}]
Write
$$
R_{m,m'}
:=
\widehat\delta_{m,m'}-\delta_\nu^{\Psi}(P,Q).
$$
Independence of the two samples and the covariance trace bound imply
$$
\mathbb E\|R_{m,m'}\|^2
=
\frac{\operatorname{tr}(\Sigma_{P,\nu}^{\Psi})}{m}
+
\frac{\operatorname{tr}(\Sigma_{Q,\nu}^{\Psi})}{m'}
\leq
v^2\left(\frac{1}{m}+\frac{1}{m'}\right)
=
\frac{v^2}{N_{\mathrm{eff}}}.
$$
Under $H_0^{\mathrm{mean}}$, $R_{m,m'}=\widehat\delta_{m,m'}$. Markov's
inequality applied to the squared norm therefore gives
$$
\Pr\left\{
\|\widehat\delta_{m,m'}\|
>
\frac{v}{\sqrt{\alpha N_{\mathrm{eff}}}}
\right\}
\leq\alpha.
$$
At an alternative, failure to reject implies by the reverse triangle inequality
that
$$
\|R_{m,m'}\|
\geq
\Delta_\nu^{\Psi}(P,Q)
-
\frac{v}{\sqrt{\alpha N_{\mathrm{eff}}}}.
$$
When the right-hand side is positive, another application of Markov's
inequality gives
$$
\Pr(\phi_{\alpha,v}=0)
\leq
\frac{v^2/N_{\mathrm{eff}}}
{\left\{
\Delta_\nu^{\Psi}(P,Q)
-v/(\sqrt{\alpha N_{\mathrm{eff}}})
\right\}^2}.
$$
The displayed sample-size condition makes the denominator at least
$v^2/(\beta N_{\mathrm{eff}})$, proving that the Type II error is at most
$\beta$.
\end{proof}

\subsection{Structured signal margins}

\begin{proof}[Proof of Corollary~\ref{cor:structured-uniform-power}]
Consider first the common-prevalence regime. For every
$(P,Q)\in\mathfrak A_{\mathrm{com}}(\tau)$,
Theorem~\ref{thm:population-lower-bound}(i) and the uniform class restrictions
give
$$
\Delta_\nu(P,Q)
\geq
p\left[
\rho_-(s;\nu)-\mathcal E_{\mathrm{com}}(P,Q)
\right]_+
\geq
g_{\mathrm{com}}(\tau).
$$
The same argument using Theorem~\ref{thm:population-lower-bound}(ii) gives
$$
\Delta_\nu(P,Q)
\geq
g_{\mathrm{het}}(\tau)
$$
for every $(P,Q)\in\mathfrak A_{\mathrm{het}}(\tau)$. Apply
Proposition~\ref{prop:finite-sample-power} with the corresponding uniform
signal lower bound.
\end{proof}

\begin{proof}[Proof of Proposition~\ref{prop:no-uniform-power-unrestricted}]
Under the null pair $(Q,Q)$, the joint law of the two samples is
$Q^{\otimes m}\otimes Q^{\otimes m}$. Under the alternative pair
$(Q,P_\varepsilon)$, it is
$Q^{\otimes m}\otimes P_\varepsilon^{\otimes m}$. The common first factor
cancels in total variation, so for any level-$\alpha$ test $\phi_m$,
$$
\Pr_{Q,P_\varepsilon}(\phi_m=1)
\leq
\alpha
+
\operatorname{TV}
\left(
P_\varepsilon^{\otimes m},Q^{\otimes m}
\right).
$$
The one-observation total variation distance is
$\operatorname{TV}(P_\varepsilon,Q)=\varepsilon$. Tensorization gives the
simple bound
$$
\operatorname{TV}
\left(
P_\varepsilon^{\otimes m},Q^{\otimes m}
\right)
\leq
m\varepsilon.
$$
For any target power $1-\beta>\alpha$, choose
$\varepsilon<(1-\beta-\alpha)/m$. The geometric mean-measure separation
remains $\tau$ by Example~\ref{ex:pop-obstruction}, but the power is below
$1-\beta$. Thus no fixed sample size controls power uniformly over the
unrestricted separated class.
\end{proof}

\subsection{Information-modulus comparison and minimax lower bound}
\label{app:minimax-information-comparison}

\subsubsection{Auxiliary information bounds}

\begin{lemma}[Information contraction under diagram formation]
\label{lem:diagram-information-contraction}
Let $\widetilde P$ and $\widetilde Q$ be laws on a labeled latent-feature
space, and let $\mathcal T$ map a labeled feature configuration to its
unlabeled persistence diagram. If
$P=\widetilde P\circ\mathcal T^{-1}$ and
$Q=\widetilde Q\circ\mathcal T^{-1}$, then
$$
\operatorname{KL}(P\|Q)
\leq
\operatorname{KL}(\widetilde P\|\widetilde Q).
$$
\end{lemma}

\begin{proof}
This is the data-processing inequality for Kullback--Leibler divergence.
\end{proof}

\subsubsection{Common-prevalence hard subfamily}
\label{app:common-prevalence-hard-subfamily}

We construct the null and alternative diagram laws used in the
common-prevalence information-modulus comparison. The construction uses the
certified singleton path from Assumption~\ref{ass:common-prevalence-minimax-richness}. Its realization for PLACE and PALACE is discussed in Remark~\ref{rem:minimax-scale-conditions}.

\begin{lemma}[Common-prevalence hard pair]
\label{lem:common-prevalence-hard-pair}
Suppose Definition~\ref{ass:uniform-structured-alternatives} and
Assumption~\ref{ass:common-prevalence-minimax-richness} hold. Then there is a
constant $C_{\mathrm{hard}}<\infty$, independent of
$\tau\in\mathcal T_{\mathrm{com}}$, such that, for every
$\tau\in\mathcal T_{\mathrm{com}}$, there exist $ (P_\tau,Q_\tau) \in \mathfrak A_{\mathrm{com}}(\tau)$ and $(P_{0,\tau},Q_{0,\tau}) \in \mathfrak H_0^{\mathrm{mean}}(v)$ satisfying
$$
\operatorname{KL}(P_\tau\|P_{0,\tau}) + \operatorname{KL}(Q_\tau\|Q_{0,\tau})
\le C_{\mathrm{hard}} g_{\mathrm{com}}(\tau)^2.
$$
\end{lemma}

\begin{proof}[Proof of Lemma~\ref{lem:common-prevalence-hard-pair}]
Because $\overline\tau<L/2$, choose $\varepsilon>0$ such that
$$
\overline\tau
<
\frac{L-2\varepsilon}{2}.
$$
For $\tau\in\mathcal T_{\mathrm{com}}$, define
$$
a_\tau=(\varepsilon,L-\varepsilon),
\qquad
b_\tau=(\varepsilon+\tau,L-\varepsilon),
\qquad
c_\tau=\frac{a_\tau+b_\tau}{2}.
$$
The direct matching cost between the two singleton diagrams is
$$
\|a_\tau-b_\tau\|_\infty=\tau.
$$
The cost of sending both points to the diagonal is
$$
\max\left\{
\delta_{\mathsf{Diag}_L}(a_\tau),
\delta_{\mathsf{Diag}_L}(b_\tau)
\right\}
=
\frac{L-2\varepsilon}{2}
>
\tau.
$$
The one-point bottleneck formula therefore gives
$$
d_B(\{a_\tau\},\{b_\tau\})=\tau.
$$
Moreover, $a_\tau$, $b_\tau$, and $c_\tau$ belong to the fixed compact set
$$
K
=
\left\{
(x,L-\varepsilon):
\varepsilon\leq x\leq\varepsilon+\overline\tau
\right\}
\subset
\{(b,d):0<b<d<L\}.
$$
The next paragraph verifies certification for the embedding under consideration.

\noindent \textbf{For PLACE}, every singleton pair is coherent at each active scale. Hence
$$
(\{a_\tau\},\{b_\tau\})
\in
\mathcal C_{\nu,\tau}^{\mathrm{PLACE}}
$$
for every $\tau$ in the applicable PLACE scale set.

\noindent \textbf{For PALACE}, non-interference is automatic for singleton diagrams. On a scale
set for which the configuration is $\tau$-admissible and
$a_\tau,b_\tau\in\mathcal R$, we have
$$
(\{a_\tau\},\{b_\tau\})
\in
\mathcal C_{\nu,\tau}^{\mathrm{PALACE}}.
$$

\smallskip

\noindent We then construct a common family of perturbation laws. Since
$K\subset\operatorname{int}(\mathbb H_L)$ is compact,
$$ r_K := \operatorname{dist}_2\left(K,\mathbb R^2\setminus\mathbb H_L\right)>0.
$$
Let $Z\sim N_2(0,I_2)$. For $z\in K$ and $\sigma>0$, define
$G_{z,\sigma}$ as the conditional law
$$
G_{z,\sigma} := \mathcal L\left(z+\sigma Z\mathrel{\Big|}z+\sigma Z\in\mathbb H_L\right).
$$
The corresponding densities are positive on the same set $\mathbb H_L$. Put
$$
q_\sigma:=\inf_{z\in K}\Pr\{z+\sigma Z\in\mathbb H_L\}.
$$
For every $z\in K$,
$$
\Pr\{z+\sigma Z\in\mathbb H_L\} \ge \Pr\{\|\sigma Z\|_2<r_K\}.
$$
Consequently,
$$
q_\sigma\longrightarrow 1, \qquad \text{as} ~~  \sigma\downarrow 0
$$
By parts~(ii) and~(iii) of
Assumption~\ref{ass:common-prevalence-minimax-richness}, we have
$e_0>0$, $A_0>1$, and $v>0$. We may therefore choose a fixed
$\sigma_0>0$ small enough that
$$
q_{\sigma_0}\ge A_0^{-1}, \qquad
2A_0L_\nu\sigma_0\sqrt{\frac{\pi}{2}} \le e_0, \qquad \text{and}  \qquad 2A_0L_\nu^2\sigma_0^2 \le v^2.
$$
This choice is independent of $\tau\in\mathcal T_{\mathrm{com}}$.

Now let $U\sim G_{z,\sigma_0}$. Since
$$
d_0(U,z) \le \|U-z\|_\infty \le \|U-z\|_2,
$$
the two-dimensional Gaussian tail formula gives, for every $t\ge0$,
$$
\begin{aligned}
G_{z,\sigma_0}\{x:d_0(x,z)>t\}
&\le
\frac{
\Pr\{\|\sigma_0Z\|_2>t\}
}{
\Pr\{z+\sigma_0Z\in\mathbb H_L\}
}
\\
&\le
A_0
\exp\left(
-\frac{t^2}{2\sigma_0^2}
\right).
\end{aligned}
$$
Thus the perturbation-tail condition holds uniformly over $z\in K$ with
scale parameter $\sigma_0$.

For $z\in K$, let $\mathsf R_z$ denote the law of the one-point diagram
${U}$ when $U\sim G_{z,\sigma_0}$. Define
$$
P_\tau:=\mathsf R_{a_\tau},
\qquad
Q_\tau:=\mathsf R_{b_\tau}.
$$
This is a template--thinning--perturbation model with one feature in each
group and common prevalence $p=1$. Hence $p\ge\pi_{\min}$. The template
separation is $\tau$, and part~(i) of
Assumption~\ref{ass:common-prevalence-minimax-richness} gives the required
certification. The perturbation penalty satisfies
$$
\begin{aligned}
\mathcal E_{\mathrm{com}}(P_\tau,Q_\tau)
&=
2A_0L_\nu\sigma_0\sqrt{\frac{\pi}{2}}
\\
&\le
e_0
\\
&\le
\overline e_{\mathrm{com}}(\tau).
\end{aligned}
$$

We next verify the covariance restriction. The Hilbert-space variance identity
gives
$$
\begin{aligned}
\operatorname{tr}
\left(
\Sigma_{\mathsf R_z,\nu}^{\Phi}
\right)
&=
\mathbb E
\left\|
\Phi(\{U\};\nu)
-
\mathbb E\Phi(\{U\};\nu)
\right\|_2^2
\\
&\le
\mathbb E
\left\|
\Phi(\{U\};\nu)-\Phi(\{z\};\nu)
\right\|_2^2.
\end{aligned}
$$
For a one-point diagram, the additive representation and the zero response of
the absence state give
$$
\Phi(\{U\};\nu)=\varphi_\nu^0(U),
\qquad
\Phi(\{z\};\nu)=\varphi_\nu^0(z).
$$
Point-level Lipschitz continuity therefore implies
$$
\operatorname{tr}
\left(
\Sigma_{\mathsf R_z,\nu}^{\Phi}
\right)
\le
L_\nu^2
\mathbb E\{d_0(U,z)^2\}.
$$
Integrating the tail bound yields
$$
\begin{aligned}
\mathbb E\{d_0(U,z)^2\}
&=
\int_0^\infty
2t\,
G_{z,\sigma_0}\{x:d_0(x,z)>t\}
\,dt
\\
&\le
2A_0\sigma_0^2.
\end{aligned}
$$
Hence
$$
\operatorname{tr}
\left(
\Sigma_{\mathsf R_z,\nu}^{\Phi}
\right)
\le
2A_0L_\nu^2\sigma_0^2
\le
v^2.
$$

Thus we have proved that
$$
(P_\tau,Q_\tau)
\in
\mathfrak A_{\mathrm{com}}(\tau).
$$
For the null pair, set
$$
P_{0,\tau}
=
Q_{0,\tau}
:=
\mathsf R_{c_\tau}.
$$
The two laws are identical. Their mean embeddings are therefore equal. The
covariance calculation above applies with $z=c_\tau$, so
$$
(P_{0,\tau},Q_{0,\tau})
\in
\mathfrak H_0^{\mathrm{mean}}(v).
$$
It remains to bound the information divergence. Put
$$
D_L := \operatorname{diam}_2(\mathbb H_L) \le \sqrt{2}L.
$$
The truncated Gaussian laws form an exponential family on the common carrier
$\mathbb H_L$. For $\eta\in\mathbb R^2$, define
$$
A(\eta) := \log \int_{\mathbb H_L} \exp\left\{\eta^\top x- \frac{\|x\|_2^2}{2\sigma_0^2}
\right\}\,dx.
$$
The natural parameter associated with the center $z$ is
$$
\eta_z:=\frac{z}{\sigma_0^2}.
$$
The log-partition function satisfies
$$
\nabla A(\eta) = \mathbb E_\eta(X), \quad \text{and } \quad \nabla^2A(\eta)
=\operatorname{Cov}_\eta(X).
$$
Since every distribution in the family is carried by $\mathbb H_L$,
$$
\|\nabla^2A(\eta)\|_{\mathrm{op}} \le D_L^2.
$$
For $z,c\in K$, the exponential-family KL identity and Taylor's theorem give
$$
\begin{aligned}
\operatorname{KL}
\left(
G_{z,\sigma_0}\|G_{c,\sigma_0}
\right)
&=
A(\eta_c)-A(\eta_z)
-
\nabla A(\eta_z)^\top(\eta_c-\eta_z)
\\
&\le
\frac{D_L^2}{2}
\|\eta_c-\eta_z\|_2^2
\\
&=
\frac{D_L^2}{2\sigma_0^4}
\|z-c\|_2^2.
\end{aligned}
$$
The map $x\mapsto{x}$ is a Borel isomorphism from $\mathbb H_L$ onto the
subspace of one-point diagrams. Passing from the point laws to the induced
diagram laws therefore preserves KL divergence. It follows that
$$
\begin{aligned}
&
\operatorname{KL}(P_\tau\|P_{0,\tau})
+
\operatorname{KL}(Q_\tau\|Q_{0,\tau})
\\
&\qquad\le
\frac{D_L^2}{2\sigma_0^4}
\left\{
\|a_\tau-c_\tau\|_2^2
+
\|b_\tau-c_\tau\|_2^2
\right\}.
\end{aligned}
$$
Since $c_\tau=(a_\tau+b_\tau)/2$,
$$
\begin{aligned}
\|a_\tau-c_\tau\|_2^2
+
\|b_\tau-c_\tau\|_2^2
&=
\frac{1}{2}
\|a_\tau-b_\tau\|_2^2
\\
&\le
\|a_\tau-b_\tau\|_\infty^2
\\
&=
\tau^2.
\end{aligned}
$$
Therefore,
$$
\operatorname{KL}(P_\tau\|P_{0,\tau})
+
\operatorname{KL}(Q_\tau\|Q_{0,\tau})
\le
\frac{D_L^2}{2\sigma_0^4}\tau^2.
$$
Part~(ii) of
Assumption~\ref{ass:common-prevalence-minimax-richness} gives
$$
\overline e_{\mathrm{com}}(\tau)
\le
(1-\gamma)\rho_-(\tau;\nu).
$$
The quantity inside the positive part defining
$g_{\mathrm{com}}(\tau)$ is therefore positive, and
$$
\begin{aligned}
g_{\mathrm{com}}(\tau)
&=
\pi_{\min}
\left\{
\rho_-(\tau;\nu)
-
\overline e_{\mathrm{com}}(\tau)
\right\}
\\
&\ge
\pi_{\min}\gamma\rho_-(\tau;\nu)
\\
&\ge
\pi_{\min}\gamma c_\rho\tau.
\end{aligned}
$$
Consequently,
$$
\tau^2
\le
\frac{
g_{\mathrm{com}}(\tau)^2
}{
\pi_{\min}^2\gamma^2c_\rho^2
}.
$$
Combining the preceding bounds gives
$$
\operatorname{KL}(P_\tau\|P_{0,\tau})
+
\operatorname{KL}(Q_\tau\|Q_{0,\tau})
\le
C_{\mathrm{KL}}
g_{\mathrm{com}}(\tau)^2,
$$
where
$$
C_{\mathrm{KL}}
:=
\frac{D_L^2}
{2\sigma_0^4
\pi_{\min}^2
\gamma^2
c_\rho^2}.
$$
The constant is finite and does not depend on
$\tau\in\mathcal T_{\mathrm{com}}$. Taking the infimum in the definition of
$\mathcal K_{\mathrm{com}}(\tau)$ proves the claim.
\end{proof}

\subsubsection{Feature-specific-prevalence hard subfamily}
\label{app:feature-specific-prevalence-hard-subfamily}

We construct the hard family by perturbing a common finite template in its
birth, death, and prevalence parameters. The perturbations are defined in
diagram space rather than in the embedding space.

\begin{lemma}[Feature-specific-prevalence hard pair]
\label{lem:heterogeneous-prevalence-hard-pair}
Suppose Assumptions~\ref{ass:additive-interface},
\ref{ass:uniformly-bounded-embedding}, and
\ref{ass:feature-specific-prevalence-minimax-richness} hold. Suppose also that
the covariance envelope satisfies $v\geq B_\nu$. Then there is a constant
$C_{\mathrm{hard,het}}<\infty$, independent of
$\tau\in\mathcal T_{\mathrm{het}}$, such that, for every
$\tau\in\mathcal T_{\mathrm{het}}$, there exist $(P_\tau,Q_\tau)\in\mathfrak A_{\mathrm{het}}(\tau)$ and $(P_{0,\tau},Q_{0,\tau}) \in\mathfrak H_0^{\mathrm{mean}}(v)
$ satisfying
$$
\operatorname{KL}(P_\tau\|P_{0,\tau}) + \operatorname{KL}(Q_\tau\|Q_{0,\tau})
\leq C_{\mathrm{hard,het}}g_{\mathrm{het}}(\tau)^2.
$$
\end{lemma}

\begin{proof}
Fix $\tau\in\mathcal T_{\mathrm{het}}$. Let $M$, the alternative parameters
$$
\{(z_{Pj,\tau},p_{Pj,\tau})\}_{j=1}^M,
\qquad
\{(z_{Qj,\tau},p_{Qj,\tau})\}_{j=1}^M,
$$
and the common reference parameters
$$
\{(z_{j,\tau}^0,p_{j,\tau}^0)\}_{j=1}^M
$$
be supplied by
Assumption~\ref{ass:feature-specific-prevalence-minimax-richness}.
All centers belong to the fixed compact set
$$
K_{\mathrm{het}}
\subset
\{(b,d):0<b<d<L\},
$$
and all prevalence probabilities belong to $[p_*,1-p_*]$. We first select a common perturbation scale. Let $Z$ be a standard Gaussian
vector in $\mathbb R^2$. For $z\in K_{\mathrm{het}}$ and $\sigma>0$, define
$$
G_{z,\sigma} :=
\mathcal L
\left(
z+\sigma Z
\mathrel{\Big|}
z+\sigma Z\in\mathbb H_L
\right).
$$
Because $K_{\mathrm{het}}$ is a compact subset of the interior of
$\mathbb H_L$,
$$
\inf_{z\in K_{\mathrm{het}}}
\Pr\{z+\sigma Z\in\mathbb H_L\}
\longrightarrow1
$$
as $\sigma\downarrow0$. Since $A_0>1$ and
$e_{0,\mathrm{het}}>0$, we may choose a fixed
$\sigma_0>0$ such that
$$
\inf_{z\in K_{\mathrm{het}}}
\Pr\{z+\sigma_0Z\in\mathbb H_L\}
\geq A_0^{-1}
$$
and
$$
2A_0nL_\nu\sigma_0\sqrt{\frac{\pi}{2}}
\leq e_{0,\mathrm{het}}.
$$
If $U\sim G_{z,\sigma_0}$, then
$d_0(U,z)\leq|U-z|_2$. Therefore, for every $t\geq0$,
$$
\begin{aligned}
G_{z,\sigma_0}\{x:d_0(x,z)>t\}
&\leq
\frac{
\Pr\{\|\sigma_0Z\|_2>t\}
}{
\Pr\{z+\sigma_0Z\in\mathbb H_L\}
}\\
&\leq
A_0
\exp\left(
-\frac{t^2}{2\sigma_0^2}
\right).
\end{aligned}
$$
Thus the required perturbation-tail condition holds uniformly over all
centers in $K_{\mathrm{het}}$.

For each group $r\in{P,Q}$ and feature $j$, let
$I_{rj,\tau}$ be independent Bernoulli variables with
$$
\Pr(I_{rj,\tau}=1)=p_{rj,\tau}.
$$
Conditional on $I_{rj,\tau}=1$, let the feature location have law
$G_{z_{rj,\tau},\sigma_0}$. Take all thinning indicators and location
perturbations to be independent. Let $P_\tau$ and $Q_\tau$ denote the
resulting diagram laws. This independent construction is permitted by the
template-thinning-perturbation model, which does not require independence but
allows it.

The template diagrams are
$$
D_{P,\tau}^\star
=
\{z_{P1,\tau},\ldots,z_{PM,\tau}\},
\qquad
D_{Q,\tau}^\star
=
\{z_{Q1,\tau},\ldots,z_{QM,\tau}\}.
$$
Assumption~\ref{ass:feature-specific-prevalence-minimax-richness} gives
$$
d_B(D_{P,\tau}^\star,D_{Q,\tau}^\star)=\tau
$$
and
$$
(D_{P,\tau}^\star,D_{Q,\tau}^\star)
\in\mathcal C_{\nu,\tau}.
$$
Write
$$
V_\tau:=V_{P_\tau,Q_\tau},
\qquad
\mathcal S_\tau
:=
\operatorname{span}\{c_\tau,q_\tau\}.
$$
The same assumption gives
$$
\operatorname{cond}(V_\tau;\mathcal S_\tau)
\leq\kappa_*
$$
and
$$
\|c_\tau\|_2\geq c_*\|q_\tau\|_2.
$$
Since $d_B(D_{P,\tau}^\star,D_{Q,\tau}^\star)>0$, we have
$q_\tau\neq0$. Hence $c_\tau\neq0$, and
$$
\begin{aligned}
\eta_{P_\tau,Q_\tau}
&=
\frac{\|c_\tau\|_2}
{
\operatorname{cond}(V_\tau;\mathcal S_\tau)
\|q_\tau\|_2
}\\
&\geq
\frac{c_*}{\kappa_*}
\geq
\eta_{\min}.
\end{aligned}
$$
The perturbation penalty satisfies
$$
\begin{aligned}
\mathcal E_{\mathrm{het}}(P_\tau,Q_\tau)
&=
A_0L_\nu\sqrt{\frac{\pi}{2}}
\left\{
\sum_{j=1}^M p_{Pj,\tau}\sigma_0
+
\sum_{j=1}^M p_{Qj,\tau}\sigma_0
\right\}\\
&\leq
2A_0nL_\nu\sigma_0\sqrt{\frac{\pi}{2}}\\
&\leq
e_{0,\mathrm{het}}\\
&\leq
\overline e_{\mathrm{het}}(\tau).
\end{aligned}
$$
Every law constructed above is supported on $\mathcal D_n$. Uniform
boundedness of the embedding therefore gives
$$
\operatorname{tr}(\Sigma_{P_\tau,\nu}^{\Phi})
\leq B_\nu^2
\leq v^2
$$
and
$$
\operatorname{tr}(\Sigma_{Q_\tau,\nu}^{\Phi})
\leq B_\nu^2
\leq v^2.
$$
We have proved that
$$
(P_\tau,Q_\tau)
\in
\mathfrak A_{\mathrm{het}}(\tau).
$$
We next construct the null pair. For each $j$, let
$I_{j,\tau}^0$ be Bernoulli with success probability $p_{j,\tau}^0$.
Conditional on $I_{j,\tau}^0=1$, let the corresponding feature location have
law $G_{z_{j,\tau}^0,\sigma_0}$. Take these variables to be independent, and
let $R_\tau^0$ be the resulting diagram law. Define
$$
P_{0,\tau}=Q_{0,\tau}=R_\tau^0.
$$
The two null laws are identical, so
$$
\theta_{P_{0,\tau},\nu}^{\Phi}
=
\theta_{Q_{0,\tau},\nu}^{\Phi}.
$$
They are supported on $\mathcal D_n$, and hence
$$
\operatorname{tr}(\Sigma_{P_{0,\tau},\nu}^{\Phi})
=
\operatorname{tr}(\Sigma_{Q_{0,\tau},\nu}^{\Phi})
\leq
B_\nu^2
\leq v^2.
$$
It follows that
$$
(P_{0,\tau},Q_{0,\tau})
\in
\mathfrak H_0^{\mathrm{mean}}(v).
$$
It remains to bound the information divergence. Introduce an absence symbol
$\partial$ and define the one-slot latent law
$$
H_{p,z}
:=
(1-p)\delta_{\partial}
+
pG_{z,\sigma_0}
$$
on ${\partial}\cup\mathbb H_L$. Since
$p,p^0\in[p_*,1-p_*]$, the absence and presence components have positive
probability under both laws. The Gaussian components also have common support
$\mathbb H_L$. Direct decomposition over the absence and presence events gives
$$
\operatorname{KL}(H_{p,z}\|H_{p^0,z^0})
=
\operatorname{kl}(p\|p^0)
+
p\operatorname{KL}
\left(
G_{z,\sigma_0}\|G_{z^0,\sigma_0}
\right),
$$
where $\operatorname{kl}$ denotes Bernoulli KL divergence. Moreover,
$$
\operatorname{kl}(p\|p^0)
\leq
\frac{(p-p^0)^2}{p_*(1-p_*)}.
$$
We next control the location contribution. Let
$$
D_L:=\operatorname{diam}_2(\mathbb H_L)\leq\sqrt{2}L.
$$
The truncated Gaussian family is an exponential family on the common support
$\mathbb H_L$. Define
$$
A(\eta)
:=
\log
\int_{\mathbb H_L}
\exp\left\{
\eta^\top x-\frac{\|x\|_2^2}{2\sigma_0^2}
\right\}\,dx.
$$
The natural parameter corresponding to center $z$ is
$$
\eta_z=\frac{z}{\sigma_0^2},
$$
and
$$
\nabla^2A(\eta)=\operatorname{Cov}_\eta(X).
$$
Every distribution in this family is supported on a set of Euclidean diameter
$D_L$. Therefore,
$$
\|\nabla^2A(\eta)\|_{\mathrm{op}}
\leq D_L^2.
$$
Taylor's theorem gives
$$
\begin{aligned}
\operatorname{KL}
\left(
G_{z,\sigma_0}\|G_{z^0,\sigma_0}
\right)
&\leq
\frac{D_L^2}{2}
\|\eta_z-\eta_{z^0}\|_2^2\\
&=
\frac{D_L^2}{2\sigma_0^4}
\|z-z^0\|_2^2.
\end{aligned}
$$
Set
$$
C_{\mathrm{loc}}
:=
\max\left\{
\frac{1}{p_*(1-p_*)},
\frac{D_L^2}{2\sigma_0^4}
\right\}.
$$
Combining the Bernoulli and location bounds yields
$$
\operatorname{KL}(H_{p,z}\|H_{p^0,z^0})
\leq
C_{\mathrm{loc}}
\left\{
|p-p^0|^2+\|z-z^0\|_2^2
\right\}.
$$
Let $\widetilde P_\tau$, $\widetilde Q_\tau$, and
$\widetilde R_\tau^0$ denote the corresponding product laws of the labeled
latent features. Additivity of KL divergence for product laws gives
$$
\begin{aligned}
&
\operatorname{KL}(\widetilde P_\tau\|\widetilde R_\tau^0)
+
\operatorname{KL}(\widetilde Q_\tau\|\widetilde R_\tau^0)\\
&\quad\leq
C_{\mathrm{loc}}
\sum_{j=1}^M
\Bigl\{
|p_{Pj,\tau}-p_{j,\tau}^0|^2
+
|p_{Qj,\tau}-p_{j,\tau}^0|^2\\
&\hspace{46mm}
+
\|z_{Pj,\tau}-z_{j,\tau}^0\|_2^2
+
\|z_{Qj,\tau}-z_{j,\tau}^0\|_2^2
\Bigr\}\\
&\quad\leq
C_{\mathrm{loc}}C_0\tau^2.
\end{aligned}
$$
The observed diagram is a measurable function of the labeled latent features.
The data-processing inequality therefore gives
$$
\operatorname{KL}(P_\tau\|P_{0,\tau})
+
\operatorname{KL}(Q_\tau\|Q_{0,\tau})
\leq
C_{\mathrm{loc}}C_0\tau^2.
$$
Finally, the margin condition in
Assumption~\ref{ass:feature-specific-prevalence-minimax-richness} gives
$$
\begin{aligned}
g_{\mathrm{het}}(\tau)
&=
\eta_{\min}\rho_-(\tau;\nu)
-
\overline e_{\mathrm{het}}(\tau)\\
&\geq
\gamma_{\mathrm{het}}
\eta_{\min}\rho_-(\tau;\nu)\\
&\geq
\gamma_{\mathrm{het}}
\eta_{\min}
c_{\rho,\mathrm{het}}\tau.
\end{aligned}
$$
The positive part may be omitted here because the displayed lower bound is
strictly positive. Consequently,
$$
\tau^2
\leq
\frac{
g_{\mathrm{het}}(\tau)^2
}{
\gamma_{\mathrm{het}}^2
\eta_{\min}^2
c_{\rho,\mathrm{het}}^2
}.
$$
Combining the last two bounds proves
$$
\operatorname{KL}(P_\tau\|P_{0,\tau})
+
\operatorname{KL}(Q_\tau\|Q_{0,\tau})
\leq
C_{\mathrm{hard,het}}
g_{\mathrm{het}}(\tau)^2,
$$
where
$$
C_{\mathrm{hard,het}}
:=
\frac{
C_{\mathrm{loc}}C_0
}{
\gamma_{\mathrm{het}}^2
\eta_{\min}^2
c_{\rho,\mathrm{het}}^2
}.
$$
This constant does not depend on
$\tau\in\mathcal T_{\mathrm{het}}$.
\end{proof}

\subsubsection{Proof of Lemma~\ref{lem:information-modulus-comparison}}

\begin{proof}
Fix $\tau\in\mathcal T_{\mathrm{com}}$. By
Lemma~\ref{lem:common-prevalence-hard-pair}, there exist
$$
(P_\tau,Q_\tau) \in \mathfrak A_{\mathrm{com}}(\tau)
$$
and
$$
(P_{0,\tau},Q_{0,\tau}) \in \mathfrak H_0^{\mathrm{mean}}(v)
$$
such that
$$
\operatorname{KL}(P_\tau\|P_{0,\tau}) + \operatorname{KL}(Q_\tau\|Q_{0,\tau})
\le
C_{\mathrm{hard}} g_{\mathrm{com}}(\tau)^2.
$$
By the definition of the information modulus,
$$
\begin{aligned}
\mathcal K_{\mathrm{com}}(\tau)
&\le
\operatorname{KL}(P_\tau\|P_{0,\tau})
+
\operatorname{KL}(Q_\tau\|Q_{0,\tau})
\\
&\le
C_{\mathrm{hard}}
g_{\mathrm{com}}(\tau)^2.
\end{aligned}
$$
The result follows with $C_{\mathrm{KL}}=C_{\mathrm{hard}}$.
\end{proof}

\subsubsection{Proof of the minimax theorem}

\begin{proof}[Proof of Theorem~\ref{thm:structured-prevalence-minimax-rate}]
For the upper bound, take equal group sizes $m=m'$. Then
$N_{\mathrm{eff}}=m/2$. Corollary~\ref{cor:structured-uniform-power} shows
that the test $\phi_{\alpha,v}$ has uniform power at least $1-\beta$ when
$$
\frac{m}{2}
\geq
\frac{v^2}{g_r(\tau)^2}
\left(
\alpha^{-1/2}+\beta^{-1/2}
\right)^2.
$$
This proves the upper bound.

For the lower bound, first suppose $0<\mathcal K_r(\tau)<\infty$. For every
$\varepsilon>0$, the definition of the infimum provides
$(P_\varepsilon,Q_\varepsilon)\in\mathfrak A_r(\tau)$ and
$(P_{0,\varepsilon},Q_{0,\varepsilon})\in
\mathfrak H_0^{\mathrm{mean}}(v)$ such that
$$
\operatorname{KL}(P_\varepsilon\|P_{0,\varepsilon})
+
\operatorname{KL}(Q_\varepsilon\|Q_{0,\varepsilon})
\leq
\mathcal K_r(\tau)+\varepsilon.
$$
For equal group sizes $m$, compare the null and alternative joint laws
$$
\mathbb P_{0,\varepsilon}^{(m)}
:=
P_{0,\varepsilon}^{\otimes m}
\otimes
Q_{0,\varepsilon}^{\otimes m},
\qquad
\mathbb P_{1,\varepsilon}^{(m)}
:=
P_\varepsilon^{\otimes m}
\otimes
Q_\varepsilon^{\otimes m}.
$$
For any test $\phi_m$ with level at most $\alpha$ uniformly over the null
class,
$$
\mathbb P_{1,\varepsilon}^{(m)}(\phi_m=1)
\leq
\alpha
+
\operatorname{TV}
\left(
\mathbb P_{1,\varepsilon}^{(m)},
\mathbb P_{0,\varepsilon}^{(m)}
\right).
$$
Pinsker's inequality and tensorization give
$$
\operatorname{TV}
\left(
\mathbb P_{1,\varepsilon}^{(m)},
\mathbb P_{0,\varepsilon}^{(m)}
\right)
\leq
\sqrt{
\frac{m}{2}
\left
\{
\mathcal K_r(\tau)+\varepsilon
\right\}
}.
$$
Hence every uniformly level-$\alpha$ test has power below $1-\beta$ at some
alternative whenever
$$
m
<
\frac{2(1-\beta-\alpha)^2}{\mathcal K_r(\tau)}.
$$
Lemma~\ref{lem:information-modulus-comparison} gives
$\mathcal K_r(\tau)\leq C_{\mathrm{KL}}g_r(\tau)^2$, and therefore
$$
m
<
\frac{2(1-\beta-\alpha)^2}
{C_{\mathrm{KL}}g_r(\tau)^2}
$$
is sufficient for the same conclusion. This proves the lower bound. The cases
$\mathcal K_r(\tau)=0$ and $\mathcal K_r(\tau)=+\infty$ follow by the usual
limiting conventions; under the lemma and $g_r(\tau)>0$, the latter case does
not occur.
\end{proof}


\section{Technical Proofs for Section~\ref{sec:confidence-certificates}}
\label{app:confidence-certificates}

\subsection{Gaussian plug-in calibration}

\begin{lemma}[Continuity of centered Gaussian laws in trace norm]
\label{lem:gaussian-trace-continuity}
Let $A_m$ and $A$ be positive trace-class operators on a separable Hilbert
space and suppose $\|A_m-A\|_1\to0$. If
$G_m\sim\mathcal N(0,A_m)$ and $G\sim\mathcal N(0,A)$, then
$G_m\rightsquigarrow G$. Consequently,
$\|G_m\|^2\rightsquigarrow\|G\|^2$.
\end{lemma}

\begin{proof}
For positive trace-class operators, the Powers--St{\o}rmer inequality implies
Hilbert--Schmidt convergence of their square roots:
$$
\|A_m^{1/2}-A^{1/2}\|_{\mathrm{HS}}^2
\leq
\|A_m-A\|_1
\longrightarrow0.
$$
Let $W$ be an isonormal cylindrical Gaussian element on the Hilbert space.
The random elements $A_m^{1/2}W$ and $A^{1/2}W$ are well defined because the
square roots are Hilbert--Schmidt, and they have laws
$\mathcal N(0,A_m)$ and $\mathcal N(0,A)$, respectively. Moreover,
$$
\mathbb E
\|A_m^{1/2}W-A^{1/2}W\|^2
=
\|A_m^{1/2}-A^{1/2}\|_{\mathrm{HS}}^2
\longrightarrow0.
$$
Thus the coupled Gaussian elements converge in $L^2$, hence in probability
and distribution. The squared norm is continuous, so the final claim follows.
\end{proof}

\begin{proof}[Proof of Theorem~\ref{thm:mean-confidence-ball}]
The Hilbert-space central limit theorem gives
$$
\sqrt m(\overline\Psi_m-\theta)
\rightsquigarrow
\mathcal G\sim\mathcal N(0,\Sigma).
$$
Proposition~\ref{prop:covariance-consistency} gives
$\|\widehat\Sigma_m-\Sigma\|_1\to0$ almost surely. Conditional on the data,
$\mathcal G_m^*$ is centered Gaussian with covariance
$\widehat\Sigma_m$. Lemma~\ref{lem:gaussian-trace-continuity} therefore
implies that its conditional squared-norm distribution converges weakly,
almost surely, to the distribution of $\|\mathcal G\|^2$.

When $\Sigma\neq0$, the weighted chi-square law of
$\|\mathcal G\|^2$ is continuous and has a unique $(1-\alpha)$ quantile.
The corresponding conditional quantile therefore satisfies
$$
\widehat q_{1-\alpha,m}
\longrightarrow
q_{1-\alpha}(\Sigma)
$$
in probability. Slutsky's theorem then yields
$$
\Pr\left\{
 m\|\overline\Psi_m-\theta\|^2
 \leq
 \widehat q_{1-\alpha,m}
\right\}
\longrightarrow
1-\alpha.
$$
If $\Sigma=0$, then
$\mathbb E\|\Psi_\nu(X)-\theta\|^2=\operatorname{tr}(\Sigma)=0$, so
$\Psi_\nu(X)=\theta$ almost surely.
\end{proof}

\paragraph{Verification of the finite-sample benchmark.}
Independence and centering give
$$
\mathbb E
\|\overline\Psi_m-\theta\|^2
=
\frac{1}{m}\operatorname{tr}(\Sigma)
\leq
\frac{v^2}{m}.
$$
Markov's inequality applied to the nonnegative random variable
$\|\overline\Psi_m-\theta\|^2$ yields
$$
\Pr\left\{
\|\overline\Psi_m-\theta\|
>
\frac{v}{\sqrt{m\alpha}}
\right\}
\leq\alpha.
$$
The bounded-embedding statement follows from
$\operatorname{tr}(\Sigma)=\mathbb E\|\Psi_\nu(X)-\theta\|^2\leq B_\nu^2$,
as recorded in Section~\ref{sec:probability}.

\subsection{Mean-measure pullback and exclusion}

\begin{lemma}[Upper transport bound for arbitrary padded mean measures]
\label{lem:arbitrary-mean-measure-upper-transport}
Under Assumption~\ref{ass:additive-interface}, for every
$\eta,\zeta\in\mathfrak M_n$,
$$
\|T_\nu[\eta]-T_\nu[\zeta]\|
\leq
nL_\nu W_{\infty,0}(\eta,\zeta).
$$
\end{lemma}

\begin{proof}
Fix a coupling $\pi$ of $\eta$ and $\zeta$. By the marginal identities and
the additive representation,
$$
T_\nu[\eta]-T_\nu[\zeta]
=
\int
\{\varphi_\nu^0(x)-\varphi_\nu^0(y)\}
\,d\pi(x,y).
$$
The point-level Lipschitz condition gives
$$
\|T_\nu[\eta]-T_\nu[\zeta]\|
\leq
L_\nu\int d_0(x,y)\,d\pi(x,y).
$$
The coupling has total mass $n$, and its transport cost is bounded
$\pi$-almost everywhere by its essential supremum. Hence
$$
\|T_\nu[\eta]-T_\nu[\zeta]\|
\leq
nL_\nu
\operatorname*{ess\,sup}_{(x,y)\sim\pi}d_0(x,y).
$$
Taking the infimum over couplings proves the claim.
\end{proof}

\begin{proof}[Proof of Proposition~\ref{prop:mean-measure-exclusion-certificate}]
Work on the event
$T_\nu[\overline\mu_P]\in\mathcal C_m$. For any candidate
$\eta\in\mathfrak M_n$,
$$
\|T_\nu[\eta]-T_\nu[\overline\mu_P]\|
\geq
\operatorname{dist}\{T_\nu[\eta],\mathcal C_m\}.
$$
Lemma~\ref{lem:arbitrary-mean-measure-upper-transport} gives
$$
W_{\infty,0}(\eta,\overline\mu_P)
\geq
\frac{
\|T_\nu[\eta]-T_\nu[\overline\mu_P]\|
}{nL_\nu},
$$
and the first displayed conclusion follows. If $\mathcal C_m$ is a closed
ball with center $\overline\Psi_m$ and radius $r_m$, then
$$
\operatorname{dist}\{T_\nu[\eta],\mathcal C_m\}
=
\left[
\|T_\nu[\eta]-\overline\Psi_m\|-r_m
\right]_+.
$$
The coverage probability of the event completes the proof.
\end{proof}

\subsection{Two-sample confidence set and geometric consequences}

\begin{proof}[Proof of Corollary~\ref{cor:two-sample-confidence-ball}]
The proof is the two-sample analogue of
Theorem~\ref{thm:mean-confidence-ball}. By
Corollary~\ref{cor:two-sample-hilbert-clt} and the effective-sample-size
rescaling used in Theorem~\ref{thm:mean-distance-null},
$$
\sqrt{N_{\mathrm{eff}}}
\{\widehat\delta_{m,m'}-\delta\}
\rightsquigarrow
\mathcal N(0,\Gamma_{P,Q,\lambda}).
$$
Trace-norm consistency of the two group covariance estimators and convergence
of the sample proportions imply
$$
\|\widehat\Gamma_{m,m'}-\Gamma_{P,Q,\lambda}\|_1
\longrightarrow0
$$
in probability. If $\Gamma_{P,Q,\lambda}\neq0$, the distribution of the
squared norm of the limiting Gaussian element is continuous.
Lemma~\ref{lem:gaussian-trace-continuity} and quantile continuity therefore
give the asserted asymptotic coverage.

If $\Gamma_{P,Q,\lambda}=0$, positivity of the covariance operators and
$\lambda\in(0,1)$ imply
$$
\Sigma_{P,\nu}^{\Psi}
=
\Sigma_{Q,\nu}^{\Psi}
=
0.
$$
Thus both embedded populations are almost surely constant,
$\widehat\delta_{m,m'}=\delta$ almost surely, and
$\widehat\Gamma_{m,m'}=0$. The plug-in confidence set is then the singleton
$\{\delta\}$ and has exact coverage one.

On the coverage event,
$$
\|\widehat\delta_{m,m'}\|-\widehat r_{m,m'}^{\alpha}
\leq
\|\delta\|
\leq
\|\widehat\delta_{m,m'}\|+\widehat r_{m,m'}^{\alpha}
$$
by the reverse and ordinary triangle inequalities. Truncating the lower
endpoint at zero gives the stated interval. The confidence-ball coverage event
is contained in the event that this scalar interval covers $\|\delta\|$.
Therefore,
$$
\liminf_{m,m'\to\infty}
\Pr_{P,Q}\left\{
\underline\Delta_{m,m'}^{\alpha}
\leq
\Delta_\nu^{\Psi}(P,Q)
\leq
\overline\Delta_{m,m'}^{\alpha}
\right\}
\geq
1-\alpha.
$$
The inequality may be strict because the scalar interval can cover even when
the Hilbert-space confidence ball does not contain $\delta$.

For an additive embedding,
Proposition~\ref{prop:mean-embedding-transfer} gives
$$
W_{\infty,0}(\overline\mu_P,\overline\mu_Q)
\geq
\frac{\Delta_\nu(P,Q)}{nL_\nu}
\geq
\frac{\underline\Delta_{m,m'}^{\alpha}}{nL_\nu}
$$
on the same event. Finally, if $(P,Q)\in\mathfrak A_r(\tau)$ is fixed, then the
uniform structured lower bound gives
$\Delta_\nu(P,Q)\geq g_r(\tau)$. Hence the event
$\overline\Delta_{m,m'}^{\alpha}<g_r(\tau)$ is incompatible with membership
in that class whenever the confidence interval covers the true effect size.
Combining this implication with the pointwise coverage statement proves the
pointwise asymptotic exclusion bound. No uniform Gaussian approximation over
$\mathfrak A_r(\tau)$ is asserted.
\end{proof}

\subsection{Nearest-centroid certificate}

\begin{proof}[Proof of Proposition~\ref{prop:centroid-robustness-certificate}]
Let $c=\widehat c(D)$ and work on the simultaneous event
$\|\widehat\theta_j-\theta_j\|\leq r_j$ for every class $j$. For any
$c'\neq c$, the triangle inequality gives
$$
\|z-\theta_{c'}\|
\geq
\|z-\widehat\theta_{c'}\|-r_{c'},
$$
and
$$
\|z-\theta_c\|
\leq
\|z-\widehat\theta_c\|+r_c.
$$
Therefore,
$$
\|z-\theta_{c'}\|-\|z-\theta_c\|
\geq
M_{c,c'}(D)-r_c-r_{c'}.
$$
The first margin condition makes this quantity positive for every competitor,
so the population nearest-centroid label is $c$.

For the perturbation statement, suppose
$\Psi_\nu=\Phi(\cdot;\nu)$, so that $z=\Phi(D;\nu)$. Now let $E$ satisfy
$d_B(D,E)\leq\varepsilon$ and write $z_E:=\Phi(E;\nu)$.
Proposition~\ref{prop:interface-consequences} gives
$$
\|z_E-z\|
\leq
nL_\nu\varepsilon.
$$
For either centroid $a$,
$$
\big|
\|z_E-a\|-\|z-a\|
\big|
\leq
\|z_E-z\|.
$$
Consequently, the pairwise empirical margin can decrease by at most
$2nL_\nu\varepsilon$ when $D$ is replaced by $E$. Combining this fact with
the centroid-error inequality above proves the perturbation statement. Solving the strict margin inequality for $\varepsilon$ shows that the
population nearest-centroid label of every
$E$ with $d_B(D,E)<\varepsilon_{\mathrm{cert}}(D)$ is
$c=\widehat c(D)$.
\end{proof}


\section{Technical Proofs for Section~\ref{sec:finite-approximation}}
\label{app:finite-approximation}

\subsection{Signal retained by orthogonal truncation}
\label{app:orthogonal-truncation-signal}

\begin{proof}[Proof of Proposition~\ref{prop:orthogonal-truncation-signal}]
Write
$$
R_K:=I-\Pi_K.
$$
Since $\Pi_K$ is an orthogonal projection, every $u\in\mathcal H_\nu$
satisfies
$$
\|u\|^2
=
\|\Pi_Ku\|^2+\|R_Ku\|^2.
$$
Assumption~\ref{ass:finite-approximation} gives
$$
\|R_K\Psi_\nu(D)\|
\leq
\epsilon_K^\Psi
$$
for every $D\in\mathcal D_n$.

Let
$$
\theta_P:=\theta_{P,\nu}^{\Psi},
\qquad
\theta_Q:=\theta_{Q,\nu}^{\Psi},
\qquad
u:=\theta_P-\theta_Q.
$$
Because $R_K$ is bounded and linear,
$$
R_K\theta_P
=
\mathbb E_P\{R_K\Psi_\nu(X)\}.
$$
Jensen's inequality therefore gives
$$
\|R_K\theta_P\|
\leq
\mathbb E_P\|R_K\Psi_\nu(X)\|
\leq
\epsilon_K^\Psi.
$$
The same argument applies to $\theta_Q$, so
$$
\|R_Ku\|
\leq
\|R_K\theta_P\|+\|R_K\theta_Q\|
\leq
2\epsilon_K^\Psi.
$$
Orthogonality now yields
$$
\begin{aligned}
\{\Delta_\nu^\Psi(P,Q)\}^2
&=
\|u\|^2 \\
&=
\|\Pi_Ku\|^2+\|R_Ku\|^2 \\
&\leq
\{\Delta_{\nu,K}^\Psi(P,Q)\}^2
+
4(\epsilon_K^\Psi)^2.
\end{aligned}
$$
Since $\Delta_{\nu,K}^\Psi(P,Q)\geq0$, it follows that
$$
\Delta_{\nu,K}^{\Psi}(P,Q)
\geq
\left[
\{\Delta_\nu^{\Psi}(P,Q)\}^2
-
4(\epsilon_K^\Psi)^2
\right]_+^{1/2}.
$$

For the diagram-level statement, let
$$
v:=\Psi_\nu(D)-\Psi_\nu(E).
$$
The uniform approximation bound gives
$$
\begin{aligned}
\|R_Kv\|
&\leq
\|R_K\Psi_\nu(D)\|
+
\|R_K\Psi_\nu(E)\| \\
&\leq
2\epsilon_K^\Psi.
\end{aligned}
$$
If $(D,E)$ satisfies the construction-specific certification conditions and
$d_B(D,E)\geq t>R_0(\nu)$, then
Assumption~\ref{ass:diagram-distortion-floor} gives
$$
\|v\|
\geq
\rho_-(t;\nu).
$$
Using orthogonality once more,
$$
\begin{aligned}
\|\Psi_{\nu,K}(D)-\Psi_{\nu,K}(E)\|^2
&=
\|\Pi_Kv\|^2 \\
&=
\|v\|^2-\|R_Kv\|^2 \\
&\geq
\rho_-(t;\nu)^2-4(\epsilon_K^\Psi)^2.
\end{aligned}
$$
Taking the nonnegative part and then the square root proves
$$
\|\Psi_{\nu,K}(D)-\Psi_{\nu,K}(E)\|
\geq
\rho_{-,K}(t;\nu).
$$

Finally, specialize to the additive embedding
$\Psi_\nu=\Phi(\cdot;\nu)$. For every
$(P,Q)\in\mathfrak A_r(\tau)$,
Theorem~\ref{thm:population-lower-bound} gives
$$
\Delta_\nu(P,Q)\geq g_r(\tau).
$$
Applying the first part of the proposition yields
$$
\begin{aligned}
\Delta_{\nu,K}(P,Q)
&\geq
\left[
\Delta_\nu(P,Q)^2-4\epsilon_K^2
\right]_+^{1/2} \\
&\geq
\left[
g_r(\tau)^2-4\epsilon_K^2
\right]_+^{1/2} \\
&=
g_{r,K}(\tau).
\end{aligned}
$$
Taking the infimum over
$(P,Q)\in\mathfrak A_r(\tau)$ completes the proof.
\end{proof}

\subsection{Testing under finite approximation}
\label{app:finite-approximation-testing}

\begin{proof}[Proof of Theorem~\ref{thm:finite-approximation-three-way}]
Write
$$
\delta_K
:=
\Pi_K\delta_\nu^\Phi(P,Q),
\qquad
Z_K
:=
\widehat\delta_{m,m',K}-\delta_K.
$$
Independence of the two samples gives
$$
\mathbb E_{P,Q}\|Z_K\|^2
=
\frac{\operatorname{tr}(\Sigma_{P,K})}{m}
+
\frac{\operatorname{tr}(\Sigma_{Q,K})}{m'}.
$$
Since
$$
\frac{N_{\mathrm{eff}}}{m}
=
\frac{m'}{m+m'},
\qquad
\frac{N_{\mathrm{eff}}}{m'}
=
\frac{m}{m+m'},
$$
the covariance bounds imply
\begin{equation*}
\mathbb E_{P,Q}\|Z_K\|^2
\leq
\frac{v^2}{N_{\mathrm{eff}}}.
\tag{A}
\end{equation*}

Suppose first that
$(P,Q)\in\mathfrak H_{0,K}^{\mathrm{mean}}(v)$. Then
$\delta_K=0$, so Markov's inequality and (A) give
$$
\begin{aligned}
\Pr_{P,Q}\{\phi_{\alpha,v}^{(K)}=1\}
&=
\Pr_{P,Q}\left\{
\|Z_K\|
>
\frac{v}{\sqrt{\alpha N_{\mathrm{eff}}}}
\right\} \\
&\leq
\frac{
\mathbb E_{P,Q}\|Z_K\|^2
}{
v^2/(\alpha N_{\mathrm{eff}})
} \\
&\leq
\alpha.
\end{aligned}
$$
The bound holds for every pair in
$\mathfrak H_{0,K}^{\mathrm{mean}}(v)$, proving uniform level control.

Now suppose that $(P,Q)\in\mathfrak A_r(\tau)$. Orthogonal projection cannot
increase covariance trace. Hence the covariance restrictions in
Definition~\ref{ass:uniform-structured-alternatives} again imply (A).
Proposition~\ref{prop:orthogonal-truncation-signal} also gives
$$
\|\delta_K\|
=
\Delta_{\nu,K}(P,Q)
\geq
g_{r,K}(\tau).
$$
Put
$$
s_N:=\frac{v}{\sqrt{N_{\mathrm{eff}}}}.
$$
The sample-size condition is equivalent to
$$
s_N
\leq
\frac{g_{r,K}(\tau)}
{\alpha^{-1/2}+\beta^{-1/2}}.
$$
It follows that
\begin{equation*}
g_{r,K}(\tau)-\frac{s_N}{\sqrt{\alpha}}
\geq
\frac{s_N}{\sqrt{\beta}}
>0.
\tag{B}
\end{equation*}

On the event $\{\phi_{\alpha,v}^{(K)}=0\}$,
$$
\|\widehat\delta_{m,m',K}\|
\leq
\frac{s_N}{\sqrt{\alpha}}.
$$
The reverse triangle inequality and (B) therefore give
$$
\begin{aligned}
\|Z_K\|
&\geq
\|\delta_K\|-\|\widehat\delta_{m,m',K}\| \\
&\geq
g_{r,K}(\tau)-\frac{s_N}{\sqrt{\alpha}} \\
&\geq
\frac{s_N}{\sqrt{\beta}}.
\end{aligned}
$$
Using Markov's inequality and (A),
$$
\begin{aligned}
\Pr_{P,Q}\{\phi_{\alpha,v}^{(K)}=0\}
&\leq
\Pr_{P,Q}\left\{
\|Z_K\|
\geq
\frac{s_N}{\sqrt{\beta}}
\right\} \\
&\leq
\frac{
\mathbb E_{P,Q}\|Z_K\|^2
}{
s_N^2/\beta
} \\
&\leq
\beta.
\end{aligned}
$$
This bound holds uniformly over $\mathfrak A_r(\tau)$, proving the stated
power guarantee.

Finally, $g_{r,K}(\tau)>0$ implies
$$
g_{r,K}(\tau)^2
=
g_r(\tau)^2-4\epsilon_K^2
$$
and $2\epsilon_K<g_r(\tau)$. Substitution into the first sample-size
condition gives the equivalent formulation.
\end{proof}

\subsection{Confidence sets under finite approximation}
\label{app:truncated-confidence}

\begin{proof}[Proof of Proposition~\ref{prop:truncated-confidence-inflation}]
For the one-sample statement, work on the projected coverage event
$$
\|\overline\Psi_{m,K}-\theta_K\|
\leq
r_{m,K}.
$$
By Assumption~\ref{ass:finite-approximation},
$$
\|\theta-\theta_K\|
\leq
\epsilon_K^\Psi.
$$
The triangle inequality therefore gives
$$
\begin{aligned}
\|\overline\Psi_{m,K}-\theta\|
&\leq
\|\overline\Psi_{m,K}-\theta_K\|
+
\|\theta_K-\theta\| \\
&\leq
r_{m,K}+\epsilon_K^\Psi.
\end{aligned}
$$
Thus the inflated Hilbert-space ball contains $\theta$ whenever the projected
ball contains $\theta_K$. Its coverage probability is therefore at least
$1-\alpha$.

For the two-sample statement, write
$$
\delta_K
:=
\delta_{\nu,K}^\Psi(P,Q)
=
\Pi_K\delta_\nu^\Psi(P,Q).
$$
Assumption~\ref{ass:finite-approximation} gives
$$
\begin{aligned}
\|\delta_\nu^\Psi(P,Q)-\delta_K\|
&=
\left\|
(I-\Pi_K)
\left\{
\theta_{P,\nu}^\Psi-\theta_{Q,\nu}^\Psi
\right\}
\right\| \\
&\leq
\|(I-\Pi_K)\theta_{P,\nu}^\Psi\|
+
\|(I-\Pi_K)\theta_{Q,\nu}^\Psi\| \\
&\leq
2\epsilon_K^\Psi.
\end{aligned}
$$
On the projected coverage event,
$$
\|\widehat\delta_{m,m',K}-\delta_K\|
\leq
r_{m,m',K}.
$$
Hence
$$
\begin{aligned}
\|\widehat\delta_{m,m',K}-\delta_\nu^\Psi(P,Q)\|
&\leq
\|\widehat\delta_{m,m',K}-\delta_K\|
+
\|\delta_K-\delta_\nu^\Psi(P,Q)\| \\
&\leq
r_{m,m',K}+2\epsilon_K^\Psi.
\end{aligned}
$$
The inflated ball therefore covers the full contrast with at least the
coverage probability of the projected ball.

Finally, the ordinary and reverse triangle inequalities give
$$
\begin{aligned}
\|\delta_\nu^\Psi(P,Q)\|
&\leq
\|\widehat\delta_{m,m',K}\|
+
r_{m,m',K}
+
2\epsilon_K^\Psi, \\
\|\delta_\nu^\Psi(P,Q)\|
&\geq
\|\widehat\delta_{m,m',K}\|
-
r_{m,m',K}
-
2\epsilon_K^\Psi.
\end{aligned}
$$
Taking the nonnegative part of the second bound proves the stated confidence
interval.
\end{proof}

\paragraph{Centered fluctuations under growing truncation.}
Let $K=K_{m,m'}$ be a deterministic sequence. Put
$$
R_K:=I-\Pi_K
$$
and define the centered tail variables
$$
U_{P,K}
:=
R_K
\left\{
\Psi_\nu(X)-\theta_{P,\nu}^\Psi
\right\},
\qquad
U_{Q,K}
:=
R_K
\left\{
\Psi_\nu(Y)-\theta_{Q,\nu}^\Psi
\right\}.
$$
They have mean zero. Moreover,
$$
\begin{aligned}
\mathbb E_P\|U_{P,K}\|^2
&=
\mathbb E_P
\left\|
R_K\Psi_\nu(X)
-
\mathbb E_P\{R_K\Psi_\nu(X)\}
\right\|^2 \\
&\leq
\mathbb E_P\|R_K\Psi_\nu(X)\|^2 \\
&\leq
(\epsilon_K^\Psi)^2.
\end{aligned}
$$
The same bound holds under $Q$. Independence and centering therefore yield
$$
\begin{aligned}
&\mathbb E_{P,Q}
\left\|
\sqrt{N_{\mathrm{eff}}}
\left[
(\widehat\delta_{m,m'}-\delta_\nu^\Psi)
-
(\widehat\delta_{m,m',K}-\delta_{\nu,K}^\Psi)
\right]
\right\|^2 \\
&\qquad=
\frac{N_{\mathrm{eff}}}{m}
\mathbb E_P\|U_{P,K}\|^2
+
\frac{N_{\mathrm{eff}}}{m'}
\mathbb E_Q\|U_{Q,K}\|^2 \\
&\qquad\leq
\left\{
\frac{m'}{m+m'}+\frac{m}{m+m'}
\right\}
(\epsilon_K^\Psi)^2 \\
&\qquad=
(\epsilon_K^\Psi)^2.
\end{aligned}
$$
Consequently, $\epsilon_{K_{m,m'}}^\Psi\to0$ implies convergence of the
centered truncation error to zero in $L^2$.

For comparison, the deterministic centering error satisfies
$$
\sqrt{N_{\mathrm{eff}}}
\|\delta_{\nu,K}^\Psi-\delta_\nu^\Psi\|
\leq
2\sqrt{N_{\mathrm{eff}}}\epsilon_K^\Psi.
$$
Thus centering the projected estimator directly at the full target requires
the stronger condition
$$
\sqrt{N_{\mathrm{eff}}}\epsilon_K^\Psi\longrightarrow0.
$$


\section{Proofs for Section~\ref{sec:extensions}}
\label{app:extensions-proofs}

\subsection{Proof of Proposition~\ref{prop:fixed-k-regression}}

\begin{proof}
Write $\widetilde Z_i=\widetilde Z_K(X_i)$ and define
$$
\widehat Q_{K,m}
=
\frac1m\sum_{i=1}^m\widetilde Z_i\widetilde Z_i^\top,
\qquad
\widehat b_{K,m}
=
\frac1m\sum_{i=1}^m\widetilde Z_iY_i.
$$
Because $K$ is fixed and $Z_K(X)$ is bounded by the assumption of
Proposition~\ref{prop:fixed-k-regression}, the entries of
$\widetilde Z_K(X)\widetilde Z_K(X)^\top$ are integrable. Moreover,
$\mathbb E(Y^2)<\infty$ and boundedness of $\widetilde Z_K(X)$ imply
integrability of $\widetilde Z_K(X)Y$. The strong law therefore gives
$$
\widehat Q_{K,m}\longrightarrow Q_K,
\qquad
\widehat b_{K,m}
\longrightarrow
\mathbb E\{\widetilde Z_K(X)Y\}
$$
almost surely.

Since $Q_K$ is positive definite, $\widehat Q_{K,m}$ is invertible with
probability tending to one and eventually almost surely along the strong-law
event. On that eventual event, the measurable least-squares minimizer is
unique and equals $\widehat Q_{K,m}^{-1}\widehat b_{K,m}$. The population normal equations give
$$
\vartheta_K^*
=
Q_K^{-1}\mathbb E\{\widetilde Z_K(X)Y\}.
$$
Hence
$$
\widehat\vartheta_K
=
\widehat Q_{K,m}^{-1}\widehat b_{K,m}
\longrightarrow
\vartheta_K^*
$$
almost surely.

Let
$$
\varepsilon_{K,i}
=
Y_i-\widetilde Z_i^\top\vartheta_K^*.
$$
The population normal equations imply
$\mathbb E(\widetilde Z_K\varepsilon_K)=0$. Boundedness of
$\widetilde Z_K$ and $\mathbb E(Y^2)<\infty$ imply
$\mathbb E\|\widetilde Z_K\varepsilon_K\|^2<\infty$. Therefore, the
multivariate central limit theorem yields
$$
\frac1{\sqrt m}
\sum_{i=1}^m
\widetilde Z_i\varepsilon_{K,i}
\rightsquigarrow
\mathcal N_{d_K+1}(0,\Omega_K).
$$
The least-squares identity gives
$$
\sqrt m
\left(\widehat\vartheta_K-\vartheta_K^*\right)
=
\widehat Q_{K,m}^{-1}
\frac1{\sqrt m}
\sum_{i=1}^m
\widetilde Z_i\varepsilon_{K,i}.
$$
Slutsky's theorem proves the stated Gaussian limit. The
$O_{\mathbb P}(m^{-1/2})$ rate follows immediately.

It remains to verify the sandwich estimator. Let
$\widehat\varepsilon_{K,i}=Y_i-\widetilde Z_i^\top\widehat\vartheta_K$.
Because $\widetilde Z_i$ is uniformly bounded and
$\widehat\vartheta_K\to\vartheta_K^*$ almost surely,
$$
\frac1m\sum_{i=1}^m
\left(\widehat\varepsilon_{K,i}^2-\varepsilon_{K,i}^2\right)
\widetilde Z_i\widetilde Z_i^\top
\longrightarrow0
$$
almost surely, entrywise. Indeed, the absolute value of each entry is bounded
by a constant times
$$
\|\widehat\vartheta_K-\vartheta_K^*\|
\left(\frac1m\sum_{i=1}^m|\varepsilon_{K,i}|
+
\|\widehat\vartheta_K-\vartheta_K^*\|
\right),
$$
which converges to zero by the strong law and
$\mathbb E|\varepsilon_K|<\infty$. Also, the strong law gives
$$
\frac1m\sum_{i=1}^m
\varepsilon_{K,i}^2\widetilde Z_i\widetilde Z_i^\top
\longrightarrow\Omega_K
$$
almost surely. Hence $\widehat\Omega_{K,m}\to\Omega_K$ almost surely, and
continuity of matrix inversion on the positive-definite cone yields
$\widehat V_{K,m}\to Q_K^{-1}\Omega_KQ_K^{-1}$ almost surely, and therefore
in probability.
\end{proof}

\subsection{Proof of Proposition~\ref{prop:regression-score-stability}}

\begin{proof}
Because $\Pi_K$ is an orthogonal projection,
$$
\begin{aligned}
\left|s_{\beta,K}(D)-s_{\beta,K}(E)\right|
&=
\left|
\left\langle
\Pi_K\beta,
\Pi_K\{\Phi(D;\nu)-\Phi(E;\nu)\}
\right\rangle
\right|\\
&\leq
\|\Pi_K\beta\|
\|\Phi(D;\nu)-\Phi(E;\nu)\|.
\end{aligned}
$$
Proposition~\ref{prop:interface-consequences} then gives
$$
\left|s_{\beta,K}(D)-s_{\beta,K}(E)\right|
\leq
nL_\nu\|\Pi_K\beta\|d_B(D,E).
$$

For the truncation bound, orthogonality gives
$$
\begin{aligned}
s_\beta(D)-s_{\beta,K}(D)
&=
\langle\beta,\Phi(D;\nu)\rangle
-
\langle\Pi_K\beta,\Pi_K\Phi(D;\nu)\rangle\\
&=
\left\langle
(I-\Pi_K)\beta,
(I-\Pi_K)\Phi(D;\nu)
\right\rangle.
\end{aligned}
$$
Consequently,
$$
\left|s_\beta(D)-s_{\beta,K}(D)\right|
\leq
\|(I-\Pi_K)\beta\|
\|(I-\Pi_K)\Phi(D;\nu)\|
\leq
\|(I-\Pi_K)\beta\|\epsilon_K^\Phi
\leq
\|\beta\|\epsilon_K^\Phi.
$$
Taking the supremum over $D$ proves the second claim. The final geometric
inequality follows by rearranging the first bound whenever
$\|\Pi_K\beta\|>0$.
\end{proof}

\end{document}